\documentclass[11pt,a4paper]{article}
\usepackage{multicol}
\usepackage[margin=2cm]{geometry}
\usepackage{amssymb,latexsym, bm, dsfont,graphicx,subfigure}
\usepackage{amsthm,amsfonts, amsmath}
\usepackage{enumerate}
\usepackage{algorithm}
\usepackage{stmaryrd}
\usepackage{float}

\usepackage{stmaryrd}
\usepackage{mathtools}
\usepackage{hyperref}
\usepackage[noend]{algpseudocode}
\usepackage{subfigure,appendix}
\usepackage{thmtools}
\usepackage{thm-restate}
\usepackage{tabularx}
\usepackage{authblk}
\usepackage{xcolor}

\newtheorem{theorem}{Theorem}[section]
\newtheorem{corollary}[theorem]{Corollary}
\newtheorem{definition}[theorem]{Definition}
\newtheorem{proposition}[theorem]{Proposition}

\newtheorem{remark}[theorem]{Remark}
\newtheorem{lemma}[theorem]{Lemma}

\newtheorem{claim}[theorem]{Claim}

\usepackage{url}
\urldef{\mailsa}\path|dadush@cwi.nl,|
\urldef{\mailsb}\path|{l.vegh, g.zambelli}@lse.ac.uk,|

\newcommand{\bO}{\mathbf{0}}
\newcommand{\bb}{\mathbb}

\newcommand{\conv}{\mathrm{conv}}
\newcommand{\cone}{\mathrm{cone}}

\newcommand{\R}{\bb R}
\newcommand{\Rp}{\R_{\ge0}}
\newcommand{\Rpp}{\R_{>0}}

\newcommand{\Snpp}{\bb{S}^n_{>0}}
\newcommand{\Snp}{\bb{S}^n_{\ge0}}
\newcommand{\Q}{\bb Q}
\newcommand{\Z}{\bb Z}
\newcommand{\N}{\bb N}

\newcommand{\B}{\bb B}

\newcommand{\ceil}[1]{\lceil#1\rceil}

\newcommand{\rec}{\mathrm{rec}}
\newcommand{\sm}{\setminus}

\newcommand{\rk}{\mathrm{rk}}

\newcommand{\spn}{\mathrm{span}}
\newcommand{\supp}{\mathrm{supp}}

\newcommand{\km}{\ensuremath{m}}
\newcommand{\da}{\ensuremath{d}}
\newcommand{\ld}{\mu}
\newcommand{\thr}{\beta}
\newcommand{\Ta}{\mathcal{T}_a}
\newcommand{\To}{\mathcal{T}_o}
\newcommand{\ee}{\mathrm{e}}

\newcommand{\T}{\top}
\newcommand{\size}[1]{\left\langle #1\right\rangle}

\DeclareMathOperator{\width}{\mathrm{width}}

\DeclareMathOperator{\im}{im}

\def\tr{\mathrm{tr}}

\def\st{\,:\,}

\def\L{\mathcal L}
\newcommand{\eqdef}{\mathbin{\stackrel{\rm def}{=}}}
\def\centre{p}
\def\separ{a}

\floatstyle{plain}
\newfloat{myalgo}{tbhp}{mya}

\newcommand\ourpagewidth{.9\textwidth}

{\begin{myalgo}[#1]
\centering
\begin{minipage}{\ourpagewidth}
\begin{algorithm}[H]}%
{\end{algorithm}
\end{minipage}
\end{myalgo}}

\newcommand\blfootnote[1]{%
  \begingroup
  \renewcommand\thefootnote{}\footnote{#1}%
  \addtocounter{footnote}{-1}%
  \endgroup
}

\begin{document}

\title{On finding exact solutions of linear programs in the oracle model%
\blfootnote{This project has received funding from the European Research Council (ERC) under the European Union's Horizon 2020 research and innovation programme (grant agreements ScaleOpt--757481 and QIP--805241). A preliminary version of this paper has appeared in the proceedings of SODA 2022 \cite{SODA-version}.}}

\author[1]{Daniel Dadush}
\author[2,3]{L{\'{a}}szl{\'{o}} A. V{\'{e}}gh}
\author[3]{Giacomo Zambelli}
\affil[1]{Centrum Wiskunde \& Informatica, The Netherlands}
\affil[2]{Hertz Chair for Algorithms and Optimization, University of Bonn, Germany}
\affil[3]{Department of Mathematics, London School of Economics and Political Science, UK}

\date{\texttt{dadush@cwi.nl, lvegh@uni-bonn.de, g.zambelli@lse.ac.uk}}
\maketitle

\begin{abstract}
We consider linear programming in the oracle model: $\max\{c^\top x\st x\in P\}$, where  the polyhedron $P=\{x\in\R^n\st Ax\le b\}$ is given by a separation oracle. We present an algorithm that finds exact primal and dual solutions  using $O(n^2\log(n/\delta))$ oracle calls and $O(n^4\log(n/\delta)+n^5\log\log(1/\delta))$ arithmetic operations, where $\delta$ is a geometric condition number associated with the system $(A,b)$. These bounds do not depend on the cost vector $c$ and do not require  a priori knowledge of $\delta$. For rational data, $\log(1/\delta)$ is polynomially bounded in the encoding size of $(A,b)$, thus providing a polynomial-time algorithm.

The algorithm works in a black box manner,  requiring a subroutine for approximate primal and dual solutions; the above running times are achieved when using  the cutting plane method of Jiang, Lee, Song, and Wong (STOC 2020) for this subroutine. Whereas approximate solvers may return primal solutions only, we develop a general framework for extracting dual certificates  based on the work of Burrell and Todd (Math.~Oper.~Res.~1985).

Our algorithm strengthens
results by Gr\"otschel, Lov\'asz, and Schrijver (Prog.~Comb.~Opt.~1984), and by Frank and Tardos (Combinatorica 1987) that rely on bit-complexity arguments. Our 
 algorithm avoids rounding-based arguments such as
simultaneous Diophantine approximation and uses geometric arguments instead.
\end{abstract}

\section{Introduction}

We consider linear programming (LP) in the oracle model. Let $P=\{x\in\R^n\st Ax\le b\}$ be a polyhedron given by $A\in\R^{m\times n}$ and $b\in \R^m$; the $i$-th row of $A$ is denoted by $a_i^\top$. In the \emph{linear feasibility problem}, the goal is to either find  $x\in P$ or to conclude that $P=\emptyset$. In the \emph{linear optimization problem}, we are given an objective function $c\in\R^n$, and we want to find a solution $x\in P$ maximizing $c^\top x$, or the conclusion that the problem is infeasible or that it is unbounded. The focus of this paper is on  finding \emph{exact} rather than approximate solutions to these problems, along with exact dual certificates.

We say that the LP is \emph{explicitly given}, if the matrix $A$ and vector $b$ are given as part of the input. In the \emph{oracle model}, these are represented implicitly,
via a separation oracle. Our main example will be what we call a \emph{polyhedral separation oracle}: given $\bar x\in \R^n$, the oracle either confirms that $\bar x\in P$ or returns a violated constraint $b_i>a_i^\top \bar x$ from the system $Ax\le b$; see Section~\ref{sec:oracles} for a discussion of different oracle models. The number of constraints $m$ may not be polynomially bounded in $n$.

\paragraph{LP algorithms in the Turing model} For an explicit rational input $(A,b,c)$, the first  polynomial time LP algorithm in the Turing model was  given by
Khachiyan in 1979, using the ellipsoid method \cite{khachiyan}.
This method relies on a volumetric progress measure, assuming that the feasible region is either full-dimensional or empty. Thus, it is a particular challenge to tackle the case when the feasible region may be contained in a lower dimensional subspace.  Khachiyan used a perturbation $\tilde b$ of the right-hand-side, based on bit-complexity arguments, such that the polyhedron $\tilde P=\{x\st Ax\le \tilde b\}$  satisfies $\tilde P=\emptyset$ if and only if $P=\emptyset$, and $\tilde P$ is full-dimensional whenever nonempty.

Gr\"otschel, Lov\'asz, and Schrijver \cite{Grotschel1984,gls} used the ellipsoid method to tackle  LPs given implicitly by a strong separation oracle, and developed the theory of \emph{rational polyhedra}. They showed that for rational polyhedra with bounded `facet complexity', the ellipsoid method either finds a feasible solution in polynomial time, or  a lower dimensional subspace containing $P$ can be identified using \emph{simultaneous Diophantine approximation},  by an application of the  basis reduction algorithm by Lenstra, Lenstra, and Lov\'asz
\cite{lll}.

\paragraph{LP in the real model of computation}
In the context of LP it is natural to use a \emph{real model of computation}: we assume the input is given by real numbers, each requiring unit storage, and one can perform a set of arithmetic operations in unit time.
Arithmetics include basic operations ($+$, $-$, $\times$, $/$); certain models allow further operations such as $\sqrt{\;}$ and $\log$. In the context of LP, Traub and Wo{\'z}niakowski \cite{traub1982} advocated using such a model. A computational theory over reals was developed by Blum, Shub and Smale \cite{blum1989}, see also the book \cite{blum1998complexity}. The ultimate goal for an explicit LP in this model is to develop a \emph{strongly polynomial algorithm}: one where the number of arithmetic operations only depends on the number of variables $n$ and constraints $m$.\footnote{A strongly polynomial algorithm in the Turing model is further required to be in PSPACE, so that the bit-comp\-lexity of the numbers used in the algorithm remains bounded in terms of the input.} This was listed as the 9th question by Smale on his list of eighteen mathematical challenges for the 21st century \cite{smale1998mathematical}.
The existence of a strongly polynomial algorithm remains wide open; such algorithms are only known for special classes of LPs.

\paragraph{Explicitly given LPs}
For explicitly given LPs, interior point methods (IPMs) yield algorithms with excellent theoretical and practical performance; for recent developments, as well as for pointers to the literature, see \cite{Cohen2021,JSWZ,LeeS2019,vdb20,vdB2020-tall-dense}. IPMs naturally work in the real model;  most variants output  approximate solutions. In the Turing model, an approximate solution with sufficiently high accuracy can be converted to an exact optimal solution. Consider $\min c^\top x$, $Ax\le b$, $A\in\R^{m\times n}$, and let $L$ denote the total bit-complexity of $(A,b,c)$. The fastest deterministic algorithm for finding exact primal and dual optimal solutions is due to van den Brand \cite{vdb20}, with a running time of  $O(m^{\max\{\omega,2.5-\alpha/2, 2+1/6\}}L\log^{O(1)}(m))$, while Jiang et al.~\cite{JSWZ} give an $O(m^{\max\{\omega,2.5-\alpha/2, 2+1/18\}}L\log^{O(1)}(m))$ randomized algorithm and van den Brand et al. \cite{vdB2020-tall-dense} give an $O((mn+n^3)L\log^{O(1)}(m))$ randomized algorithm\footnote{Here $\omega$ and $\alpha$ denote the fast multiplication exponent and its dual. For the current best values of $\omega\approx 2.38$ and $\alpha\approx 0.31$, both algorithms in \cite{vdb20,JSWZ} run in time $O(m^{\omega}L\log^{O(1)}(m))$}.  

Tardos \cite{Tardos86} gave an algorithm in the Turing model with running time dependence only on the bit-complexity of $A$, but independent of $b$ and $c$. This was generalized to the real model of computation
by Vavasis and Ye \cite{Vavasis1996}, who gave a poly$(n,m,\log \bar\chi_A)$ algorithm for solving explicit LPs exactly, where $\bar\chi_A$ is a certain condition number associated with the constraint matrix. Allamigeon, Dadush, Loho, Natura, and V\'egh \cite{ADLNV2022} gave an interior point method whose running time matches the straight-line complexity of the central path up to a strongly polynomial factor. This generalizes \cite{Vavasis1996} and other previously known strongly polynomial classes, and played a key role in the recent work by Dadush, Koh, Natura, Olver, and V\'egh \cite{DKNOV_STOC} that gave a strongly polynomial algorithm for LPs with at most two nonzeros per column in the constraint matrix.

\paragraph{LP in the oracle model}
Several important problems in combinatorial optimization, including matching, network design, and submodular optimization problems, can be formulated by LPs with an exponential number of constraints. For such LPs the explicit description would be exponential; at the same time, one can find  violated constraints for infeasible points in polynomial time. This motivated the development of oracle algorithms by Gr\"otschel, Lov\'asz, and Schrijver \cite{gls}, based on the ellipsoid method.

Vaidya \cite{vaidya96} gave a more efficient cutting plane algorithm in the oracle setting; see \cite{Atkinson1995,Bertsimas2004,Jiang2020,Lee2015} for improvements and related algorithms. These algorithms  return approximate solutions.
Given a convex set $K\subseteq \R^n$  defined by a strong separation oracle and contained in a ball of radius $r$, the algorithm by Jiang, Lee, Song, and Wong \cite{Jiang2020} (henceforth referred to as the JLSW algorithm) either finds a feasible point in $K$, or concludes that $K$ does not contain a ball of radius $\varepsilon$. The algorithm makes  $O(n\log(nr/\varepsilon))$ oracle calls and uses $O(n^3\log(nr/\varepsilon))$  arithmetic operations. This oracle complexity is the same as for Vaidya's algorithm \cite{vaidya96} and is asymptotically optimal \cite{nemirovski1979}. Moreover, \cite{Jiang2020} presents evidence that the arithmetic complexity of $O(n^2)$ operations per oracle call may also be optimal.

While ellipsoid and other cutting plane methods deliver approximate solutions only, finding exact solutions is crucial for the applications in combinatorial optimization. Prior to our work,  all known results on finding exact LP solutions in the oracle model  were based on bit complexity assumptions. Strengthening the result of
Gr\"otschel, Lov\'asz, and Schrijver \cite{Grotschel1984,gls}, Frank and Tardos \cite{franktardos} showed that, assuming that the matrix $A$ and vector $b$ describing the system $Ax\le b$ are integral with the absolute values of the entries bounded by  $B$, then linear optimization in the oracle model
can be solved in time poly$(n,\log B)$. This is independent of the  encoding length of the cost function $c$.
The result is achieved by rounding $c$ using simultaneous Diophantine approximation. %

Strongly polynomial algorithms for solving LPs for many important problems such as submodular function minimization or minimum-cost matchings were given in \cite{gls} and \cite{franktardos}. Still,
 in contrast to explicitly given LPs,
 one cannot hope for strongly polynomial algorithms in the oracle model.
Indeed, according to the next claim,  there does not  exist a
deterministic  algorithm using $f(n)$ oracle calls for any function $f$;   this is proved in Appendix~\ref{sec:impossibility}.

\begin{restatable}{proposition}{nostrongly}\label{prop:no-strongly}
 There exists no function $f:\N\to \N$ and {deterministic} algorithm $\mathcal{A}$ that solves the optimization problem $\max c^\top x$, $x\in P$ using at most $f(n)$ oracle calls, where
$P\subseteq \R^n$ is a nonempty full-dimensional polyhedron and $c\in \R^n$, and $P$ is accessed via the following oracle: for each $\bar x\in \R^n$, either returns $\bar x\in P$, or a facet defining inequality violated by $\bar x$.
\end{restatable}

\medskip

\subsection{Our contributions}\label{sec:our-contr}
Assume the polyhedron $P=\{x\st Ax\le b\}\subseteq \R^n$ is given by a polyhedral separation oracle, and consider the problem of maximizing $c^\top x$ for an objective function $c\in \R^n$. Our main result is an  algorithm such that the number of arithmetic operations and oracle calls is polynomial in $n$ and the logarithm of the inverse of a certain positive condition number $\delta$ dependent on $(A,b)$, but independent of $c$.

To the best of our knowledge, this is the first oracle-polynomial algorithm to find exact solutions to LP in the real model of computation. What we mean by ``polynomial algorithm in the real model of computation'' is that the algorithm provides exact solutions without any rationality assumptions and moreover,  if each inequality of $Ax\le b$ is rational with encoding size at most $\varphi$, then the running time is  polynomially bounded in $n$ and $\varphi$. Indeed, this holds as we will observe that  $\log(1/\delta)$ is  polynomially bounded in $n$ and $\varphi$.

We extend and strengthen  the results in \cite{franktardos,Grotschel1984,gls}. Further, our results imply simpler and more efficient algorithms for many applications also in the bit-complexity model (see discussion in Section~\ref{sec:rational}).
We now introduce the main condition number of interest.
\begin{definition}\label{def:delta}
Let $V\subseteq \R^n$ be a set of vectors. We define
    $\delta_V$ to be the largest value such that, for any set of linearly independent vectors $\{v_i\st i\in I\}\subseteq V$ and $\lambda\in \R^I$,
    \[
    \left\|\sum_{i\in I}\lambda_i v_i\right\|\ge \delta_V \max_{i \in I} |\lambda_i| \cdot\|v_i\|\, .
    \]
\end{definition}
We note that $\delta_V>0$ if and only if the set $\{v/\|v\|\st v\in V\}$ is finite (Lemma~\ref{lem:delta-lem}). For a matrix $M\subseteq \R^{m\times n}$, we let $\delta_M$ denote the value corresponding to the rows of $M$.

This condition number was previously studied in the context of the shadow simplex algorithm by Brunsch and R\"oglin \cite{brunsch2013}, by Eisenbrand and Vempala \cite{eisenbrand2017}, and by Dadush and H\"ahnle \cite{dadush2016shadow}. They used the following equivalent characterization (see Lemma~\ref{lem:delta-lem}): $\delta_V$ is the largest number such that,  for any set of linearly independent vectors $\{v_i\st i\in I\}$,  the sine of the angle between the vector  $v_i$ and the subspace spanned by the vectors $\{v_j\st j\in I\setminus\{i\}\}$ is at least $\delta_V$. Further, $\delta_V$ bounds the minimum singular value of a matrix with columns $v_j$ (see \cite{brunsch2013}). %
For an integer matrix $M\in \Z^{m\times n}$, let $\Delta_M$ denote the largest absolute value of any non-singular subdeterminant; then, $1/(n\Delta_M^2)\le \delta_M$  \cite{brunsch2013}. In particular, in the rational model,  $\log(1/\delta_M)$ is polynomially bounded by $n$ and the sizes of numbers.
The quantity $\delta_M$ was also studied in the context of lattice basis reduction by Seysen \cite{seysen1993}. A related condition number appears in the characterization of  Hoffman constants \cite{Guler1995,Klatte1995,Pena2020}.

In what follows, for vectors $v \in \R^k, w \in \R^l$, we use the shorthand
$(v \mid w)  \coloneqq \begin{pmatrix} v \\ w \end{pmatrix} \in \R^{k+l}$ to denote
corresponding column vector and $(v^\top,w^\top)$ to denote the corresponding
row vector interpreted as an element in $(\R^{k+l})^*$. For
$P=\{x\in\R^n\st Ax\le b\}$, let us consider the matrix
\begin{equation}\label{eq:M-matrix}
M=\begin{pmatrix}\mathbf{0}& 1\\
-A & b
\end{pmatrix}\, ,
\end{equation}
corresponding to the conic embedding of $P$ defined by the cone $K=\{(x \mid t)\in\R^{n+1}\st M(x \mid t)\ge \bO\}$.
 We observe that
 \begin{equation}\label{eq:deltaA>deltaM}
 \delta_A\ge \delta_{M}\, ,
 \end{equation}
see Proposition~\ref{prop:A-M} in the Appendix for the simple proof.

\paragraph{Dual certificates} We recall standard concepts from duality theory. For a feasibility problem,  a \emph{Farkas certificate} that $P=\emptyset$ is a vector of  nonnegative coefficients $\lambda\in \Rp^J$  for a subset $J\subseteq [m]$ 
such that $\sum_{j\in J} \lambda_j a_j= 0$ and  $\sum_{j\in J} \lambda_j b_j<0$. For $c\in \R^n$, the dual polyhedron corresponding to $\max \{c^\top x\st x\in P\}$ is $D_c=\{y\in \Rp^m\st A^\top y=c\}$. The LP has a finite optimum if and only if both primal and dual programs are feasible. In this case, a \emph{dual certificate of optimality} for the solution $x^*\in P$ is defined by a subset $J\subseteq [m]$ and a vector of nonnegative coefficients $\lambda\in \Rp^J$ such that  $\sum_{j\in J} \lambda_j a_j= c$ and  $a_j^\top x^*=b_j$ for all $j\in J$. By duality theory, the set $J$ and the coefficients can always be chosen such that $|J|\le n$.

\medskip

Our main result shows that the one can find an exact solution in time $O(n)$ times the running time of the current best approximate algorithm \cite{Jiang2020}, replacing $r/\varepsilon$ by $1/\delta_{M}$, and an additional $O(n^5 \log\log(1/\delta_{M}))$ term.

\begin{theorem}\label{thm:LP-main}
Consider the LP problem $\max \{c^\top x\st Ax\le b\}$ for  $A\in \R^{m\times n}$, $b\in \R^m$, $c\in \R^n$,  given by a polyhedral separation oracle. Let the matrix $M$ be defined as in \eqref{eq:M-matrix}. {For parts {\em (ii)} and {\em (iii)}}, assume that  a polyhedral separation oracle for the recession cone ${\rm rec}(P) \coloneqq  \{x\in \R^{n}\st Ax\le \bO\}$ is also provided.
\begin{enumerate}[(i)]
\item A primal feasible solution or a Farkas certificate of infeasibility can be found  using $O(n^2\log(n/\delta_M))$ oracle queries and $O(n^4\log(n/\delta_M))$ arithmetic operations.
\item A dual feasible solution or a Farkas certificate of dual infeasibility can be found in $O(n^2\log(n/\delta_{A}))$ oracle queries and $O(n^4\log(n/\delta_{A})$ $+n^5\log\log(1/\delta_A))$ arithmetic operations.
\item  If both primal and dual systems are feasible, then primal and dual optimal solutions can be found in $O(n^2\log(n/\delta_M))$ oracle queries and $O(n^4\log(n/\delta_M)$ $+n^5\log\log(1/\delta_M))$ arithmetic operations.
\end{enumerate}
\end{theorem}

A few remarks about the result are in order.
\begin{itemize}
\item We use a \emph{black box} approach. The algorithms work in the conic setting via the conic embedding in Section~\ref{sec:conic-reduce}, and require a subroutine that produces \emph{`approximate dual certificates'}.
The running time stated in Theorem~\ref{thm:LP-main} refers to the JLSW algorithm \cite{Jiang2020}. In Section~\ref{sec:dual certificates}, we present a general scheme that  allows to extract dual certificates from a broad range of methods, including the ellipsoid method and geometric rescaling methods
~\cite{DVZ,rothvoss}.
\item Assuming $P$ is given by a  polyhedral separation oracle, our result strengthens that by Frank and Tardos \cite{franktardos}: for $A\in \Z^{m\times n}$, $b\in \Z^m$ with all entries having absolute value $B$, for $M$ as in \eqref{eq:M-matrix}, the maximum subdeterminant $\Delta_M$ is bounded as $\Delta_M\le B^n n^{n/2}$ by the Hadamard--inequality, and we have $\delta_M\ge 1/(n \Delta_M^2)\ge 1/(B^{2n} n^{n+1})$. Thus, our algorithm makes $O(n^3 \log(n B))$ oracle calls to solve a linear optimization program with an arbitrary objective function max $c^\top x$.\footnote{We note that we do not generalize  \cite{franktardos} for arbitrary oracle settings. The main result in \cite{franktardos} is a preprocessing step replacing $c$ by an equivalent $\tilde c$ of small encoding length; but they do not need stronger assumptions on the oracle.}
\item The algorithms in \cite{franktardos,Grotschel1984,gls,Jiang2021} rely on bit-complexity arguments. In contrast, our algorithms are in the real model of computation and are entirely geometric. For the rational settings,  our running time bounds depend on the condition number $\delta_M$, which can be significantly better than the  lower bounds implied by the bit-complexity. Our algorithm does not require an explicit knowledge of  $\delta_M$.
\item Cutting planes methods require the feasible region to be enclosed in a ball of known radius.
In the rational setting, the enclosing radius is estimated based on the encoding size of the coefficients. Our method does not require any such assumptions.
\end{itemize}

In Section~\ref{sec:circuits}, we show the connection between  $\delta_A$, and the condition number $\bar\chi_A$ used by Vavasis and Ye \cite{Vavasis1996} and the circuit imbalance measure $\kappa_A$ used in subsequent papers \cite{DadushHNV20,DadushNV20,ENV22}. In particular, we show that our results imply similar condition number dependence for LP feasibility as these works. 

Note that solving the dual feasibility problem only depends on $\delta_A$ of the constraint matrix $A$, but not on $b$ or $c$. One may ask whether if also the optimization problem could be solved in time dependent only on $A$.
This would be the analogue for  the oracle model of the Vavasis--Ye \cite{Vavasis1996} result for explicitly given LPs, and would be the best one can hope for in the oracle model in light of Proposition~\ref{prop:no-strongly}.  However, the following proposition illustrates that this is not possible;   the proof is given in Appendix~\ref{sec:impossibility}.
\begin{restatable}{proposition}{needb}\label{prop:need-b}
Let $\theta_A$ be a condition number associated with a matrix $A$ that remains unchanged by creating a duplicate copy of any row.
There exists no function $f:\N\times \R\to \N$ and algorithm $\mathcal{A}$ that solves  max $c^\top x$ s.t. $Ax\le b$ for $A\in \R^{m\times n}$, $b\in \R^m$, $c\in \R^n$ in $f(n,\theta_A)$ oracle queries,  assuming the system is given by a polyhedral separation oracle.
\end{restatable}

In light of the negative results,  Theorem~\ref{thm:LP-main} is conceptually the best possible one can hope for in the oracle model for linear optimization. The only scope for improvement may be to find algorithms that depend on better condition numbers of $(A,b)$, or use fewer oracle calls or arithmetic operations.

Even though Theorem~\ref{thm:LP-main} uses a more restrictive oracle model than the standard strong separation oracle assumption, we show that it can reproduce many important results for rational polyhedra in \cite{gls}. In particular, simultaneous Diophantine approximation can be avoided in most applications, and dual optimal solutions can be found much more efficiently.  These results are discussed in Section~\ref{sec:rational}.

\subsubsection{Reduction to the conic setting}\label{sec:conic-reduce}

The algorithms in Theorem~\ref{thm:LP-main} are derived from conic
optimization problems using the conic embedding.

We recall that a cone $K \subseteq \R^n$ is a convex set that is closed under
positive scalings, that is, $\lambda K = K$ for any $\lambda > 0$. We define a
conic separation oracle for $K$, to be an oracle that on input $\bar{x} \in
\R^n$, either outputs that $\bar{x} \in K$ or, if $\bar{x} \not\in K$,
outputs a \emph{non-zero} vector $v \in \R^n \setminus \{0\}$ such that $v^\top \bar{x} \leq 0$ and $v^\top x
\geq 0$, $\forall x \in K$. We do not assume that $K$ is closed, non-empty or
even polyhedral in this definition.\footnote{The only difference from the standard notion of a separation oracle is
that the right hand side is required to be 0. This follows automatically for any separator if $K \neq
\emptyset$ (since $0$ must be in the closure of $K$), so it is only a
non-trivial requirement when $K = \emptyset$ (which will be an important case for
the computation of Farkas certificates). We further note that the separator
produced by the oracle is not required to be strict,
i.e., we do not require  $v^\top \bar x\le 0$ or that $v^\top x > 0$ $\forall x \in K$. This is because such a
separator may not exist if $K$ is not closed.} 
Our algorithms make use of the following subroutine:

\begin{center} \fbox{\begin{minipage}{0.8\textwidth} \noindent
{\sf Oracle } {\sc Approx-Conic-Dual} \\
{\bf Input:}  A cone $K$ given by a conic separation oracle, and $\varepsilon>0$.\\
{\bf Output:} Either a point $x\in K$, or an {\em $\varepsilon$-approximate conic Farkas certificate}, which is defined by a set $\{m_j\st j\in J\}$ of vectors returned by the separation oracle, along with
multipliers $\lambda\in \Rpp^J$ such that
\begin{equation}\label{eq:approx-conic}\left\|\sum_{j\in J}\lambda_j m_j\right\|<\varepsilon, \qquad \sum_{j\in J}\lambda_j \|m_j\|_1\ge 1\, .\end{equation}
\end{minipage}}\end{center}
We let $\To(n,\varepsilon)$ denote the number of oracle calls and $\Ta(n,\varepsilon)$ the number of arithmetic operations of this subroutine. We assume these are of the form  $\To(n,\varepsilon)=g_o(n)\log^{\nu_o}(n/\varepsilon)$ and  $\Ta(n,\varepsilon)=g_a(n)\log^{\nu_a}(n/\varepsilon)$
for some constant values $\nu_o,\nu_a\ge 1$.
We assume that the number $|J|$ of oracle separators involved in the $\varepsilon$-approximate conic Farkas certificate \eqref{eq:approx-conic} is bounded by a function $\tau(n)$.

In Section~\ref{sec:dual certificates} we show the following using the JLSW algorithm \cite{Jiang2020}.
\begin{theorem}\label{thm:approx-conic-oracle}
There exists an oracle-polynomial algorithm for \textsc{Approx-Conic-Dual} with
$\To(n,\varepsilon)=O(n\log(n/\varepsilon))$, $\Ta(n,\varepsilon)=O(n^3\log(n/\varepsilon))$, and $\tau(n)=O(n)$. 
\end{theorem}

The above corresponds to the requisite approximate problem we need for solving certain conic problems exactly. For the purpose of exact solutions, we will
require further assumptions on the possible outputs of the oracle as in the
preceding section.

Consider cones of the form $K = \{x \in \R^n: M_T
x \geq 0, M_S x > 0\}$, defined by a matrix $M \in \R^{m \times n}$ and a (possibly trivial) partition $(S,T)$ of  $[m]$, where $M_S \in \R^{S \times n}$, $M_T \in \R^{T
\times n}$ denote the corresponding row-submatrices of $M$. Slightly abusing notation,
we let $m_i \in \R^n$, $i \in [m]$, satisfy $m_i^\top = M_{\{i\}}$, i.e., the
column vector whose transpose is the $i^{th}$ row of $M$.

\begin{center} \fbox{\begin{minipage}{0.8\textwidth} \noindent
A \emph{polyhedral conic separation oracle} for $K= \{x \in \R^n: M_T
x \geq 0, M_S x > 0\}$ is an oracle that, given
$\bar x\in \R^n$, either returns that $\bar x \in K$, or a vector $v \in
\R^n$, such that $\exists i \in [m]$ satisfying $v = m_i$ and for which
$v^\top \bar{x} < 0$ if $i \in T$ or $v^\top \bar{x} \leq 0$ if $i \in S$.
\end{minipage}}\end{center}

From the perspective of implementation, the separator does not specify the
index $i$, it needs only reveal whether $v^\top$ is a row indexed by $S$ or by
$T$. For our purposes in this paper, the list of strict inequalities induced by $M_S x >
0$, will in fact be known  in advance and will satisfy $|S| \le 2$.

\medskip

We now formulate our three main conic problems. In each case,
the goal is to provide algorithms that are oracle-polynomial in $n$ and
$\log(1/\delta_M)$.  The problems are defined
over a closed polyhedral cone $K
\subseteq \R^n$ of the form $K = \{x \in \R^n \st Mx
\geq 0\}$; that is, $S=\emptyset$. The particular oracle assumptions will be detailed in Theorem~\ref{thm:conic-main}.
 In the first problem, the first row $m_1^\top$ of $M$ plays
a special role and is given to us.

\begin{itemize}
    \item {\bf Strong conic feasibility problem:} either find an $x\in K$ such that $m_1^\top x>0$, or find a vector $y\in \Rp^m$ with $y_1=1$ such that $M^\top y=0$ certifying that no such $x$ exists.

\item {\bf Conic validity problem:}  Given $c\in \R^n$, either return
 an
    $\bar x\in K$ with $c^\top \bar x<0$, or find 
	 a certificate $y\in \Rp^k$ such that $M^\top y=c$ showing that $c^\top x\ge 0$ holds for every $x\in K$. 
\item {\bf Conic minimum-ratio problem:} The input is given by $c,d\in \R^n$, along with a certificate that $d^\top x\ge 0$ is valid for $K$. This certificate is expressed by a set $\{m_i\st i\in I\}$ for some $I\subseteq [\km]$ with $|I|\le n$ and  $y^{(d)}\in \Rp^m$ with $M^\top y^{(d)}=d$ and $\supp(y^{(d)})=I$. The goal is to find an optimal solution to the problem
\begin{equation}\label{eq:min-ratio}
\gamma^*\coloneqq \max\left\{\gamma\in\R\st (c-\gamma d)^\top x\ge 0\, ,\,\forall x\in K\right\}=\min\left\{\frac{c^\top x}{d^\top x}\st x\in K,\, d^\top x>0\right\}.\end{equation}
It is easy to check the equivalence of the two definitions.
 Depending on whether the problem is feasible and bounded, we ask for the following output.
	\begin{itemize}
		\item {\em Optimality:} if $\gamma^*$ is finite, return $\gamma^*$ and $x^*\in K$ with $(c-\gamma^* d)^\top x^*=0$, $d^\top x^*>0$, along with a dual  certificate $y\in \Rp^m$ such that $M^\top y=c-\gamma^* d$.
		\item {\em Infeasibility:} if $d^\top x=0$ for all $x\in K$, then return $y\in \Rp^m$ such that $M^\top y=-d$.
		\item {\em Unboundedness:} if \eqref{eq:min-ratio} is unbounded, return $\bar x\in K$ with $d^\top \bar x>0$, and $x\in K$ with $c^\top x<0$ and $d^\top x=0$.
	\end{itemize}
\end{itemize}
Standard LP duality implies that exactly one of the above cases is applicable.

\begin{remark} Throughout the paper, as above, we often refer to vectors $y\in \R^\km$, and require the computation of $M^\top y$, even though $M$ is only implicitly defined by a separation oracle. Whenever we use such notation, what we mean is that $y$ is represented by a set of rows $\{m_i\st i\in I\}$ for some $I\subseteq [m]$, $|I|\leq \mathrm{poly}(n)$, and by a vector $\tilde y\in\R^I$ such that $y_i=\tilde y_i$ for $i\in I$, $y_i=0$ for $i\notin I$.
\end{remark}

The strong feasibility problem and the validity problem can be reduced to each other:  the validity problem is the strong feasibility problem over the matrix $M'=\binom{c^\top}{M}$, whereas the strong feasibility problem is the validity problem for $c=-m_1$. We differentiate them since for the  validity problem  our goal is to find an algorithm whose running time only depends on $n$ and $\delta_M$, but not on $c$.
The strong feasibility algorithm is also significantly simpler than the validity algorithm and has a better running time estimate.

\begin{theorem}\label{thm:conic-main}
 Assume that there exists an oracle-polynomial-time algorithm for \textsc{Approx-Conic-Dual} that, for any $n \in \N$ and $\varepsilon > 0$, performs at most
$\To(n,\varepsilon)$ oracle calls, $\Ta(n,\varepsilon)$ arithmetic
operations, and that $\tau(n)$ is the size of the $\varepsilon$-approximate
conic Farkas certificate returned. Letting $K=\{x\in \R^n\st M x\ge 0\}$, $M \in \R^{m \times n}$, the following holds:
\begin{enumerate}[(i)]
\item\label{part:feasibility}  Given a polyhedral conic separation oracle for
\[K_1 =\{x \in \R^n: Mx \geq 0, m_1^\top x > 0\}\, ,\]
the strong conic feasibility problem can be solved using
$O(n)\cdot\To(n,\delta_M/O(n))$ oracle calls and
$O(n)\cdot\Ta(n,\delta_M/O(n))+O(n^3)\cdot\To(n,\delta_M/O(n))+O(n^3\tau(n)\log\log(1/\delta_M))$ arithmetic operations.
\item\label{part:validity} Given $c \in \R^n$ and a polyhedral conic separation oracle for
\[
K_c = \{x \in \R^n: Mx \geq 0, -c^\top x > 0\}\, ,
\]
 the conic validity problem can be solved using $O(n)\cdot\To(n,\delta_M/O(n))$ oracle calls and
$O(n)\cdot\Ta(n,\delta_M/O(n))+O(n^3)\cdot\To(n,\delta_M/O(n))+O((n^5+n^3\tau(n))\log\log(1/\delta_M))$ arithmetic operations.
\item\label{part:ratio}  Given $c,d \in \R^n$, $y^{(d)}\in\Rp^m$ such that $d = M^\top y^{(d)}$, $I=\supp(y^{(d)})$, $|I|\le n$, and
 polyhedral conic separation
oracles for the two cones
\[
K_{-d} = \{x \in \R^n: Mx \geq 0, d^\top x > 0\}\quad\mbox{ and }\quad K_I^{=} =
\{x \in \R^n: Mx \geq 0, M_I x =0\}\, ,\]
 the conic minimum-ratio problem
can be solved using $O(n)\cdot\To(n,\delta_M/O(n))$ oracle calls and
$O(n)\cdot\Ta(n,\delta_M/O(n))+O(n^3)\cdot\To(n,\delta_M/O(n))+O((n^5+n^3\tau(n))\log\log(1/\delta_M))$
arithmetic operations.
\end{enumerate}
\end{theorem}

\begin{remark}\em Note that if a polyhedral conic separation oracle for $K$ were available, one could implement the required oracles for $K_1$, $K_c$, $K_{-d}$ with $O(n)$ additional arithmetic operations, whereas the separation oracle for $K^=_I$ could be implemented with $O(n^2)$ additional arithmetic operations (assuming that a projection matrix to $\ker(M_I)$ is pre-computed). Unfortunately, for our application to Theorem~\ref{thm:LP-main}, we will not be able to assume that a polyhedral conic separation oracle for $K$ is available, whereas we will be able to implement the specific separation oracles for $K_1$, $K_c$, $K_{-d}$, and $K^=_I$. The issue is that the polyhedral separation oracles for $P=\{x\in\R^n\st Ax\le b\}$ and for $\rec(P)$ do not provide a conic polyhedral separation oracle for the standard conic embedding $K=\{(x\mid t)\in \R^{m+1}\st Mx\ge 0\}$, with $M$ as defined in \eqref{eq:M-matrix}; see Footnote~\ref{footnote:oracle} for an explanation. However,
we will show in the proof of  Theorem~\ref{thm:LP-main} in Section~\ref{sec:main-proof} that we can use the polyhedral separation oracles for $P$ and $\mathrm{rec}(P)$ to implement all required oracles in Theorem~\ref{thm:conic-main}. \end{remark}

\subsubsection{Application to rational polyhedra}
Let us now focus on \emph{rational polyhedra}, i.e. polyhedra where all facets can be described by rational inequalities of bit complexity at most $\varphi$, called the \emph{facet complexity}.
 Bounded facet complexity guarantees bounded \emph{vertex complexity}, i.e. all extreme point solutions are rational numbers of bounded encoding length.
The seminal work of Gr\"otschel, Lov\'asz, and Schrijver,
summarized in the book {\em Geometric Algorithms and Combinatorial
  Optimization} \cite{gls}, provided a polynomial-time algorithm for optimizing over rational polyhedra given by a strong separation oracles.

They use a black-box argument that requires a subroutine to find either a feasible point or a small-volume enclosing ellipsoid for the a convex set. Such a subroutine can be implemented using the ellipsoid method. Exploiting that a small volume ellipsoid must be sufficiently thin in a certain direction, they use \emph{simultaneous Diophantine approximation} to identify an affine subspace containing the feasible region.
This subspace can be computed using the lattice basis reduction algorithm by Lenstra, Lenstra, and Lov\'asz \cite{lll}. Recently, Jiang \cite{Jiang2021} found an algorithm with better oracle complexity by using a more direct reduction to \cite{lll}. This algorithm also crucially relies on lattice techniques and hence it does not extend to the real model of computation.

\medskip

For the sake of simplicity, let us discuss the problem of finding dual optimal solutions under the following assumption:
\begin{equation}\label{eq:bounded-facet-assumption}
\begin{aligned}
\parbox{0.8\textwidth}{\emph{The encoding sizes of the vectors returned by the strong separation oracle are polynomially bounded by the facet complexity $\varphi$.}}
\end{aligned}
\end{equation}
Under this assumption, one  can find an {\em optimal dual solutions with oracle inequalities} \cite[Lemma 6.5.15]{gls}.
This assumption is not without loss of generality; we discuss this and different concepts of dual solutions in Section~\ref{sec:rational}. In Section~\ref{sec:without-bounded}, we sketch how one can still recover the results of \cite{gls} from our approach also in the case that \eqref{eq:bounded-facet-assumption} does not hold.

For finding the optimal dual solutions, \cite{gls} needs several runs of the (primal) ellipsoid method, including the final one where they apply the ellipsoid method to explicitly solve the dual problem restricted to the dual variables corresponding to a large (albeit still polynomially bounded) set of inequalities. The running time  depends on a higher power of $\varphi$.

Under the same assumption \eqref{eq:bounded-facet-assumption}, Theorem~\ref{thm:conic-main} enables a much simpler and more efficient algorithm.
Even though Theorem~\ref{thm:conic-main} requires a polyhedral separation oracle, in Section~\ref{sec:continued} we show that one can convert a strong separation oracle to a polyhedral separation oracle by rounding the right hand sides of the inequalities using the continued fractions method.  Lemma~\ref{lem:delta-phi} shows that $\delta_M\ge 1/2^{O(n^3\varphi)}$ for the associated conic system.
Compared to the general framework in \cite{gls}, this method has the following advantages, under assumption \eqref{eq:bounded-facet-assumption}.

\begin{itemize}
	\item
	We can identify lower dimensional subspaces without simultaneous Diophantine approximation. The only `number theoretic' subroutine we use is the continued fractions method; otherwise, we rely on the purely geometric measure $\delta_M$.  Our algorithm can recurse by setting some inequalities returned by the oracle to equality.
	\item We recover dual certificates along with the primal solutions, without the need of solving a second, much larger linear program given by the restricted dual. The running time in  \cite{gls} depends on a higher degree polynomial of $\varphi$; our running time depends linearly on $\log(1/\delta_M)$.
    \item  The algorithms in \cite{gls} require accuracy depending on $\varphi$ from the approximate subroutines.
The running time of our algorithm depends on the condition number $\delta_M$, which can be drastically better than the lower bound implied by $\varphi$. 
\end{itemize}

\medskip

Dual optimal solutions for LPs in the oracle model can be important for applications in combinatorial optimization. For example, recent constant factor approximations to the asymmetric travelling salesman problem \cite{Svensson2020,traub2022} crucially use an optimal dual solution to the Held--Karp relaxation; prior to our work, this could only be obtained using the method in \cite{gls}. {For this relaxation, one can naturally obtain a polyhedral separation oracle that returns a violated degree constraint or blossom inequality. Therefore, we do not even need to round the right hand sides. Our algorithm proceeds directly by identifying tight inequalities in an optimal solution, and terminates with exact primal and dual optimal solutions in strongly polynomial time.}

\subsubsection{Implementing the approximate conic oracle}

Both \cite{gls} and Theorem~\ref{thm:conic-main} are black-box methods. However, \cite{gls} requires a seemingly weaker `primal-only' subroutine, whereas Theorem~\ref{thm:conic-main} requires an approximate dual certificate.
We next explain that this difference is illusory: a $\varepsilon$-approximate conic Farkas certificates can be naturally extracted from the ellipsoid method as well as other convex feasibility algorithms.

The algorithm of Theorem~\ref{thm:approx-conic-oracle} is based on the JLSW \cite{Jiang2020} cutting plane method.
In Section~\ref{sec:dual certificates}, we present a general technique to extract $\varepsilon$-approximate conic Farkas certificates from various methods to solve convex feasibility problems; we only list the oracle complexities here.

\begin{itemize}
\item The ellipsoid method~\cite{gls} can be modified to provide an algorithm for \textsc{Approx-Conic-Dual} with $\To(n,\varepsilon)=O(n^2\log(1/\varepsilon))$ oracle calls (Section~\ref{sec:ellipsoid}).
\item Geometric rescaling algorithms in~\cite{DVZ,rothvoss} can be modified to provide an algorithm for \textsc{Approx-Conic-Dual} with $\To(n,\varepsilon)=O(n^3\log(1/\varepsilon))$ oracle calls  (Section~\ref{sec:rescale}).
\item Volumetric cutting plane methods \cite{Jiang2020,Lee2015,vaidya96} can be used to provide an algorithm for \textsc{Approx-Conic-Dual} with $\To(n,\varepsilon)=O(n\log(1/\varepsilon))$ oracle calls (Section~\ref{sec:Vaidya}).  We note that \cite{Jiang2020} implies  the bounds $\To(n,\varepsilon)=O(n\log(n/\varepsilon))$ and $\Ta(n,\varepsilon)=O(n^3\log(n/\varepsilon))$, see Theorem~\ref{thm:JLSW}.
\end{itemize}
A common feature of the above methods is that, when applied to the intersection $K\cap \B^n(1)$ of the cone and the unit ball, whenever they do not determine feasible point,  an ``$\varepsilon$-thin direction'' is identified, that is, an oracle inequality $m_t$ such that $m_t^\top x\le \varepsilon\|m_t\|\cdot \|x\|$ for every $x\in K$. By convex duality, there must exist a dual certificate of this bound using inequalities returned by the oracle during the course of the algorithm; such certificate provides an $\varepsilon$-approximate Farkas certificate.

Duality  is ostensibly absent from the original ellipsoid method or from Vaidya's method: at first
sight, they  appear to be ``primal'' methods only, where infeasibility
is concluded by a volumetric argument, relying on the assumption that the feasible region has a sufficiently large volume, without returning a Farkas
certificate of infeasibility. Furthermore, the ellipsoid method does not  maintain a certificate of the fact the feasible region is contained within the current ellipsoid.

In a remarkable paper, Burrell and Todd \cite{burrell-todd} showed that, in the context of the ellipsoid method, both these
 shortcomings are illusory. They introduced a new view of the ellipsoid method in terms of what we will refer to in Section~\ref{sec:dual certificates} as \emph{`certified concave quadratic forms'}. The ellipsoid $E$ produced by the algorithm at any iteration is   maintained in the form $E=\{x\in\R^n:\: q(x)\ge 0\}$, where the
 strictly concave quadratic form $q(x)$ is built from the defining constraints of $P$ and the initial ball constraint $\|x\|\le r$  in a way that immediately
 verifies the containments $P\subseteq E$. Furthermore  Burrell and Todd  showed that, from such a representation, one can construct dual certificates for any bound that holds for a linear function over the ellipsoid $E$.

We extend Burrell and Todd's framework beyond the ellipsoid method.  A further illustration is given on geometric rescaling algorithms, by showing how certified quadratic forms can be maintained during the execution of such an algorithm.
For volumetric cutting plane methods, there is no additional overhead in maintaining the quadratic forms. We show that the final output of the algorithm can be converted to a certified concave quadratic form.
\medskip

Nemirovski, Onn, and Rothblum \cite{nemirovski2010} extended the work of Burrell and Todd by giving a very
general certification procedure for the oracle model. Consider any convex
minimization problem given by oracle access, returning separators for
infeasible points and subgradients of the objective function for
feasible points, and consider  an algorithm  (such as variants of
cutting plane methods),
that can find a feasible solution with objective value within
$\varepsilon>0$ from the optimum value. Under mild
assumptions, they show that it is possible to construct a dual certificate of
the approximate optimality of the solution as an appropriate conic
combination of the separators and subgradients obtained during the algorithm.
Any such certification procedure should be applicable to implement \textsc{Approx-Conic-Dual}.

\subsection{Overview of techniques}

\paragraph{Adaptive bound on $\delta_M$}
Even for an explicitly given matrix $M$,
 $\delta_M$ is computationally hard to approximate, see  Remark~\ref{rmk:remark-inapproximable}.  Nevertheless, our algorithms are all ``oblivious'' to the value of $\delta_M$: we do not need to know this parameter to terminate within the claimed number of oracle calls. We use a standard guessing procedure, starting with the optimistic estimate $\hat\delta=1/n$, and run the algorithm with this value. The precision $\varepsilon$ required from {\sc Approx-Conic-Dual} will depend on our adaptive estimate  of $\hat\delta$. The algorithm may succeed even if $\hat\delta>\delta_M$. In case the algorithm fails to deliver the required conclusions, it will be able ``certify'' such failure, by returning a set of linearly independent rows $\{m_i\st i\in J\}$ along with coefficients $\lambda_i\in \R^J$ such that $\sum_{i\in J}\lambda_i m_i<\hat\delta\max_{i\in J}\lambda_i \|m_i\|$, thus showing  $\hat\delta>\delta_M$. We can then update our guess to either the bound implied by these vectors or to $\hat\delta^2$, whichever is smaller, and simply resume the algorithm from the current iteration with the new estimate $\hat \delta$.

Hence, if our algorithm has not succeeded for the very first time, we will have the guarantee that $\hat\delta\ge \delta_M^2$ for all subsequent trials.
Assuming the running time of each trial is bounded by $\mathrm{poly}(n,\log(1/\hat\delta))$, the overall running time bound of all  trials will be dominated by the running time of the final, successful trial.
\medskip

\paragraph{Strong conic feasibility}
Consider the strong conic feasibility algorithm for a cone $K$ and constraint $m_1^\top x>0$. We  call the subroutine \textsc{Approx-Conic-Dual} for the cone $K_1=K\cap \{x\in\R^n \st m_1^\top x>0\}$ and $\varepsilon=\hat\delta/O(n)$. The algorithm terminates if a feasible solution is found. Otherwise, an $\varepsilon$-approximate conic Farkas certificate $\lambda\in\Rp^m$ is returned.

Assume first the certificate $\lambda$  satisfies $M^\top \lambda = 0$. The certificate shows that $K\subseteq \ker(M_J)$, where $J=\supp(\lambda)$. In particular, if $\lambda_1>0$, then this shows that $m_1^\top x=0$ for all $x\in K$, and the algorithm terminates. Otherwise, the algorithm recurses on the lower dimensional space $\ker(M_J)$. In case $M^\top \lambda \neq 0$, we use a  Carath\'eodory-style subroutine that either succeeds in finding another nonzero $\lambda'\in\Rp^m$,  with linearly independent support such that $M^\top\lambda' =0$, in which case we proceed as above, or fails in finding such a vector, in which case it returns a certificate that our adaptive estimate $\hat\delta$ was incorrect (that is, $\hat\delta>\delta_M$).

\paragraph{Conic validity and conic minimum-ratio} The algorithms for conic validity and conic minimum-ratio are more involved, due to the fact that the number of iterations should only depend on $n$ and $\delta_M$, but not on $c$ and $d$. In particular, the simple strategy adopted for strong conic validity does not work, as it would require a level of precision $\varepsilon$ dependent on $c$ and $d$.

We briefly outline the main idea for the conic validity algorithm; {the conic-minimum ratio algorithm is a further extension of this idea.} The conic validity algorithm for the cone $K$ and vector $c$, calls the subroutine \textsc{Approx-Conic-Dual} for the cone $K_c=K\cap \{x\in\R^n \st c^\top x>0\}$ and $\varepsilon=\hat\delta^2/O(n^2)$. We terminate in case a feasible solution is found. Otherwise, we consider the $\varepsilon$-approximate conic Farkas certificate $(\lambda,\tau)\in\Rp^m\times \Rp$ returned, where $\tau$ is the multiplier for the inequality $-c^\top x<0$.
If $\tau$ is sufficiently small, then we can remove $c$ from the certificate and recurse as in the strong feasibility algorithm. If $\tau$ is large, then---assuming  $\hat\delta\le \delta_M$---the inequalities of $Mx\ge 0$ corresponding to suitably large entries of $\lambda$ have to be satisfied at equality for every $x\in K$. We recurse on the lower dimensional space and continue. At any step we will therefore have a set $F\subseteq [m]$ with the guarantee that $K_c\cap \ker(M_F)\neq\emptyset$, assuming $\hat\delta\le \delta_M$. However, observe that, if our estimate $\hat\delta$ was actually not correct, it is possible that we recursed on the wrong subspace, that is, $K_c\cap \ker(M_F)=\emptyset$. (This is in contrast with
the feasibility algorithm described above, where $K_1\subseteq \ker(M_F)$ is always guaranteed when recursing.) The algorithm recurses until it either finds $x\in K\cap \ker(M_F)$ such that $c^\top x>0$, in which case we stop with the feasible solution $x$, or a dual certificate $\tilde \lambda$ of the fact that $c^\top x\le 0$ for all $x\in K\cap \ker(M_F)$. The main technical tool at this stage is an algorithm, described in Lemma~\ref{lem:combine-vectors}, is a \emph{pullback} subroutine. Starting from $\tilde \lambda$, it either produces a dual certificate $\lambda\in\Rp^m$ such that $M^\top \lambda=-c$, in which case we stop, or detects a failure for $\hat \delta$, in which case we update our bound $\hat \delta$ and continue.

\subsection{Further related work}
The Burrell--Todd representation \cite{burrell-todd} was also used by Lamperski, Freund, and Todd \cite{lamperski} to  developed an ``oblivious ellipsoid algorithm'' that terminates in finite time, assuming $P$ is explicitly given by inequalities, and that $P$ is either full-dimensional or empty. In contrast, our result is applicable also for non full-dimensional LPs defined by a separation oracle. We also note that whereas \cite{lamperski} uses a modification of the standard ellipsoid method, our approach uses the standard method in a black-box manner.

\emph{Geometric rescaling} is a more recent class of polynomial-time linear programming algorithms: the common theme of such algorithms is to boost simple iterative algorithms by adaptively changing the scalar
  product. The first such algorithms were given  by Betke \cite{betke} and by Dunagan and Vempala
  \cite{Dunagan-Vempala}, and a number of papers have since appear on the subject. We refer the reader to \cite{DVZ} for an overview of such results. Whereas most of these algorithms work only under the assumption that the constraints defining the cone are explicitly given as part of the input, some variants, including those described in \cite{DVZ,rothvoss}, can be naturally extended to the oracle setting.  We implement the  approximate conic oracle in Section~\ref{sec:rescale} for these variants.

Theorem~\ref{thm:conic-main} also gives an answer to a question raised in \cite{DVZ} on finding a ``primal-dual'' geometric rescaling algorithm for the conic maximum support problem that does not depend on a priori bounds on the condition numbers. Such an algorithm was also obtained for explicitly given systems by Pena and Soheili \cite{Pena2020}. Their algorithm runs the rescaling algorithms in parallel on the primal and dual problems, with increasing estimates on a certain condition number.

\paragraph{Organization of the paper}

 Section~\ref{sec:prelim} introduces basic notation and concepts, provides the proof of Theorem~\ref{thm:LP-main}, and proves some fundamental facts about the condition measure $\delta_M$.
Sections \ref{sec:conic-feas}, \ref{sec:conic-validity}, and \ref{sec:optimization} describe the algorithms for the strong conic feasibility, conic validity, and minimum conic ratio problems, along with their analyses.
Sections~\ref{sec:quadratic-certify} and \ref{sec:thin} describe our general approach to finding approximate Farkas certificates from certified quadratic forms. Sections~\ref{sec:ellipsoid}--\ref{sec:rescale} describe different implementations of the oracle {\sc Approx-Conic-Dual}, based on different methods (ellipsoid, volumetric cutting plane, and geometric rescaling). Section~\ref{sec:rational} relates the results presented to the classical framework of rational polyhedra. Finally, Section~\ref{sec:circuits} shows the connection between the condition number $\delta$ and the circuit imbalance measure.  Appendix~\ref{sec:impossibility} includes missing proofs, in particular for the the impossibility results of Propositions~\ref{prop:no-strongly} and \ref{prop:need-b}.

\section{Preliminaries}\label{sec:prelim}
We let $\Rp$ denote the set of nonnegative reals.
 For natural numbers $k<m$, let $[m]=\{1,2,\ldots,m\}$, $[k,m]=\{k,k+1,\ldots,m\}$.
For any number $\alpha\in\R$, we let $\alpha^+=\max\{\alpha,0\}$ and
$\alpha^-=\max\{-\alpha,0\}$.  For a vector $x\in \R^n$, $x^+$ and $x^-$ in $\R^n$ are defined by  $(x^+)_i=x_i^+$, $(x^-)_i=x_i^-$, $i\in[n]$. Thus, $x=x^+-x^-$.
Let $\vec{e}_j$ denote the $j$th unit vector in $\R^n$.
For a set of vectors $\{v_j\st j\in J\}\subseteq \R^n$, we let $\spn(v_j\st j\in J)\subseteq \R^n$ the linear subspace they span; for a matrix $B\subseteq \R^{n\times m}$, let $\spn(B) \subseteq\R^n$ denote the linear subspace spanned by the columns of $B$.

For any matrix $H\in\R^{k\times n}$ and every $J\subseteq [k]$, we denote by $H_J$ the submatrix of $H$ defined by the rows indexed by $J$, and similarly, for $v\in\R^k$, $v_J$ defined the restriction of $v$ to the entries indexed by $J$.

  $K \subseteq \R^n$ is a \emph{cone} if $K$ is convex and $K$
is closed under positive scalings, that is, for every $x\in K$ and $\lambda>0$ it follows that $\lambda x
\in K$. For a set of vectors $v_1,\dots,v_k \in \R^n$,
we let ${\rm cone}(v_1,\dots,v_k) \coloneqq  \left\{\sum_{i=1}^k \lambda_i v_i\,:\, 
\lambda_1,\dots,\lambda_k \geq 0\right\}$ denote the \emph{closed} cone generated by
$v_1,\dots,v_k$.

For a convex set $C \subseteq \R^n$, we say that $F \subseteq C$ is a \emph{face} of $C$ if $F$ is convex
and if for all $x,y \in C$, we have that $\lambda x + (1-\lambda) y \in F,
\lambda \in (0,1)$, implies that $x,y \in F$. The
\emph{lineality space} of $C$ is the largest linear subspace $W\subseteq \R^n$ such that $C+W =
C$. A closed convex set $C$ is said to be \emph{pointed} if its lineality space is $W =
\{0\}$. For a closed pointed cone $K$, the set of $1$-dimensional faces of
$K$ are called the \emph{extreme rays} of $K$. Slightly abusing notation, we will
also say that $v \in K \setminus \{0\}$ is an extreme ray of $K$ if 
$\spn(v)$
is a $1$-dimensional face of $K$. For two sets $X,Y\subseteq \R^n$, we let $X+Y\coloneqq\{x+y\st x\in X\, , \, y\in Y\}$ denote their Minkowski sum.

Given $p\in \R^n$ and $r>0$, we denote by $\bb{B}^n(p,r)$ the ball of radius $r$ in $\R^n$ centered at $p$. We use the notation $\bb{B}^n(r)$ for $\bb{B}^n(0,r)$.
We denote by $\Snpp$ and $\Snp$ the sets of symmetric $n\times n$ positive
definite and positive semi-definite matrices, respectively. For $P,Q\in \Snp$, we use $P\preceq Q$ if
$Q-P\in \Snp$.
For $Q\in \Snpp$ and a vector $v\in\R^n$, we let $\|v\|_Q
\eqdef\sqrt{v^\top Qv}$; this defines a norm over $\R^n$.
 We use
$\|\cdot\|_1$ for the $\ell_1$-norm and $\|\cdot\|_2$ for the Euclidean norm.
 When there is no risk of confusion we simply
write $\|\cdot\|$ for $\|\cdot\|_2$.

The following simple claim will be needed for running time estimations when using an adaptive bound on the condition number $\delta_M$.
\begin{lemma}\label{lem:delta-log}
Let $1/n=\delta_1> \delta_2>\ldots> \delta_t$ and $\delta>0$ be real numbers such that $\delta_{i+1}<\delta_i^2$ for $i\in [t-1]$, $\delta_t> \delta^2$, and $\nu\ge 1$. Then,
\[
\sum_{i=1}^t \log^\nu(n/\delta_i)=O(1)\cdot\log^\nu(n/\delta)\, .
\]
\end{lemma}

\paragraph{Computational assumptions}
In  Sections~\ref{sec:conic-feas}--\ref{sec:conic-validity}, we use the real RAM model of computation, allowing basic arithmetic operations $+$, $-$, $\times$, $/$ and comparisons. The computations used in the black-box subroutines beyond the oracle calls can also be implemented in the Turing model without any significant modifications.

In Section~\ref{sec:dual certificates}, we construct dual certificates using the output of various algorithms. The algorithms used here, in particular the JLSW algorithm \cite{Jiang2020}, use the real model and also use squareroot computations. In this section, we also use squareroot computations. These might be circumvented, but we use this for simplicity, and given that it is already used in the algorithms we rely on.

\paragraph{Separation oracle variants}\label{sec:oracles}
 For a convex set $K\subseteq \R^n$, a \emph{strong separation oracle} takes as input a point $\bar x\in \R^n$, and either
 either returns the answer $\bar x\in K$, or a nonzero
 vector $\separ\in \R^n$, such that $\separ^\top x<\separ^\top \bar x$ for every $x\in K$.
This is the  standard  separation oracle model used for the ellipsoid and other cutting plane methods.
 The notion of conic separation oracle required for {\sc Approx-Conic-Dual} oracle is, as discussed in the Introduction, identical to the strong separation oracle if the cone $K$ it defines is non-empty.

Recall that our main results Theorem~\ref{thm:LP-main} and Theorem~\ref{thm:conic-main}, make stronger oracle assumptions. Namely, we assume that a polyhedron $P$, is defined by a \emph{polyhedral separation oracle}: if $\bar x\not\in P$, the oracle returns an  inequality $a^\top  x\le \beta$ violated by $\bar x$, where the set of all inequalities returned by the oracle for all possible choices of $\bar x\notin P$ is finite. We will often write that $P=\{x\st Ax\le b\}$ is defined by a polyhedral separation oracle to mean that $Ax\le b$ comprises all such possible inequalities, with the understanding that $Ax\le b$ is not explicitly given, but we have access to it via the oracle.

In Section~\ref{sec:rational} we show that for rational polyhedra satisfying assumption \eqref{eq:bounded-facet-assumption} on bounded bit-complexity, a strong separation oracle can be converted into a  polyhedral separation oracle.

\subsection{Reducing LP to the conic setting: proof of Theorem~\ref{thm:LP-main}}\label{sec:main-proof}

We now give the proof of Theorem~\ref{thm:LP-main} using Theorems~\ref{thm:approx-conic-oracle} and~\ref{thm:conic-main}.
Consider a polyhedron $P=\{x\in \R^n\st Ax\le b\}$ for some $A\in \R^{m\times
n}$, $b\in \R^m$ and its recession cone ${\rm rec}(P) = \{x \in \R^n \st -Ax
\geq 0\}$, both of which are given by a polyhedral separation oracle.

We  derive parts {\em (i), (ii), (iii)} of Theorem~\ref{thm:LP-main} from the corresponding part of Theorem~\ref{thm:conic-main} using
the standard homogenization of $P$ into $\R^{n+1}$. Namely, we examine
\begin{equation}
K = \{(x \mid t) \in \R^{n+1}: tb - Ax \geq \bO, t \geq 0\} = \{(x \mid t) \in \R^{n+1}: M(x \mid t) \geq \bO\} \label{eq:homogenize},
\end{equation}
where $M \in \R^{(m+1) \times (n+1)}$ is as in \eqref{eq:M-matrix}. Note that $(x
\mid 0) \in K \Leftrightarrow x \in {\rm rec}(P)$ and that $(x \mid t) \in K,
t > 0
\Leftrightarrow x/t \in P$.\footnote{\label{footnote:oracle}We cannot directly build a
polyhedral conic separation oracle for $K$ given our assumptions. In
particular, for an input $(x\mid 0)$ to the oracle, if $(x\mid 0) \notin K$, we would
need to return $(-a_i \mid b_i)$ such that $-a_i^\top x < 0$. Our polyhedral
separator for ${\rm rec}(P)$ would give us access to $a_i$ but not to $b_i$,
noting that the inequality $(-a_i \mid 0)^\top (x \mid t) \geq 0$ is not
necessarily valid for $K$. We will be able to circumvent this issue however,
as the problems we wish to solve will only require polyhedral conic
separation oracles for sub-cones of $K$, which we will be able to build
directly.}

\paragraph{\bf (i) Primal feasibility} We must compute a solution to $Ax \leq b$
or a Farkas certificate of infeasibility $A^\top \lambda = 0, b^\top \lambda <
0, \lambda \geq 0$. We reduce to solving strong conic feasibility using
Theorem~\ref{thm:conic-main} on $K$ as above with the constraint $t > 0$ corresponding to $m_1$. For
this purpose, we require a polyhedral separator for $K_1 \coloneqq  \{(x \mid t) \in K: t >
0\}$, which can be derived directly from the polyhedral separation oracle for
$P$. Namely, given $(x' \mid t')$, if $t' \leq 0$, we return the separator $t >
0$, and if $t > 0$, we call the separator for $P$ on $x'/t'$. If $x'/t'$
violates $a_i x \leq b_i$ for $P$, we return $(-a_i \mid b_i)$
for $K_1$. From here, if Theorem~\ref{thm:conic-main} returns $(x \mid t) \in K_1$,
we return $x/t \in P$, and if it returns $\lambda \in \Rp^m$ satisfying
$\lambda^\top(-A,b) + 1 \cdot (0,1) = 0$, we return $\lambda$ as the Farkas
certificate.

\paragraph{\bf (ii) Dual feasibility} We must either find a solution to $Ax \leq 0$, $c^\T x > 0$, or to $A^\top
\lambda = c$, $\lambda \geq 0$.
For this purpose, we reduce to the conic validity problem using
Theorem~\ref{thm:conic-main} on the cone $K = {\rm rec}(P) \coloneqq  \{x \in \R^n:
-Ax \geq 0\}$ and the vector $\bar{c} \coloneqq  -c$. This requires a polyhedral
conic separation oracle for $K_{\bar{c}} = \{x \in \R^n: -Ax \geq 0, c^\top x
> 0\}$. This is immediate to implement since we have access to a polyhedral
separation oracle for ${\rm rec}(P)$ and $c$ is known to us. Since the conic
validity problem is a direct restatement of the dual feasibility, the
correctness of the reduction is evident.

\paragraph{\bf (iii) Optimization} Assuming both $Ax \leq b$ and $A^\top
\lambda = c, \lambda \geq 0$ are feasible, we must find an optimal
primal-dual pair $x^*,\lambda^*$ satisfying complementary slackness, namely
$(\lambda^*)^\top (b-Ax^*) = 0$. We reduce this to the conic minimum-ratio
problem on $K$ given by $\min \{-c^\top x/t: (x \mid t) \in K, t > 0\}$ using
Theorem~\ref{thm:conic-main}. Note that this problem can be rewritten as
$\min \{\bar{c}^\top (x \mid t)/d^\top (x \mid t): (x \mid t) \in K,
d^\top (x \mid t) > 0\}$, where $\bar{c} = (-c \mid 0)$ and $d=(0
\mid 1)$.

For this purpose, we first require that $d = (0 \mid 1)$ be given as a
conic combination of the original constraints of $K$, which trivially holds
since $(0 \mid 1)$ induces an original constraint itself; hence $I=\{1\}$ and $y^{(d)}_1=1$. We also require
polyhedral separators for $K_{-d} = \{(x \mid t) \in K: d^\top x > 0\} = \{(x \mid t)
\in K: t > 0\}$ and for $K_I^= = \{(x \mid t) \in K: d^\top x = 0\} = \{(x \mid 0)
\in K\} \coloneqq  ({\rm rec}(P) \mid 0)$. As explained in the previous paragraphs, these
separators can be directly constructed from the corresponding polyhedral
separators for $P$ and ${\rm rec}(P)$. Furthermore, we recall that
feasibility of $Ax \leq b$ is equivalent to $K_{-d} \neq \emptyset$ (i.e.,
$d^\top(x \mid t) = 0$ is not valid for $K$) and feasibility of $A^\top
\lambda = c, \lambda \geq 0$ is equivalent to $-c^\top x = \bar{c}^\top(x \mid t)
\geq 0$ being a valid inequality for $K_I^=$.

Given the above, the conic minimum-ratio solve must output $\gamma^* \in \R$,
$(x^* \mid t^*) \in K_{-d}$, $\lambda^* \in \Rp^m, \beta^* \geq 0$ such that 
$(\bar{c} - \gamma^* d)^\top (x^* \mid t^*)=0$ and
$(\lambda^*)^\top (-A,b) + \beta^*(0,1) = \bar{c}^\top - \gamma^*
d^\top = (-c,-\gamma^*)$. We claim that $x^*/t^*, \lambda^*$ are the
desired optimal primal-dual pair. To begin, we note that the inclusion $x^*/t^* \in P$ is direct since $t^* > 0$. Furthermore, by the guarantees of the output we have that
\[
0 = (\bar{c} - \gamma^* d)^\top (x^* \mid t^*) = ((\lambda^*)^\top(-A,b) + \beta^* (0,1)) (x^* \mid t^*) = (\lambda^*)^\top(t^* b - Ax^*) + \beta^* t^* \geq 0 + 0 = 0.
\]
Since $t^* > 0$, the above implies that $\beta^* = 0$ and that
$(\lambda^*)^\top(t^* b - Ax^*) = 0$. Since $\beta^* = 0$, we see that
$\lambda^*$ is a valid dual solution. Finally, complementary slackness
follows from $(\lambda^*)^\top(t^* b - Ax^*)=0$ after dividing by $t^*$.

For all three problems above, the desired running times now follow directly
by combining Theorem~\ref{thm:approx-conic-oracle} with the corresponding
part of Theorem~\ref{thm:conic-main}.

\subsection[Properties of the delta-measure]{Properties of the $\delta$-measure}

We start by showing that the $\delta$-measure introduced in Definition~\ref{def:delta} is equivalent to the definition of the ``$\delta$-distance property'' studied in \cite{brunsch2013,dadush2016shadow,eisenbrand2017}, and that it is positive if and only if $V$ is finite.

\begin{lemma}\label{lem:delta-lem}
For a set of vectors $V\subseteq \R^n$, $\delta_V$ is the largest value such that, for every $W\subseteq V$ and every $v\in V\sm \spn( W)$, the Euclidean distance between $v$ and $\spn(W)$ is at least $\delta_V\|v\|$. Further, $\delta_V>0$ if and only if $\{v/\|v\|\st v\in V\setminus\{\bO\}\}$ is a finite set.
\end{lemma}
\begin{proof} 
Let $\delta'_V$ be the quantity in the statement; we prove $\delta'_V=\delta_V$. We can  assume that $\bO\notin V$ and that $\|v_j\|=1$ for all $j\in V$, noting that  both quantities $\delta_V$ and $\delta'_V$ are invariant under rescaling vectors in $V$. In particular, we may assume that $V=\{v/\|v\|\st v\in V\}$. 

For the first part, we start by showing $\delta_V\ge \delta'_V$. Let $\{v_j\st j\in I\}\subseteq V$ be a linearly independent set of vectors; for any vector $v_i$, $i\in I$, let $z_i$ be the projection of $v_i$ to $\spn( v_j\st j\in I\sm\{i\})^\perp$. By definition, $\|z_i\|\ge \delta'_V \|v_i\|=\delta'_V$. By the definition of $z_i$,  for every $\lambda\in\R^I$  the projection of $\sum_{k\in I}\lambda_k v_k$ to $\spn( v_j\st j\in I\sm\{i\})^\perp$ is $\lambda_i z_i$, hence $\left\|\sum_{k\in I}\lambda_k v_k\right\|\ge |\lambda_i|\|z_i\|\ge  \delta'_V |\lambda_i|$. This shows $\delta_V\ge \delta'_V$.

For the direction $\delta'_V\ge \delta_V$, let $W\subseteq V$ and  $v\in V\sm \spn(W) $. W.l.o.g.,  we can assume that the vectors in $W$ are linearly independent, hence the vectors in $W\cup \{v\}$ are linearly independent. Denoting by $z$ be the projection of $v$ to $\spn(W)^\perp$, it follows that there exist a unique vector $\lambda\in\R^{|W|}$ such that $W\lambda = v-z$. Since the vectors in $W\cup \{v\}$ are linearly independent, it follows from the definition of $\delta_V$ that $\delta'_V=\|z\|=\|W\lambda-v\|\ge \delta_V \|v\|=\delta_V$, thus, $\delta'_V\ge \delta_V$. %

For the second part of the statement, if $|V|$ is finite, then $\delta'_V$ is defined as the minimum of a finite number of positive numbers, showing $\delta_V=\delta'_V>0$. If $|V|$ is infinite, then
there exists a convergent sequence in $V$ (since we assume that $V\subseteq \B^n(0,1)$); let $v_k$, $k\in \N$ be such a sequence. Thus, for any $\varepsilon>0$, there exists $v_j, v_k\in V$
 such that $\left\|v_j-v_k\right\|<\varepsilon$. Setting $I=\{j,k\}$, $\lambda_j=1$, $\lambda_k=-1$, we have $\|\sum_{i\in I}\lambda_i v_i\|=\left\|v_j-v_k\right\|<\varepsilon$; on the other hand, $|\lambda_j|=|\lambda_k|=1$. This shows that $\delta_V<\varepsilon$ for every $\varepsilon>0$.
\end{proof}

The following characterization can be shown with a similar argument.
\begin{lemma}[{\cite[{Lemma 5(i)}]{brunsch2013}}]\label{lem:inverse-norm}
Consider a matrix $M\in \R^{m\times n}$ such that all rows $m_i^\top$ have norm one. For a matrix $B\in \R^{m\times m}$, let $\gamma(B)$ denote the maximum column norm of $B$.
Then, \[
\frac{1}{\delta_M}=\max\left\{\gamma(N^{-1}):\, N\mbox{ is an $m\times m$ submatrix of }M\right\}
\]
\end{lemma}

The following is the key property of $\delta_M$ in the conic setting. Namely,
it gives a lower bound  on the angles that extreme rays of
$Mx \geq 0$ can form with constraints that they are not incident to.

\begin{lemma}\label{lem:cone-angle-delta}
Let $K=\{x\in\R^n\st Mx\ge 0\}$ be a closed polyhedral cone with $M\in\R^{m\times n}$. We have that $K \cap \spn(M^\top)$ is a  closed pointed cone and for any of its  extreme rays $v$
and any $i \in [m]$, we have that either $m_i^\top v = 0$ or $m_i^\top v \geq \delta_M \|m_i\|\cdot\|v\|$, $\forall i \in [m]$.
\end{lemma}
\begin{proof}
As the statement is invariant under rescaling $v$, we can assume without loss of generality that  $\|v\|=1$.
We have that $\ker(M)=\spn(M^\top)^\perp$ is  the lineality space, hence $K \cap \spn(M^\top)$ is pointed. Let $v$ be an extreme ray of $K \cap \spn(M^\top)$ and $i
\in [m]$, and let $S=\{j\in [m]: m_j^\top v=0\}$. Thus, $\ker(M_S)\cap \spn(M^\top)$ is a one-dimensional space spanned by $v$, because $v$ is an extreme ray. If $i\in S$ then
$m_i^\top v = 0$. Hence, assume $i\notin S$, and let $\bar m_i$ the orthogonal projection of $m_i$ onto $\ker(M_S)$. Since $\ker(M_S)\cap\spn(M^\top)$ is the 1-dimensional space spanned by $v$ and since $m_i \in \spn(M^\top)$, it follows that $\bar m_i=(m_i^\top v)v$, so  $m_i^\top v = \|\bar m_i\|$. Since by construction $\|\bar m_i\|$ is the Euclidean distance between $m_i$ and $\spn(M_S^\top)$, by Lemma~\ref{lem:delta-lem} we have $\|\bar m_i\|\ge \delta_M \|m_i\|$  and the statement thus follows. 
\end{proof}

\section{The strong conic feasibility algorithm}\label{sec:conic-feas}

In this section, we prove part (i) of Theorem~\ref{thm:conic-main}, pertaining the strong conic feasibility problem.
Recall that we assume that a polyhedral conic separation oracle is available for
\[K_1 =\{x \in \R^n: Mx \geq 0, m_1^\top x > 0\}\, ,\]
 and that the subroutine \textsc{Approx-Conic-Dual} for $K_1$ is provided as in Section~\ref{sec:our-contr}, requiring $\mathcal{T}_o(n,\varepsilon)$ oracle calls, $\mathcal{T}_a(n,\varepsilon)$ arithmetic operations, and returning an $\varepsilon$-approximate conic Farkas certificate of size at most $\tau(n)$.

The next lemma captures the key recursive step:
\begin{lemma}\label{lem:exact-feas-recurse}
Let $K=\{x\in \R^n\st M x\ge 0\}$ for $M\in\R^{\km \times n}$, given by a conic separation oracle, and let $m_1^\top$ be the first row.
There exists an oracle-polynomial-time algorithm using $O(\mathcal{T}_o(n,\delta_M/(2n\sqrt{n})))$
oracle calls and $O(\mathcal{T}_a(n,\delta_M/(2n\sqrt{n}))+n^2\tau(n)\log\log(1/\delta_M))$ arithmetic operations
that either finds an $x\in K$ with $m_1^\top x>0$, or a nonzero  $\lambda\in \Rp^\km$  that is a minimal support solution to $M^\top \lambda=0$.
\end{lemma}
In Section~\ref{sec:recurse-feas}, we show how the strong conic feasibility algorithm can be obtained by at most $n$ calls to this subroutine.
The proof of Lemma~\ref{lem:exact-feas-recurse} relies on the decomposition stated in the next lemma. This is essentially a careful reading of the proof of Carath\'eodory's theorem. It is a consequence of \cite[Lemma 4.1]{DadushNV20}; we include the proof for completeness.

\begin{lemma}\label{lem:conic-Caratheodory}
There exists an $O(n^2 |J|)$ time algorithm that, given vectors $\{v_j\st j\in J\}$, $\lambda\in \Rp^J$,  and $w\in \R^J$ such that $w^\top \lambda\ge 0$,
 outputs one of the following.

\begin{enumerate}[(i)]
\item\label{eq:ct-indep}  A vector $\bar\lambda \in \Rp^J$
such that $\sum_{j\in J} \bar\lambda_j v_j=\sum_{j\in J}\lambda_j v_j$,  $w^\top \bar\lambda\ge w^\top \lambda$, and the vectors $\{v_j: \bar\lambda_j>0\}$ are linearly independent.
\item \label{eq:ct-circuit} A  vector $\mu \in \Rp^J$, 
which is a support-minimal solution to $\sum_{j\in J} \mu_j v_j=0$, and such that  $w^\top \mu> 0$.
\end{enumerate}
\end{lemma}
\begin{proof}
We initialize  $\bar\lambda=\lambda$, and maintain  $\sum_{j\in J} \bar\lambda_j v_j=\sum_{j\in J}\lambda_j v_j$,  $w^\top \bar\lambda\ge w^\top \lambda\ge 0$ throughout. At every iteration, either $\supp(\bar\lambda)$ becomes strictly smaller, or we find a vector $\mu$ as in \eqref{eq:ct-circuit}. If we do not end with outcome \eqref{eq:ct-circuit} at some iteration, then we terminate once $\{v_j: \bar\lambda_j>0\}$ are linearly independent.

If the vectors in the support of $\bar\lambda$ are linearly dependent, then let $\mu\neq 0$ be a support-minimal vector such that $\sum_j \mu_j v_j=0$, $\supp(\mu)\subseteq \supp(\bar \lambda)$. If $\mu\ge 0$ and $w^\top \mu> 0$ or $\mu\le 0$ and $w^\top \mu< 0$, then we output $\mu$ or $-\mu$, respectively, and terminate with outcome \eqref{eq:ct-circuit}. Otherwise, possibly by replacing $\mu$ by $-\mu$, we can assume that $w^\top \mu\ge 0$ and $\mu$ has a negative component. Let $\alpha>0$ be the largest number such that $\bar\lambda+\alpha\mu\ge 0$, and update $\bar\lambda\coloneqq\bar\lambda+\alpha\mu$. Note that $|\supp(\bar\lambda)|$ decreases by at least one, and by the choice of $\mu$, we have $\sum_{j\in J} \bar\lambda_j v_j=\sum_{j\in J}\lambda_j v_j$,  $w^\top \bar\lambda\ge w^\top \lambda$.

The running time bound is standard; it takes $O(n |J|\min\{n,|J|\})$ arithmetic operations to bring the system $\sum_{j\in J} v_j \mu_j=0$ (in the variables $\mu_j$, $j\in J$) in normal echelon form via Gaussian elimination. At every iteration, if we reduce the support of $\bar \lambda$, we remove the vectors corresponding to zero components of $\bar \lambda$, and it takes $O(n|J|)$ operations per vector removed to bring the system back to echelon normal form, for a total of $O(n|J|^2)$ operations.
\end{proof}

For $w=\bO$, outcome~\eqref{eq:ct-circuit} cannot occur, and hence we get the algorithmic version of Carath\'eodory's theorem, formulated as follows:
\begin{corollary}[Carath\'eodory reduction]\label{cor:carath}
There exists an $O(n^2 |J|)$ time algorithm that, given vectors $\{v_j\st j\in J\}$, $\lambda\in \Rp^J$,  returns $\bar\lambda \in \Rp^J$
such that $\sum_{j\in J} \bar\lambda_j v_j=\sum_{j\in J}\lambda_j v_j$,  and the vectors $\{v_j: \bar\lambda_j>0\}$ are linearly independent.
\end{corollary}

\begin{proof}[Proof of Lemma~\ref{lem:exact-feas-recurse}]
We maintain an estimate $\hat\delta$ on $\delta_M$, initializing $\hat\delta\coloneqq1/n$. We say that a nonzero vector $\lambda\in\R^m$ is a {\em failure for $\hat \delta$} if $\{m_j\st j\in\supp(\lambda)\}$ are linearly independent and $\varphi\coloneqq\|M^\top \lambda\|/(\max_{j\in [m]}\lambda_j\|m_j\|)< \hat \delta $, proving $\delta_M\le \varphi<\hat \delta$. Whenever we detect a failure, we update $\hat\delta\coloneqq\min\{\hat\delta^2,\varphi\}$.

We call the subroutine \textsc{Approx-Conic-Dual} for $K_1$  for $\varepsilon\coloneqq\hat\delta/(2n\sqrt{n})$.  This requires $\mathcal{T}_o(n,\hat \delta/(2n\sqrt{n}))$ oracle calls and $\mathcal{T}_a(n,\hat \delta/(2n\sqrt{n}))$ operations. Either we obtain an $x\in K$ with $m_1^\top x>0$, or $\lambda\in \Rp^\km$ such that $\sum_{j=1}^\km \lambda_j\|m_j\|_1\ge 1$ and $\|M^\top \lambda\|\le \hat\delta/(2n)$; let $J=\supp(\lambda)$. If $M^\top\lambda=0$ then we can readily return the vector $\lambda$.

If $M^\top\lambda\neq 0$, then we apply Lemma~\ref{lem:conic-Caratheodory} to $\lambda$, the vectors $\{m_j: j\in J\}$, and to the vector $w\in\R^\km$ defined by $w_j=\|m_j\|$, $j\in[k]$. Recall that $|J|\le \tau(n)$, hence this step requires $O(n^2|J|)\subseteq O(n^2\tau(n))$ arithmetic operations.
If outcome {\em (i)} occurs, then we obtain $\bar\lambda\in \Rp^\km$ with $\supp(\bar\lambda)\subseteq J$, $\sum_{j=1}^\km \bar\lambda_j\|m_j\|\ge 1$ such that
$M^\top \bar\lambda=M^\top\lambda$, and the rows of $M$ in the support of $\bar\lambda$ are linearly independent, which implies $|\supp(\bar\lambda)|\le n$.
Then, $\bar \lambda$ is a failure for $\hat \delta$, since
\begin{equation}\label{eq:feas-delta-bound}
\begin{aligned}
{\hat\delta}&\ge 2n\sqrt{n} \left\| M^\top\lambda\right\|=2n\sqrt{n}\left\|M^\top \bar\lambda\right\|\ge 2n\sqrt{n} \delta_M\max_{i\in \supp(\bar\lambda)} \bar\lambda_i \|m_i\|\\
&\ge 2n \delta_M\max_{i\in \supp(\bar\lambda)} \bar\lambda_i \|m_i\|_1\ge 2{\delta_M}\, ,
\end{aligned}
\end{equation}
where the last inequality follows from $\sum_{j\in J}\bar\lambda_j\|m_j\|_1\ge 1$ and $|\supp(\bar\lambda)|\le n$. In this case we update $\hat\delta$ and $\varepsilon$ accordingly, and call $\varepsilon=\hat\delta/(2n\sqrt{n})$ for the new value of $\hat \delta$.
If outcome {\em (ii)} occurs, we obtain a nonzero $\mu\in \Rp^\km$, $\supp(\mu)\subseteq J$ such that $M^\top \mu=0$ and $\mu$ is support minimal; we can return $\mu$ as the output.
The running time bound follows using Lemma~\ref{lem:delta-log}, using the choice of the estimates $\hat\delta$.
\end{proof}

\subsection{The recursive algorithm}\label{sec:recurse-feas}

We now describe the overall strong conic feasibility algorithm, with the running time bound  stated in Theorem~\ref{thm:conic-main}.
This can be achieved by making at most $n$ calls to the
algorithm in Lemma~\ref{lem:exact-feas-recurse}. We gradually identify a subset $F\subseteq [\km]$ and find coefficients $\xi\in \Rp^\km$,  such that $\supp(\xi)=F$, $M^\top  \xi=0$, $|F|\le 2\rk(M_F)$.
This certifies that $M_F x=0$ for all $x\in K$.

This set is initialized to $F=\emptyset$. After the first call to Lemma~\ref{lem:exact-feas-recurse}, if the output is a support minimal solution $\lambda \ge 0$ to $M^\top \lambda= 0$, then we select $F=\supp(\lambda)$. $F$ will be extended at every iteration; thus, the algorithm terminates by making at most $n$ calls.

\medskip

The following notation and subsequent Lemmas~\ref{lem:projected-oracle} and~\ref{lem:projected-condition-numbers} will also be used in later sections, and applied for any $F\subseteq [\km]$ (that is, we do not require $K\subseteq \ker(M_F)$).
Given an index set $F\subseteq [\km]$, let $\Pi^{F}\in\R^{n\times n}$ be the orthogonal projection matrix onto $\ker(M_F)$. Computing $\Pi^F$ requires $O(n^\omega)$ operations.
For every vector $v\in\R^n$, let
\[v^F\coloneqq\Pi^F v.\]
Let
$T_F=\{i\in [\km]\st \Pi^F m_i\neq 0\}$ and let $M^F\in \R^{T_F\times n}$ be the matrix with rows $(m^F_i)^\top$, $i \in T^F$.  Let
\begin{equation}\label{eq:projected-cone}
K^F=\{x\in \R^n\st  M^F x\ge 0\}\quad\mbox{and}\quad K^F_1=\{x\in \R^n\st  M^F x\ge 0, (m^F_1)^\top x>0\}\, .
\end{equation}

\begin{lemma}\label{lem:projected-oracle}
For any index set $F\subseteq [\km]$ and $\bar x \in \R^n$, $\bar x\in K^F$ if and only if $\Pi^F\bar x \in K$. In particular, $K^F=(K \cap \ker(M_F))+\spn(M_F^\top)$. Given a conic separation oracle for $K_1$, we can implement a conic separation oracle for $K^F_1$, requiring $O(n^2)$ additional arithmetic operations for call to the $K_1$ oracle, assuming $\Pi^F$ is pre-computed.
\end{lemma}
\begin{proof}
For the first part, note that $\bar x \in K^F \Leftrightarrow M (\Pi^F \bar x )\geq 0 \Leftrightarrow \Pi^F \bar x \in K$. The second statement follows directly from the fact that $\Pi^F \bar x \in \{x \in \R^n: m_i^\top x = 0, i \in F\}$.
For the separation oracle for $K^F_1$, if $m^F_1=0$ then we have $K^F_1=\emptyset$; for the rest of the proof, let us assume $m^F_1\neq 0$. Take any $\bar x\in \R^n$, and run the conic separation oracle for $K_1$ for the orthogonal projection $\Pi^F \bar x$. If $\Pi^F \bar x\in K_1$, we can return $\bar x\in K^F_1$.
If not, then the oracle either returns   $m_1^\top (\Pi^F \bar x)\le 0$, that is,  $(m_1^F)^\top \bar x\le 0$, or $m_i^\top  (\Pi^F \bar x)< 0$, that is, $(m_i^F)^\top \bar x< 0$ where $i\in T_F$.
 Computing $\Pi^F \bar x$ and $\Pi^F m_i$ requires time $O(n^2)$, assuming $\Pi^F$ is available.
\end{proof}

\begin{lemma}\label{lem:projected-condition-numbers}
For any $F\subseteq [m]$, we have $\delta_{M^F}\ge \delta_M$.
\end{lemma}
\begin{proof}
Consider a set $J\subseteq T^F$ such that $\{m_j^F\st j\in J\}$ are linearly independent. Consider $\lambda\in \R^J$. There exists $\theta\in \R^F$ such that $\{m_i\st i\in \supp(\theta)\}$ is linearly independent, and
\[
\left\|\sum_{j\in J}\lambda_j \Pi^F m_j\right\|=\left\|\sum_{j\in J}\lambda_j m_j+M_F^\top \theta\right\|\ge \delta_M\max\left\{\max_{j\in J}\lambda_j \|m_j\|, \max_{i\in F}\theta_i \|m_i\|\right\}\ge \delta_M \max_{j\in J}\lambda_j \|\Pi^F m_j\|\, ,\]
where the first inequality follows from the definition of $\delta_M$, since the vectors in $\{\Pi^F m_j\st j\in J\}\cup \{m_i\st i\in \supp(\theta)\}$ are linearly independent,
and the second inequality follows from $\|m_j\|\ge \|\Pi^F m_j\|$ for all $j\in J$.
\end{proof}

Equipped with the above notation, the description of the algorithm follows. We initialize $F=\emptyset$. If at any iteration $m^F_1=0$ then $m_1^\top x=0$ must hold for all $x\in K$; thus, no strong feasible solution exists. In this case, we can obtain an infeasibility certificate as follows. If $1\in F$, then $\lambda'=\xi_1^{-1} \xi$ is a nonnegative vector with $M^\top \lambda'=0$ such that $\lambda'_1=1$. Else, if $1\notin F$, then $m^F_1=0$ is equivalent to the existence of  $\mu\in \R^\km$ with $M^\top \mu=0$ such that $\supp(\mu)\subseteq F\cup\{1\}$ and $\mu_1=1$. Then, for sufficiently large $\alpha>0$, $\lambda'=\mu+\alpha\xi$ is a nonnegative vector with $M^\top \lambda'=0$ such that $\lambda'_1=1$.

Each iteration calls the algorithm described in Lemma~\ref{lem:exact-feas-recurse} for $M^F$ and $K^F$ with the
projected separation oracle as in Lemma~\ref{lem:projected-oracle}.
If the output is a  point $x\in K^F$ with $(m_1^F)^\top x>0$, then we return the point $\Pi^F x\in K$ with $m_1^\top (\Pi^F x)>0$, that is a solution to the strong feasibility problem.

The other possible output is a nonzero vector $\hat \lambda\in \Rp^{T^F}$ that is a support minimal solution to $(M^F)^\top \hat \lambda=0$. It follows that $M^\top \hat\lambda$ is orthogonal to $\ker(M_F)$ (where, with an abuse of notation, we regard $\hat\lambda$ as the vector in $\R^m$ obtained by setting to zero the components outside of $T^F$). It follows that there exists a $\theta\in \R^F$ such that $M^\top \hat\lambda + M_F^\top \theta=0$. Such vector $\theta$ can be computed in time $O( n^3)$ by Gaussian elimination, recalling that $|F|\le 2n$. Regarding $\theta$ as a vector in $\R^m$ (like we did for $\hat\lambda$), we choose $\alpha\ge 0$  such that $\theta_i+\alpha\xi_i> 0$ for all $i\in F$. Let $J=\supp(\hat\lambda)$, and define $F'=F\cup J$ and $\xi'=\hat\lambda+\theta+\alpha\xi$. We have that $M^\top \xi'=0$, $\supp(\xi')=F'$, $\xi'\ge 0$. Furthermore, $\rk(M_{F'})=\rk(M_F)+\rk(M_J)$ and $|J|\le \rk(M_J)+1$ since $\hat \lambda$ is support minimal. By induction, $|F'|\le 2\rk(M_{F'})$, so we can update $F\coloneqq F'$, $\xi\coloneqq\xi'$.

If none of the recursive calls finds a strongly feasible solution, within at most $n$ iterations we reach $m^F_1=0$ and obtain an infeasibility certificate as above.

\paragraph{Running time analysis}

Each call to the algorithm in Lemma~\ref{lem:exact-feas-recurse} needs $O(\mathcal{T}_o(n,\delta_M/(2n)))$
oracle calls and $O(\mathcal{T}_a(n,\delta_M/(2n))+n^2\tau(n)\log\log(1/\delta_M))$ arithmetic operations, and we call this algorithm at most $n$ times. By Lemma~\ref{lem:projected-oracle}, each oracle call for $K^F$ requires $O(n^2)$ arithmetic operations. This gives a total number of operations of $O(n^3\mathcal{T}_o(n,\delta_M/(2n)))+ O(n\mathcal{T}_a(n,\delta_M/(2n))+n^3\tau(n)\log\log(1/\delta_M))$.
Whenever we call such an algorithm, we need to update $F$ and $\xi$, which requires $O(n^3)$ arithmetic operations, as well as $\Pi^F$, in $O(n^\omega)$ arithmetic operations, for a total of $O(n^4)$ operations.

\section{The  conic validity algorithm}\label{sec:conic-validity}

Next, we prove Theorem~\ref{thm:conic-main}\eqref{part:validity} on conic validity. Recall that in the conic validity problem, the input is a cone  $K\subseteq \R^n$ of the form $K=\{x\in \R^n\st Mx\ge 0\}$ with $M\in \R^{m\times n}$, given by a conic separation oracle, and an objective vector $c\in \R^n$, $c\neq 0$.
The goal is to either find $y\in \Rp^m$ with $M^\top y=c$, or an $x\in K$ with $c^\top x<0$.
Here, we assume a polyhedral conic separation oracle is available for
\begin{equation}\label{eq:K-c}
K_c=K\cap \{x\in \R^n\st -c^\top x>0\}\, ,
\end{equation}
and that a subroutine \textsc{Approx-Conic-Dual} for $K_c$ is provided  with running time $\mathcal{T}_o(n,\varepsilon)$ oracle calls and $\mathcal{T}_a(n,\varepsilon)$ arithmetic operations, and returns an $\varepsilon$-approximate conic Farkas certificate comprised of at most $\tau(n)$ oracle separators.
The next lemma formulates the main recursive step, analogously to Lemma~\ref{lem:exact-feas-recurse}. Note that here we use an arbitrary estimate $\hat \delta\in(0,1)$ as opposed to the true value $\delta_M$. Outcome \eqref{outcome:fail} provides a certificate of $\hat\delta>\delta_M$.

\begin{lemma}\label{lem:conic-max-recurse}
Let $K=\{x\in \R^n\st M x\ge 0\}$, $c\in \R^n$, $c\neq 0$,
and let $K_c$ be defined as in \eqref{eq:K-c}. Assume a polyhedral conic separation oracle is given for $K_c$.
Let $\hat \delta\in(0,1)$.
There exists an oracle-polynomial-time algorithm using $\mathcal{T}_o(n,\hat \delta^2/(8n^3))$ oracle calls and $\mathcal{T}_a(n,\hat \delta^2/(8n^3))+O(n^2\tau(n))$ arithmetic operations that returns one of the following: 
\begin{enumerate}[(i)]
\item\label{outcome:pos-c} a vector $x\in K$ with $c^\top x<0$;
\item\label{outcome:m-i} a nonzero vector $\lambda\in \Rp^\km$ that is a support minimal solution to $M^\top \lambda=0$;
\item\label{outcome:solid} a vector $\lambda\in \Rp^{\km}$ such that $\{m_j\st j\in\supp(\lambda)\}$ are linearly independent, along with a nonempty subset $J\subseteq [\km]$ such that for every $j\in J$, $\lambda_j \|m_j\|\hat\delta> {\|M^\top \lambda -c\|}$.
\item\label{outcome:fail} a vector $\lambda\in \Rp^{\km}$ such that $\{m_j\st j\in\supp(\lambda)\}$ are linearly independent and $\|M\lambda\|<\hat \delta \max_{j\in [m]}\lambda_j\|m_j\|$.
\end{enumerate}
\end{lemma}
\begin{proof}
Observe that outcome \eqref{outcome:pos-c} corresponds to finding a point in $K_c$.
Let us  call the subroutine
\textsc{Approx-Conic-Dual} for $K_c$ with $\varepsilon=\hat\delta^2/(8n^3)$; this takes $\mathcal{T}_o(n,\hat \delta^2/(8n^3))$ oracle calls and $\mathcal{T}_a(n,\hat \delta^2/(8n^3))$ operations.
Either we obtain an $x\in K$ with $c^\top x<0$, or an $\varepsilon$-certificate consisting of oracle inequalities. If these only include original separating inequalities $m_i$, then we have
 $\lambda\in \Rp^\km$, such that  $\sum_{j=1}^\km \lambda\|m_j\|_1\ge 1$ and $\|M^\top \lambda\|\le \varepsilon$. As in the proof of Lemma~\ref{lem:exact-feas-recurse}, we can obtain outcomes \eqref{outcome:m-i} or \eqref{outcome:fail} in $O(n^3+n\tau(n)^2)$ time, using Lemma~\ref{lem:conic-Caratheodory}.

Assume next the combination also includes $-c$: we get $\|M^\top \bar \lambda-\tau c\|\le \varepsilon$ for  $(\bar \lambda,\tau)\in \Rp^\km\times \Rp$ with $\sum_{j=1}^\km \lambda\|m_j\|_1+\tau\|c\|_1\ge 1$.
First, assume that $\tau\|c\|\le \hat\delta/(4n^2)$. Then, $\|M^\top \bar\lambda\|\le \varepsilon+\tau\|c\|\le \varepsilon+\hat\delta/(4n^2)<3\hat\delta/(8n^2)$.
At the same time, $\sum_{j=1}^\km \lambda\|m_j\|_1\ge 1-\tau\|c\|_1\ge 1-\sqrt{n}\tau\|c\|>3/4$. For $ \lambda\coloneqq 4\bar\lambda/3$, we have  $\|M^\top \lambda\|\le \hat\delta/(2n)$ and $\sum_{j=1}^\km \bar\lambda\|m_j\|_1\ge 1$.
As in the previous case, we can obtain either of outcomes \eqref{outcome:m-i} and \eqref{outcome:fail}.

For the rest, assume that $\tau\|c\|>\hat\delta/(4n^2)$. Using a Carath\'eodory reduction (Corollary~\ref{cor:carath}), in $O(n^2\tau(n))$ time we can find  a vector $\lambda\in\Rp^\km$ such that $M^\top \lambda=M^\top \bar \lambda/\tau$ and such that $\{m_j\st j\in \supp(\lambda)\}$ are linearly independent. %

We derive outcome \eqref{outcome:solid} by showing
that the set $J\coloneqq\left\{j\in [\km]\st \lambda_j\|m_j\|{\hat\delta}> {\|M^\top \lambda- c\|}\right\}$ is nonempty.
Note that
{\begin{equation}\label{eq:M-lambda-c}
\|M^\top \lambda- c\|=\frac{1}{\tau}\|M^\top \bar\lambda- \tau c\|
\le \frac{4n^2\|c\|}{\hat\delta} \cdot\varepsilon=\frac{\|c\|\hat\delta}{2n}\, .
\end{equation}}
From the triangle inequality, we get
\[
\sum_{j=1}^m\lambda_j	\|m_j\|+\|M^\top \lambda -c\|\ge \|c\|\, .
\]
Combining this with
\eqref{eq:M-lambda-c}, and the assumption $\hat \delta<1$, we obtain
\[
\sum_{j=1}^m\lambda_j \|m_j\|\hat\delta\ge (\| c\|-\|M^\top \lambda-  c\|)\hat\delta\ge (2n-\hat\delta){\|M^\top \lambda- c\|}>n\|M^\top \lambda- c\|\, .
\]
By the linear independence assumption, $|\supp(\lambda)|\le n$ and hence $\arg\max_{j\in [m]}\lambda_j \|m_j\|\in J$.
\end{proof}

\begin{lemma}\label{lem:strong-validity-correct-delta} If outcome \eqref{outcome:solid} of Lemma~\ref{lem:conic-max-recurse} occurs and $\hat \delta\le \delta_M$, then  \[K_c \neq \emptyset \Rightarrow K_c \cap \{x\in\R^n\st m_j^\top x= 0,\, j\in J\} \neq \emptyset .\]
\end{lemma}

\begin{proof}
We first consider the case that $c \notin \spn(M^\top)$. In this case, $\ker(M)\cap \{x\in\R^n\st c^\top x<0\}\neq \emptyset$, and clearly $\ker(M)\cap \{x\in\R^n\st c^\top x<0\}\subseteq K_c \cap \{x\in\R^n\st m_j^\top x= 0,\, j\in J\}$.

We may now examine the case $c \in \spn(M^\top)$, and assume
$K_c \neq \emptyset$.  By the Minkowski--Weyl theorem, $K = \ker(M) + {\rm cone}(v_1,\dots,v_k)$,
where $v_1,\dots,v_k$ are the extreme rays of $K \cap \spn(M^\top)$. Since $c\in \spn(M^\top)$, it follows that $c^\top x=0$ for all $x\in\ker(M)$, hence there must exist  $\ell \in [k]$ such that $c^\top v_\ell < 0$, hence $v_\ell\in K_c$. It suffices to show that $m_j^\top v_\ell = 0$ for all $j \in J$, since then $v_\ell \in K_c \cap \{x \in \R^n: m_j^\top x = 0, j \in J\}$. Since $v_\ell \in K_c$, applying the Cauchy--Schwartz inequality we get that
\begin{equation}
0 \leq -c^\top v_\ell + \sum_{k\in [\km]}\lambda_k m_k^\top v_\ell \le \|M^\top \lambda- c\|\cdot\|v_\ell\|\, . \label{eq:kc-ip-ub}
\end{equation}
Using that all terms on the left hand side are nonnegative, for every $j\in J$ this yields
\[
0\le  m_j^\top v_\ell\le \frac{1}{\lambda_j}\|M^\top \lambda- c\|\cdot\|v_\ell\|< \hat\delta \|m_j\|\cdot \|v_\ell\|\le \delta_M \|m_j\|\cdot \|v_\ell\|\, .
\]
Since $v_\ell$ is an extreme ray of
$K \cap \spn(M^\top)$, Lemma~\ref{lem:cone-angle-delta} gives
$ m_j^\top v_\ell=0$, completing the proof.
\end{proof}

\subsection{The recursive algorithm}\label{sec:recurse-max}
Similarly to  Section~\ref{sec:recurse-feas}, the recursive calls restrict the problem to a subspace $\ker(M_F)$ for  an index set $F\subseteq [\km]$ such that $|F|\le 2\rk(M_F)$.
We use the same notation $\Pi^F$, $v^F$, $T_F$, $M^F$, and $K^F$.
Similarly to Lemma~\ref{lem:projected-oracle}, we can implement the required oracle for $(K^F)_{c^F}$. Assuming $\Pi^F$ is pre-computed, this takes $O(n^2)$ additional arithmetic operations for every oracle call.
We recall from Lemma~\ref{lem:projected-condition-numbers} that $\delta_{M^{F}}\ge\delta_M$. Hence, if we find a certificate  $\hat\delta>\delta_{M^{F}}$ in a recursive call then this also implies that $\hat\delta$ was a wrong estimate on $\delta_M$.

We initialize $F=\emptyset$. If at any iteration we find some $x\in K^F$ with $(c^F)^\top x<0$, then  $\Pi^F  x\in K$ and  $c^\top (\Pi^F  x)=(c^F)^\top x<0$, thus providing a solution in $K_c$.

As in the proof of Lemma~\ref{lem:exact-feas-recurse}, we maintain an estimate $\hat\delta$ of $\delta_M$, updated whenever we detect a failure.
We define a subroutine \textsc{Recursive-Conic-Validity}$(K,c,F,\hat \delta)$, which takes as arguments $F\subseteq [\km]$, $c\in\R^n$, and $\hat\delta\in (0,1)$. The output of this algorithm is one of the following:
\begin{enumerate}[(a)]
\item\label{outcome:farkas-c} A vector $y\in \Rp^{T_F}$ with $|\supp(y)| \leq 2 \rk(M^F)$, such that $(M^F)^\top y=c^F$, certifying that $(c^F)^\top x\ge 0$ for all $x\in  K^F$.
\item\label{outcome:feasible-c}  A point $\bar x\in K$ with $c^\top \bar x<0$.
\item\label{outcome:fail-c} A certificate of $\hat\delta>\delta_{M_F}$, namely, a vector $\lambda\in\R^{T_F}$ with $\{m^F_j\st j\in\supp(\lambda)\}$ linearly independent and
$\varphi=\|(M^F)^\top\lambda\|/(\max_{j\in T^F}\lambda_j\|m^F_j\|)< \hat \delta $.
\end{enumerate}

Before presenting the algorithm, we need two more  lemmas. The key technical challenge is to `pull back' infeasibility certificates in case \eqref{outcome:farkas-c} from the recursive call to the original instance.

\begin{lemma}\label{lem:easy-extension-y} Let $F\subset F'\subseteq[m]$
such that $J\coloneqq F'\sm F$ satisfies $J\subseteq T_F$ and $|J| \le 2\rk(M^F_J)$. Let
$v \in \R^n$ and $y' \in \R^{T_{F'}}$ such that $(M^{F'})^\top y' = v^{F'}$
and $\supp(y') \leq 2\rk(M^{F'})$. Assume the projection matrix $\Pi^F$ is given.  Then, in time $O(n^\omega)$ we can
compute $y \in \R^{T_F}$ such that $(M^{F})^\top y = v^{F}$,
$\supp(y)\subseteq T_{F'}\cup J$, $|\supp(y)| \leq 2\rk(M^F)$
and $y_i=y'_i$ for $i\in T_{F'}$.
\end{lemma}
\begin{proof}
Note that each row $(m^{F'}_i)^\top $ of $M^{F'}$, $i\in T_{F'}$, $m^{F'}_i$ is the orthogonal projection of $m_i^F$ onto $\ker(M^F_J)$, and similarly $v^{F'}$ is the orthogonal projection of $v^F$ onto $\ker(M^F_J)$. It follows that there exists a matrix $Q\in\R^{T_{F'}\times J}$ such that $M^{F'}=M^F_{T_{F'}}+Q M^F_J$ and a vector $\theta\in \R^J $ such that $v^{F'}=v^F+(M^F_J)^\top \theta $. Substituting into the equation $(M^{F'})^\top y' = v^{F'}$ we get $(M^F_{T_{F'}})^\top y' +(M^F_J)^\top (Q^\top y'-\theta)=v^{F}$.
This shows that the system
\begin{equation}\label{eq:easy-extention-y}
(M^F_J)^\top \mu=v^{F}-(M^F_{T_{F'}})^\top y'
\end{equation}
has some solution $\mu\in\R^J$. Note that, given any such $\mu$, the vector $y_i = y'_i, i \in T_{F'}$, $y_i = \mu_i$, $i \in J$, satisfies the requirements of the lemma, observing that $2(\rk(M_J^F)+\rk(M^{F'})) = 2\rk(M^F)$. To compute a solution $\mu$ to \eqref{eq:easy-extention-y}, we must first compute the right hand side $v^F - (M^F_{T_{F'}})^\top y' = \Pi^F v - \Pi^F M_{T_{F'}}^\top y'$, which requires $O(n^2)$ time since $|\supp(y')| \leq 2n$, using that $\Pi^F$ is provided.  From here, the matrix $(M^F_J)^\top = \Pi^F M_J^\top$ requires $O(n^{\omega})$ to compute (noting that $|J| \leq 2n$), and finally solving the system \eqref{eq:easy-extention-y} to obtain  $\mu$ also requires $O(n^\omega)$ time.
\end{proof}

The above lemma captures the easy part of the pullback. However, it may not return a proper certificate, since the coefficients of $m_j$ for $j\in F\setminus F'$ may be negative. The next lemma will allow us to take such a `pre-certificate' and to either turn it to a valid outcome as in case \eqref{outcome:farkas-c} for the original system, or recover a certificate of $\hat\delta>\delta_{M_F}$ as in outcome \eqref{outcome:fail-c}. The proof is deferred to Section~\ref{sec:prox}.

\begin{lemma}\label{lem:combine-vectors}
Let $ H \in \R^{k\times n}$, let $[k]=\L_1 \dot\cup \L_2$, and let $\hat \delta\in(0,1)$. Consider $y,y'\in\R^k$ such that $y\ge 0$, $y'_{\L_2}\ge 0$ and
\[y_i\|h_i\|\hat\delta\ge \|H^\top (y'-y)\|, \qquad \forall i \in \L_1.\]
In time $O(k^3 n)$ we can find one of the following:
\begin{enumerate}[(i)]
\item A nonnegative vector $q \in\Rp^k$ such that $H^\top q=H^\top y'$ and $|\supp(q)|\le \rk(H)$.
\item $\lambda\in\R^k$, such that $\{h_i\st i\in \supp(\lambda)\}$ are linearly independent and $\left\|H^\top  \lambda\right\|<\hat \delta\max_{i\in [k]} |\lambda_i|\cdot \|h_i\| $.
\end{enumerate}
\end{lemma}

The overall algorithm  initializes $\hat\delta\coloneqq1/n$, and calls \textsc{Recursive-Conic-Validity}$(K,c,\emptyset, \hat\delta)$. If outcomes \eqref{outcome:farkas-c} or \eqref{outcome:feasible-c} occur, then we terminate with the desired solution. If outcome \eqref{outcome:fail-c} occurs, then update $\hat\delta\coloneqq\min\{\hat\delta^2,\varphi\}$ and restart.

We now describe $\textsc{Recursive-Conic-Validity}(K,c,F,\hat\delta)$. If $c^{F}=0$, we return the trivial solution $y =0$ to
the system $(M^F)^\top y=c^F$, $y\ge 0$ as outcome (a). If $c^F\neq 0$ but $\rk(M^F)=0$, then we return $\bar x=-c^F\in K^F$ as outcome (b) (since $M\bar x=-M\Pi^F c=0$, while $c^\top \bar x=-c^\top \Pi^F c=-\|c^F\|^2<0$). 
If $c^F \neq 0$ and $\rk(M^F)>0$, we run the subroutine in Lemma~\ref{lem:conic-max-recurse}
for $K^F$, $c^F$ and $\hat \delta$, and perform one of the following
actions according to the outcome:

\begin{enumerate}[(i)]
\item Lemma~\ref{lem:conic-max-recurse} returns $x\in K^F$ with $(c^F)^\top x<0$; we return $\bar x\coloneqq\Pi^F x$ as outcome (b) and the algorithm terminates.
\item Lemma~\ref{lem:conic-max-recurse} returns a nonzero $\lambda\in \Rp^{T_F}$ that is a support minimal solution to $(M^F)^\top \lambda=0$. Let $J\coloneqq\supp(\lambda)$; $\lambda$ provides a proof that $(K^F)_{c^F}\subseteq K^F \subseteq \{x\st M^F_J x=0\}$. We set $F'\coloneqq F\cup J$. Note that $|J|\le \rk(M_J)+1$, by the minimality of $\lambda$, hence $|F'|\le 2\rk(M_{F'})$. We call \textsc{Recursive-Conic-Validity}$(K,c,F',\hat \delta)$.
If this recursive call outputs $\bar x\in K$ with $c^\top \bar x<0$ (outcome \eqref{outcome:feasible-c}), we return $\bar x$. If the recursive call outputs a failure  (outcome \eqref{outcome:fail-c}), that is, a certificate $\lambda'\in\R^{T_{F'}}$ that $\hat\delta>\delta_{M^{F'}}$, we can recover a certificate $\tilde\lambda\in\R^{T_{F}}$ that $\hat\delta>\delta_{M^{F}}$. This can be done similarly to the proof of Lemma~\ref{lem:projected-condition-numbers}, since we can find $\theta\in\R^J$ such that, if we define $\tilde{\lambda}\in\R^{T_F}$ by $\tilde{\lambda}_j=\lambda'_j$ for $j\in T_{F'}$, $\tilde{\lambda}_j=\theta_j$ for $j\in J$, and $\tilde{\lambda}_j=0$ otherwise, we then have $(M^{F'})^\top \lambda' =(M^F)^\top \tilde\lambda$. By construction the rows of $M^F$ indexed by the support of $\tilde\lambda$ are linearly independent, and furthermore
\[\hat \delta \max_{j\in T_{F}} \tilde\lambda_j\|m^{F}_j\|\ge \hat \delta \max_{j\in T_{F'}} \lambda'_j\|m^{F'}_j\|>\|(M^{F'})^\top \lambda'\|=\|(M^F)^\top \tilde\lambda\|,\] 
where the first inequality follows from $\|m_j^F\|\ge\|m^{F'}_j\|$ for $j\in T^{F'}$ and $T_{F'}\subseteq T_F$, while the second follows because $\lambda'$ is a failure for $\hat\delta$ for the matrix  $M^{F'}$.
We then return the certificate $\tilde\lambda$ that $\hat\delta>\delta_{M^{F}}$.

Finally, assume that the recursive call outputs $y' \in \Rp^{T_{F'}}$ such that $(M^{F'})^\top  y' = c^{F'}$ (outcome \eqref{outcome:farkas-c}) and $|\supp(y')| \leq 2\rk(M^{F'})$. By Lemma~\ref{lem:easy-extension-y}, we can compute a vector $y \in \R^{T_F}$ with $\supp(y)\in T_{F'}\cup J$, $|\supp(y)| \leq 2\rk(M^F)$, such that $(M^{F})^\top y =c^{F}$, and $y_i= y'_i$ for $i \in T_{F'}$. However, $y_i<0$ is possible for $i\in J$. Since $(M^F)^\top\lambda=0$ and $\lambda_J>0$, for sufficiently large $\alpha>0$ we obtain $\bar y= y+\alpha\lambda \ge 0$ such that  $(M^F)^\top \bar y=c^F$; we return the vector $\bar y$ as outcome \eqref{outcome:farkas-c}.

 \item  Lemma~\ref{lem:conic-max-recurse} returns $\lambda\in \Rp^{T_F}$ such that $\{m_i^F \st i\in\supp(\lambda)\}$ are linearly independent, along with a nonempty $J\subseteq \supp(\lambda)$ such that $\lambda_j\|m^F_j\|>\|(M^F)^\top \lambda - c^F\|/\hat\delta$ for all $j\in J$. If $(M^F)^\top \lambda=c^F$, we can output this vector $\lambda$ as outcome \eqref{outcome:farkas-c}. For the rest, assume $(M^F)^\top \lambda \neq c^F$. %

     By Lemma~\ref{lem:strong-validity-correct-delta}, if $\hat\delta\le\delta_M\le\delta_{M^F}$, then {$(K^F)_c \neq \emptyset \Rightarrow (K^F)_{c^F} \cap \{x\st M^F_J x=0\} \neq \emptyset$}. Therefore, we set
 $F'=F\cup J$ and call \textsc{Recursive-Conic-Validity}$(K,c,F',\hat\delta)$. We have $|F'|\le 2\rk(M_{F'})$ by the linear independence assumption on $\lambda$.

 Note that, unlike in case $(ii)$ above, we do not have a proof that $M_Jx\ge 0$ can be set at equality; indeed, if $\hat \delta>\delta_{M^F}$, then we may have set at equality an incorrect set of inequalities. As we will now explain, the algorithm will either return a correct solution, or detect a failure, in which case we will restart from $F=\emptyset$ and an updated value of $\hat\delta$.

As in the previous case, if the output of \textsc{Recursive-Conic-Validity}$(K,c,F',\hat\delta)$ is $\bar x\in K$ with $c^\top \bar x<0$ (outcome \eqref{outcome:feasible-c}), we return $\bar x$, whereas if the output is a failure (outcome \eqref{outcome:fail-c}), we return the corresponding $\varphi$ and $\lambda$.

Assume the output is $\bar{y} \in \Rp^{T_{F'}}$ with $|\supp(\bar y)| \leq 2 \rk(M^{F'})$, such that $(M^{F'})^\top \bar  y=c^{F'}$. By Lemma~\ref{lem:easy-extension-y}, we can compute a vector $y' \in \R^{T_F}$ with $\supp(y')\in T_{F'}\cup J$, $|\supp(y')| \leq 2\rk(M^F)$ such that $(M^{F})^\top y'=c^{F}$, and $y'_i=\bar  y_i$ for $i\in T_{F'}$. If $y'\ge 0$, we can return it as outcome \eqref{outcome:farkas-c}. 

In the case that $y'_i<0$ for some $i\in J$ we invoke Lemma~\ref{lem:combine-vectors}.
To remove the negative components of $y'_{J}$, we apply the algorithm in
Lemma~\ref{lem:combine-vectors} with the choice  $H=M^F$,  $y=\lambda$,
$\L_1=J$ and $\L_2=T_F\setminus J$. Observe that $H^\top(y'-y)=c^F-(M^F)^\top\lambda$, hence $y$, $y'$ satisfy the assumptions of
Lemma~\ref{lem:combine-vectors}. If outcome $(ii)$ of
Lemma~\ref{lem:combine-vectors} occurs, then we detected a fail since
$\hat\delta>\delta_{M^F}\ge \delta_M$, and output the corresponding bound
$\varphi$ and combination $\lambda$. Otherwise, outcome $(i)$ of Lemma~\ref{lem:combine-vectors}
occurs: we obtain $q\in \Rp^{T_F}$ with $(M^{F})^\top q=c^F$, $|\supp(q)|
\leq \rk(M^F)$, and we return $q$ as outcome~\eqref{outcome:farkas-c}.

\item Lemma~\ref{lem:conic-max-recurse} returns a failure for $\hat \delta$, in which case we return outcome~\eqref{outcome:fail-c}, along with the corresponding bound $\varphi$ and combination $\lambda$.
\end{enumerate}

\paragraph{Correctness} It is clear that, if the procedure terminates, it
terminates with a correct output, so we only need to argue termination. Note
that, for every value of $\hat\delta$, each call to
\textsc{Recursive-Conic-Validity}$(K,c,\emptyset, \hat\delta)$ will make at
most $\rk(M)\le n$ recursive calls of the form
\textsc{Recursive-Conic-Validity}$(K,c,F, \hat\delta)$. To see this, note
that $\rk(M^F)$ decreases by at least one at every successive
recursive call. Once $\rk(M^F) = 0$, the algorithm terminates with outcome (a) in case $c^F=0$, and with outcome (b) in case $c^F\neq 0$.

Lastly, if $\hat \delta\le\delta_M$, then \textsc{Recursive-Conic-Validity}$(K,c,\emptyset, \hat\delta)$ will not detect any failure, and so it will terminate with one of the two desired outcomes.

\paragraph{Running time analysis }

For each value of $\hat \delta$ set by the algorithm, we have at most $n$ recursive calls to \textsc{Recursive-Conic-Validity}. Recall that each time we update $\hat\delta$ to a value which is less than or equal to $\hat\delta^2$, and we terminate with $\hat \delta\ge \delta_M^2$, for a maximum of $O(\log\log(\delta_M))$ updates. In each recursive call to \textsc{Recursive-Conic-Validity} we  call to the algorithm in Lemma~\ref{lem:conic-max-recurse}, which requires $\mathcal{T}_o(n,\hat \delta^2/(8n^2))$ oracle calls and $\mathcal{T}_a(n,\hat \delta^2/(8n^2))+O(n^2\tau(n))$ arithmetic operations. By Lemma~\ref{lem:projected-oracle}, each oracle call for $K^F$ requires $O(n^2)$ arithmetic operations. By Lemma~\ref{lem:delta-log}, it follows that the total time required by the calls to  Lemma~\ref{lem:conic-max-recurse} is dominated by the time for the last value of $\hat\delta$, hence it requires $O(n\mathcal{T}_o(n,\delta_M^2/O(n)))$ oracle calls and $O(n^3\mathcal{T}_o(n,\delta_M^2/O(n))+n\mathcal{T}_a(n\delta_M^2/O(n)))+O(n^3\tau(n)\log\log(1/\delta_M))$ arithmetic operations.

At each recursive call, we need to compute the projection matrix $\Pi^F$, which requires $O(n^\omega)$ arithmetic operations. In case (ii) of the recursion, the running time is dominated by the application of Lemma~\ref{lem:easy-extension-y}, which requires $O(n^\omega)$ operations (since $|F|\le 2n$). In case (iii) of the recursion, the running time is dominated by the application of Lemma~\ref{lem:combine-vectors}, which requires $O(n^4)$ operations (observe that this is because, when we apply the lemma to $H=M^F$, we can  limit ourselves to the rows of $H$ corresponding to $\supp(y)\cup \supp(y')$, and by construction $|\supp(y)|,|\supp(y')|=O(n)$). Since we have $n$ recursive call per value of $\hat\delta$, and $\hat\delta$ is updated at most $\log\log(\delta_M)$ times, it follows that the running time of all these operations is bounded by $O(n^5\log\log(1/\delta_M))$.

\subsubsection{The proof of Lemma~\ref{lem:combine-vectors}}\label{sec:prox}
The proof of Lemma~\ref{lem:combine-vectors} will use the following Carath\'eodory-type result. We note that this can also be derived from \cite[Lemma 4.3 in the arXiv version]{DadushNV20}; we give here a simpler statement with a more direct proof.

\begin{lemma}\label{lem:constructive-prox}
Let $ H \in \R^{k\times n}$ and $g\in\R^n$. Let $[k]=\L_1 \dot\cup \L_2$ and $\ell\in\R^{\L_2}$, $\ell\le 0$. Given $z\in\R^k$ such that $H^\top z=g$, $z_i\ge \ell_i$ for all $i\in \L_2$, 
in time $O(k^3 n)$ we can find $r,v\in\R^k$, such that
\begin{itemize}
\item $H^\top r=H^\top v=g$, $r_i\ge \ell_i$ for all $i\in \L_2$,
\item $\max_{i\in [k]} |r_i|\cdot \|h_i\|  \leq \max_{i\in [k]} |v_i|\cdot \|h_i\|$,
\item the vectors in $\{h_i\st i\in \supp(v)\}$ are linearly independent.
\end{itemize}
\end{lemma}

\begin{proof}
We set $r\coloneqq z$, $v\coloneqq z$, and at every iteration maintain vectors $r$ and $v$ such that 
\begin{enumerate}[(a)]
\item  $H^\top r=H^\top v=g$, $r_i\ge \ell_i$ for all $i\in \L_2$, 
\item  $|r_{i(r)}|\cdot \|h_{i(r)}\|\le |v_{i(r)}|\cdot \|h_{i(r)}\|$, where $i(r)\coloneqq \arg\max_{i\in [k]} |r_i|\cdot \|h_i\|$, 
\item  $\supp(v)\subseteq\supp(r)$, and
\item $v_i\le 0$ for all $i\in S(r)\coloneqq \{i\in\L_2\st r_i<0\}$. 
\end{enumerate}
If at some iteration the vectors in $\{h_i\st i\in \supp(v)\}$ are  linearly independent, we output $r$ and $v$ and stop. At every iteration in which the vectors in $\{h_i\st i\in \supp(v)\}$ are not linearly independent, we will update $r$ and $v$ so that either $|\supp(r)|$ decreases or $|\supp(v)|$ decreases, while $|\supp(r)|$ does not increase. For this, we compute $\mu\in\ker(H^\top)$, $\mu\neq 0$, with $\supp(\mu)\subseteq \supp(v)$.  We consider two cases:

\noindent{\bf Case I:} If $\mu_i= 0$ for all $i\in S(r)$, we choose $i^*\coloneqq \arg\min\{|r_i|/|\mu_i|\st i\in\supp(\mu)\}$ and let $\alpha\coloneqq -r_{i^*}/\mu_{i^*}$,  $r\coloneqq r+\alpha \mu$, and $v\coloneqq r$. Clearly, $|\supp(r)|$ decreased and $r$ and $v$ maintained the conditions {\em(a)--(d)}. 

\noindent{\bf Case II:} There exists $i\in S(r)$ with $\mu_i\neq 0$. Possibly by taking $-\mu$ instead of $\mu$ we can assume that there exists $i\in S(r)$ such that $\mu_i>0$. If $\mu_i\ge 0$ for all $i\in S(r)$, then we let $\alpha\coloneqq \min\{-r_i/\mu_i\st i\in\supp(\mu)\cap \L_2,\, r_i\mu_i<0\}$ , $r\coloneqq r+\alpha \mu$, and $v\coloneqq r$. Clearly, $|\supp(r)|$ decreased and $r$ and $v$ maintained the conditions {\em (a)--(d)}. The last case to consider is that there exist both $i\in S(r)$ with $\mu_i>0$ and $i\in S(r)$ with $\mu_i<0$.  Possibly by taking $-\mu$ instead of $\mu$, we can assume $\mu_{i(r)} v_{i(r)}\ge 0$ if $v_{i(r)}\ge 0$ and $\mu_{i(r)} v_{i(r)}< 0$ if $v_{i(r)}< 0$. Let  $\alpha\coloneqq \min\{-v_i/\mu_i\st \mu_i >0,\, i\in S(r)\}$, $v\coloneqq v+\alpha\mu$, and $r$ unchanged. Clearly $|\supp(v)|$ decreased, by the choice of $\alpha$ we maintained $v_i\le 0$ for all $i\in S(r)$, and also  $|v_{i(r)}|$ did not decrease. Hence conditions {\em (a)--(d)} are maintained.

Note that the procedure must terminate, since after at most $k$ iterations $|\supp(r)|$ decreases by at least one, because $|\supp(v)|$ decreases by at least one whenever $|\supp(r)|$ does not, and  $|\supp(r)|$ cannot decrease more than $k$ times. Overall we update $r$ at most $k$ times and $v$ at most $k^2$ times. For the running time bound, if we denote by $B$ the column submatrix of $H^\top$ defined by the columns indexed by $\supp(v)$, it takes $O(nk\min\{n,k\})$ arithmetic operations to bring the system $B\mu=0$ in normal echelon form via Gaussian elimination whenever we update $r$ and reset $v\coloneqq  r$. This contributes $O(k^3n)$ overall. At every iteration in which we update $v$ but not $r$, we remove at least one column in the system and it takes $O(nk)$ operations per column removed to bring the system back to echelon normal form,  for a total of $O(nk^3)$ operations.
\end{proof}

\begin{proof}[Proof of Lemma~\ref{lem:combine-vectors}]
Let us define $g:=H^\top(y'-y)$, $z=y'-y$, and $\ell_i=-y_i$ for $i\in \L_2$. Clearly $H^\top z=g$, $z_i\ge \ell_i$ for all $i\in\L_2$. We apply Lemma~\ref{lem:constructive-prox} to obtain $r,v\in\R^k$ as in its statement in time $O(k^3 n)$. If $\|H^\top  v\|<\hat\delta \max_{i\in [k]} | v_i|\cdot \|h_i\|$, then we output $\lambda=v$ and stop with outcome $(ii)$. Assume that $\|H^\top v\|\ge \hat\delta \max_{i\in [k]} |v_i|\cdot \|h_i\|$. Since  $H^\top v=H^\top r$ and $\max_{i\in [k]} |r_i|\cdot \|h_i\|  \leq \max_{i\in [k]} |v_i|\cdot \|h_i\|$, we also have $\|H^\top r\|\ge \hat\delta \max_{i\in [k]} |r_i|\cdot \|h_i\|$.
We claim that the vector $q\coloneqq y+ r$ satisfies $H^\top q=H^\top y'$, $q\ge 0$. Recalling that $H^\top r=H^\top(y'-y)$,  we have $H^\top q=H^\top y'$. For $i\in \L_2$, $q_i= y_i+r_i\ge y_i+\ell_i=0$.  For $i \in \L_1$,
\[q_i=y_i+ r_i\ge y_i-\frac{\left\|H^\top r\right\|}{\hat\delta\|h_i\|}=y_i-\frac{\left\|H^\top(y'-y)\right\|}{\hat\delta\|h_i\|}\ge 0\, ,\]
where the last inequality follows the assumption made in the lemma.

By Corollary~\ref{cor:carath}, in time $O(k^2 n)$ we can turn $q$ into a basic solution of $H^\top q=H^\top y'$, $q\ge 0$, ensuring $|\supp(q)|\le \rk(H)$, where the running time follows. 
\end{proof}

\section{The conic minimum-ratio algorithm}\label{sec:optimization}
Recall the conic minimum-ratio problem: the input is a cone $K\subseteq \R^n$ of the form $K=\{x\in \R^n\st Mx\ge 0\}$ with $M\in \R^{m\times n}$, given via a conic separation oracle, and vectors $c,d\in \R^n$, $y^{(d)}\in \Rp^m$ such that $M^\top y^{(d)}=d$, $I=\supp(y^{(d)})$, $|I|\le n$. The goal is to find the maximum value $\gamma^*\in \R$ such that $(c-\gamma^* d)^\top x\ge 0$ for all $x\in K$, or conclude that the problem is infeasible or unbounded.
We assume polyhedral separation oracles are available for the cones
\[
K_{-d} = \{x \in \R^n: Mx \geq 0, d^\top x > 0\}\quad\mbox{ and }\quad K_I^{=} =
\{x \in \R^n: Mx \geq 0, M_I x =0\}\, ,\]
 Further, we will provide certificates as described in  Section~\ref{sec:our-contr}.
For  an index set $F\subseteq [\km]$,
we will use the same notation $\Pi^F$, $v^F$, $T_F$, $M^F$, and $K^F$ as in
Sections~\ref{sec:conic-feas} and~\ref{sec:conic-validity}.

\paragraph{Detecting infeasibility and unboundedness}
First, let us decide if  problem  \eqref{eq:min-ratio} is infeasible, i.e. if $d^\top x=0$ for all $x\in K$. The combination $y^{(d)}$ already certifies that $d^\top x\ge 0$ for all $x\in K$. Let us run the conic validity algorithm for the cost vector $-d$, using the oracle for $K_{-d}$. Then, we either find
$\bar x\in K$ such that $d^\top \bar x>0$, or a certificate $y\in\Rp^m$ such that $M^\top y=-d$.

Assume next the problem is feasible, i.e., we have found a
$\bar x\in K$ such that $d^\top \bar x>0$.
We now check if  problem \eqref{eq:min-ratio} is unbounded. We need to verify if there exists an $x\in K$ with $c^\top x<0$ and $d^\top x=0$. According to the combination $y^{(d)}\in \Rp^m$, $M^\top y^{(d)}=d$, $I=\supp(y^{(d)})$, we have $d^\top x=0$ if and only if $x\in K_I^=$.
We thus run the conic validity algorithm for $K_I^=$ and the projection $c^I$. (The required oracle for $(K_I^=)_{c^I}$ can be implemented using the oracle assumed for $K_I^=$ and checking the additional inequality $(c^I)^\top x < 0$.)
 If the original problem is unbounded, then we find an $x\in K^=_I$ with $(c^F)^\top x<0$. If we establish that $c^\top x\ge 0$ for all $x\in K_I^=$ %
 then we conclude that \eqref{eq:min-ratio} is bounded.

\paragraph{The optimal set and the main subroutine}
After the above preprocessing, from now on we assume that
\begin{equation}\label{bounded-assumption}
K_{-d}\neq\emptyset\mbox{ and  $c^\top x\ge 0$ for all $x\in K_I^=$. }
 \end{equation}
Hence, a finite optimal $\gamma^*$ exists, in which case $(c-\gamma^* d)^\top x\ge 0$ is valid for $K$ and there exists an $x^*\in K$ with $(c-\gamma^* d)^\top x^*=0$, $d^\top x^*>0$.
Let
\[
K_0=\left\{x\st Mx\ge 0\, ,\, (c-\gamma^* d)^\top x=0\, ,\,  d^\top x>0\right\}
\]
denote the set of optimal solutions. By the above assumptions, $K_0\neq\emptyset$. Note that if $c$ and $d$ are linearly dependent, that is, $c=\gamma d$ for some $\gamma\in\R$, then $c^\top x/d^\top x=\gamma$ for every $x\in K_{-d}$. Hence, $K_0=K_{-d}$ in this case.

The purpose of the next lemma, which is an analogue of
Lemma~\ref{lem:conic-max-recurse} for the set $K_0$, is to identify inequalities of $Mx\ge 0$ that can be set at equality for $K_0$.

\begin{lemma}\label{lem:conic-max-ratio-recurse}
Let $K=\{x\in \R^n\st M x\ge 0\}$ for $M\in\R^{\km\times n}$, and let $c,d\in \R^n$ such that
\eqref{bounded-assumption} holds, and $c$ and $d$
 are linearly independent.
 Furthermore, let $y^{(d)}\in\Rp^\km$ such that $M^\top y^{(d)}=d$. Assume that polyhedral separation oracles are available for the cones $K_{-d}$ and $K^=_{I}$.
Let $\hat \delta\in(0,1)$.
There exists an oracle-polynomial-time algorithm using $\mathcal{T}_o(n,\hat \delta^2/(8n^3))$ oracle calls and $\mathcal{T}_a(n,\hat \delta^2/(8n^3))+O(n^2\tau(n))$ operations that returns one of the following:
\begin{enumerate}[(i)]
\item\label{outcome-max:m-i} a nonzero vector $\lambda\in \Rp^\km$ that is a minimum support solution to $M^\top \lambda=0$;
\item\label{outcome-max:solid} a vector $\lambda\in \Rp^{\km}$ such that $\{m_j\st j\in\supp(\lambda)\}$ are linearly independent, a point $\bar x\in K$ such that $d^\top \bar x>0$, along with a nonempty subset $J\subseteq [\km]$ such that, for $\gamma=c^\top \bar x/d^\top \bar x$,  every $j\in J$ satisfies $\lambda_j\|m_j\|{\hat\delta}> {\|M^\top \lambda -(c-\gamma d)\|}$;
\item\label{outcome-max:fail} a certificate of $\hat \delta>\delta_M$, namely, a vector $\lambda\in \Rp^{\km}$ such that $\{m_j\st j\in\supp(\lambda)\}$ are linearly independent and $\varphi=\|M\lambda\|/\max_{j\in [m]}\lambda_j\|m_j\|<\hat \delta $.
\end{enumerate}
\end{lemma}

\begin{proof}
Let $\hat K$ be the cone implicitly defined by the following conic separation oracle, based
 on the one provided for $K_{-d}$. For every $\bar x\in \R^n$, the oracle returns one of the following:
\begin{enumerate}[(i)]
	\item If $\bar x\notin K_{-d}$, then return a violated inequality $m_i^\top x\ge 0$ or $d^\top x>0$.
	\item\label{case:d-pos} If $\bar x\in K_{-d}$, then set $\bar \gamma\coloneqq c^\top \bar x/d^\top \bar x$,
    and return the inequality $(c-\bar\gamma d)^\top x< 0$.
\end{enumerate}

Observe that, by definition of the oracle, $\hat K=\emptyset$, since every point in $\R^n$ is separated. %
By the assumption that $c$ and $d$ are linearly independent, $c-\bar\gamma d\neq 0$.
For every inequality returned in case~\eqref{case:d-pos}, we have $\bar\gamma \ge \gamma^*$, since $(c-\bar\gamma d)^\top \bar x=0\le (c-\gamma^* d)^\top \bar x$, and $d^\top \bar x>0$.

We call the subroutine \textsc{Approx-Conic-Dual}  with $\varepsilon=\hat\delta^2/(8n^3)$ using the above separation oracle. Since $\hat K=\emptyset$, \textsc{Approx-Conic-Dual} always terminates with an $\varepsilon$-certificate consisting of inequalities returned by the oracle.
The combination may include vectors corresponding to the inequalities of $Mx\ge 0$, $d^\top x\ge 0$, and inequalities $-(c-\gamma_t d)^\top x\ge 0$, $t\in S$, for different values of $\gamma_t$. The certificate is a vector of the form $(\bar \lambda,\bar \tau, \bar \mu)\in\Rp^\km\times \Rp^S \times \Rp\ $ such that
\[
\left\|M^\top \bar \lambda-\sum_{t\in S}\bar \tau_t (c-\gamma_t d)+\bar\mu d\right\|\le \varepsilon\, ,\quad  \sum_{j=1}^\km \bar \lambda_j\|m_j\|_1+\sum_{t\in S}\bar \tau_t\|c-\gamma_t d\|_1+\bar \mu \|d\|_1\ge 1\, .
\]
Recall that $d=M^\top y^{(d)}$ for a given vector $y^{(d)}\in\Rp^k$.
Let us define $(\lambda,\tau)\in\Rp^\km\times\Rp$ as follows.
 If $S=\emptyset$, then we let $\lambda\coloneqq \bar\lambda+\bar\mu y^{(d)}$ and $\tau\coloneqq 0$.
If $S\neq\emptyset$, let $\gamma\coloneqq \min_{t\in S}\gamma_t$, and
define
\[
\lambda\coloneqq \bar\lambda+\left(\bar \mu+\sum_{t\in S}\bar\tau_t (\gamma_t-\gamma)\right) y^{(d)}\, ,\quad \mbox{and} \quad \tau\coloneqq \sum_{t\in S} \bar\tau_t\, .
\]
In both cases, $M^\top  \lambda- \tau (c-\gamma d) =M^\top \bar \lambda-\sum_{t\in S}\bar \tau_t (c-\gamma_t d)+\bar{\mu} d$, and consequently, $\|M^\top  \lambda- \tau (c-\gamma d)\|\le \varepsilon$. Furthermore,
\[
\begin{aligned}
&\sum_{j=1}^\km  \lambda_j\|m_j\|_1+\tau\|c-\gamma d\|_1\\
=&  \sum_{j=1}^\km  \bar \lambda_j\|m_j\|_1+\sum_{t\in S} \bar\tau_t\left(\|c-\gamma d\|_1+(\gamma_t-\gamma)\sum_{j=1}^{\km}  y^{(d)}_j\|m_j\|_1\right)+\bar\mu \sum_{j=1}^\km y^{(d)}_j\|m_j\|_1\\
\ge& \sum_{j=1}^\km \bar \lambda_j\|m_j\|_1+\sum_{t\in S}\bar \tau_t\|c-\gamma_t d\|_1+\bar \mu \|d\|_1\ge 1
\end{aligned}
\]
by the triangle inequality. Note that at a certain iteration of the algorithm, the inequality $-(c-\gamma d)^\top x\ge 0$ was returned to separate a point $\bar x\in K$ such that $d^\top \bar x>0$ and $\gamma=c^\top \bar x/(d^\top \bar x)$.
The rest of the proof is now identical to the proof of Lemma~\ref{lem:conic-max-recurse}, applied for the vector $c-\gamma d$ in place of $c$.
\end{proof}

\begin{lemma}\label{lem:max-correct-delta} In Lemma~\ref{lem:conic-max-ratio-recurse}, if outcome \eqref{outcome-max:solid} occurs and $\hat \delta\le \delta_M$, then  \[\emptyset\neq K_0\subseteq \{x\in\R^n\st m_j^\top x= 0,\, j\in J\}.\]
\end{lemma}
\begin{proof}
Let $\bar x\in K$, $\gamma$, $J$ and $\lambda\in\Rp^\km$ as in Lemma~\ref{lem:conic-max-ratio-recurse}\eqref{outcome-max:solid}. Note that, by definition, $\gamma\ge \gamma^*$.
Recall that $K_0\neq\emptyset$ by assumption \eqref{bounded-assumption}.
Let $K'$ be the face of $K$ defined by the valid inequality $(c-\gamma^*
d)^\top x\ge 0$, so that $K_0=\{x\in K'\st d^\top x>0\}$. We will show that
$K' \subseteq \{x \in \R^n \st m_j^\top x =0,\, j \in J\}$, which clearly
suffices. By the Minkowski--Weyl theorem, $K' = \ker(M) + {\rm cone}(v_1,\dots,v_q)$,
where $v_1,\dots,v_q$ are the extreme rays of $K' \cap \spn(M^\top)$. To
verify the desired containment, it suffices to show that $m_j^\top
v_\ell = 0$ for all  $j \in J$ and $i \in [q]$. Since $K' \cap \spn(M^\top)$ is a face
of $K \cap \spn(M^\top)$, we note that $v_1,\dots,v_q$ are also extreme rays
of $K \cap \spn(M^\top)$. Therefore, by Lemma~\ref{lem:cone-angle-delta}, for $\ell \in [q]$ and
$j \in J$, either $m_j^\top v_\ell =0$ or $m_j^\top v_\ell \geq
\delta_M \|m_j\| \|v_\ell\|$.

On the other hand, for any $z\in K'$, by the Cauchy--Schwartz  inequality we get that
\[
 -(c-\gamma d)^\top z+ \sum_{k\in [\km]}\lambda_k m_k^\top z\le \|M^\top \lambda- (c-\gamma d)\|\cdot\|z\|\, .
\]
Note that all terms on the left-hand-side are nonnegative, since $(c-\gamma d)^\top z\le (c-\gamma^* d)^\top z=0$ (because $d^\top z>0$ and $\gamma\ge \gamma^*$) and  $Mz\ge 0$. 
It follows that, for every $\ell\in [q]$ and $j\in J$. we get
\[
0\le \lambda_j m_j^\top v_\ell\le \|M^\top \lambda- (c-\gamma d)\|\cdot\|v_\ell\|<\hat\delta \lambda_j \|m_j\|\cdot\|v_\ell\|\le \delta_M \lambda_j \|m_j\|\cdot\|v_\ell\|\, ,
\]
implying  $m_j^\top v_\ell =0$ as above.
\end{proof}

\paragraph{The recursive algorithm}
The algorithm follows the same lines as those in Section~\ref{sec:recurse-feas} for strong conic feasibility and in Section~\ref{sec:recurse-max} for conic validity. We assume that \eqref{bounded-assumption} holds.
For $F\subseteq [\km]$ we use the same notation $M^F$, $\Pi^F$, $T_F$, $K^F$ and $c^F$; further, let  $d^F=\Pi^F d$.

As in the previous algorithms, we will maintain a set of inequalities $F\subseteq [\km]$, $|F|\le 2\rk(M_F)$ and an estimate $\hat \delta\in(0,1)$ of $\delta_M$. We will defined an  algorithm \textsc{Recursive-Min-Ratio}$(K,c,d,F,\hat \delta)$ which outputs either of the following:
\begin{enumerate}[(a)]
\item \label{outcome-max:optimum} A point $x^*\in K$ such that  $d^\top x^*>0$  along with the value $\gamma^*\coloneqq c^\top x^*/d^\top x^*$ and a certificate
$y\in \Rp^{T^F}$ with $|\supp(y)|=O(n)$ such that $(M^F)^\top y=c^F-\gamma^*  d^F$.
\item \label{outcome-max:fail-ratio} Claim a failure for $\hat \delta$.
\end{enumerate}

Our algorithm is the following:
 initialize $\hat\delta\coloneqq 1/n$ and call \textsc{Recursive-Min-Ratio}$(K,c,d,\emptyset,\hat\delta)$. If outcome \eqref{outcome-max:optimum} occurs, then $\gamma^*$ is the optimal value, $x^*$ and optimal solution, and $y\in\R^\km$ a dual certificate of optimality. %
If outcome \eqref{outcome-max:fail-ratio} occurs, then update $\hat\delta\coloneqq \hat\delta^2$ and repeat. (For the sake of simplicity, in the description we give below we do not compute an explicit combination of linearly independent rows of $M$ certifying a failure, as we had done in the previous sections. Hence, we do not have an explicit upper bound $\delta_M\le\varphi<\hat\delta$. This does not make a difference in term of the worst-case running time. See also the discussion in Remark~\ref{rem:could-certify} on how explicit certificates could be constructed.)

We now describe \textsc{Recursive-Min-Ratio}$(K,c,d,F,\hat\delta)$. If $d^F=0$, then  we declare a failure, i.e., outcome~\eqref{outcome-max:fail-ratio}. This is because, by \eqref{bounded-assumption}, $K_0\neq 0$; however, for any  $x\in K_0$, we have $d^\top x>0$ whereas 
$(d^F)^\top x=0$, implying that $x\notin \ker(M)$. By Lemma~\ref{lem:max-correct-delta}, it follows that $\hat \delta> \delta_M$.

If $d^F\neq 0$, but $c^F=\gamma d^F$ for some $\gamma\in\R$, then we run the conic validity problem for $K^F$ and $-d^F$ (Section~\ref{sec:conic-validity}). If this returns a feasible solution $\bar x$, then we return outcome \eqref{outcome-max:optimum} with $\bar x$, $\gamma$, and the trivial dual certificate $y=0$. If this returns infeasibility, then we again declare the failure outcome~\eqref{outcome-max:fail-ratio} as above.

Finally, if $d^F\neq 0$ and $c^F$ and $d^F$ are linearly independent, then
  we call the subroutine in
Lemma~\ref{lem:conic-max-ratio-recurse} for the lower dimensional problem
$K^F$, $c^F$, and $d^F$, and consider the possible outcomes:
\begin{enumerate}[(i)]
\item From Lemma~\ref{lem:conic-max-ratio-recurse} we obtain a nonzero $\lambda\in \Rp^{T_F}$ that is a minimal support solution to $(M^F)^\top \lambda=0$. Let $J=\supp(\lambda)$ and $F'=F\cup J$, and call \textsc{Recursive-Min-Ratio}$(K,c,d,F',\hat\delta)$.
    If this  call returns  $x^*\in K$ such that $d^\top x^*>0$, $\gamma^*=c^\top x^*/d^\top x^*$, and  $y\in \Rp^{T_{F'}}$ such that $(M^{F'})^\top y= c^{F'}-\gamma^* d^{F'}$ (outcome~\eqref{outcome-max:optimum}), then we compute $y'\in\R^{T_F}$ as in Lemma~\ref{lem:easy-extension-y}, and return $x^*$ and $\bar y\coloneqq y'+\alpha\lambda\ge 0$, for some $\alpha$ large enough, satisfying
    $(M^{F})^\top \bar y=c^F-\gamma^* d^{F}$.
     If the recursive call  returns a failure (outcome \eqref{outcome-max:fail-ratio}), then we return a failure.

 \item  From Lemma~\ref{lem:conic-max-ratio-recurse} we obtain  $\lambda\in \Rp^{T_F}$ such that $\{m_j\st j\in\supp(\lambda)\}$ are linearly independent, $\bar x\in K^F$ with $(d^F)^\top \bar x>0$, $\gamma=(c^F)^\top\bar x/(d^F)^\top \bar x$, and a nonempty $J\subseteq \supp(\lambda)$ such that  $\lambda_j\|m^F_j\|>\|(M^F)^\top \lambda - (c^F-\gamma d^F)\|/\hat\delta$ for all $j\in J$.  If $(M^F)^\top \lambda=c^F-\gamma d^F$, then return outcome \eqref{outcome-max:optimum}, along $x^*=\Pi^F\bar x\in K$, $\gamma^*=\gamma$, and $y=\lambda$.
     Otherwise, we set $F'=F\cup J$, and call \textsc{Recursive-Min-Ratio}$(K,c,d,F',\hat\delta)$. Indeed, recall that, if $\hat\delta\le\delta_M$, then by Lemma~\ref{lem:max-correct-delta} $\emptyset\neq K_0\subset \ker(M_{F'})$.
     It this call returns a failure (outcome \eqref{outcome-max:fail-ratio}), then we return a failure.

     Consider the case where \textsc{Recursive-Min-Ratio}$(K,c,d,F',\hat\delta)$ returns  $x^*\in K$ such that $d^\top x^*>0$, $\gamma^*=-c^\top x^*/d^\top x^*$, and $\bar y\in \Rp^{T_{F'}}$ with $|\supp(\bar y)|=O(n)$ such that $(M^{F'})^\top \bar y= c^{F'}-\gamma^* d^{F'}$ (outcome~\eqref{outcome-max:optimum}). If $\gamma^*>\gamma$, then we return a failure, because $\gamma=(c^F)^\top\bar x/(d^F)^\top \bar x$ and the optimal value $\gamma^*$ should satisfy $(c^F)^\top\bar x/(d^F)^\top \bar x\ge \gamma^*$.
     
     Assume $\gamma^*\le \gamma$. By Lemma~\ref{lem:easy-extension-y}, we can compute $y'\in \Rp^{T_F}$ such that
$(M^{F})^\top y'= c^{F}-\gamma^* d^{F}$;
      $\supp(y')\subseteq T_{F'}\cup J$ and $y'_i=\bar y_i$ for all $i\in T_{F'}$, where possibly $y_i<0$ for some $i\in J$.

      Recall that the input included a vector $y^{(d)}\in \Rp^k$ such that $M^\top y^{(d)}=d$; observe that the restriction $\tilde y=y^{(d)}_{T^F}$ satisfies $( M^F)^\top \tilde y=d^F$. Define $y=\lambda+(\gamma-\gamma^*) \tilde y$. Observe that, by construction, $(M^F)^\top y - (c^F-\gamma^* d^F)=(M^F)^\top \lambda - (c^F-\gamma d^F)$ and $y\ge \lambda$. It follows that we can apply Lemma~\ref{lem:combine-vectors} to $H=M^F$, $y$ and $y'$. The outcome will be either a failure for $\hat\delta$, in which case we return a failure, or a vector $q\in\Rp^{T_F}$ such that $(M^F)^\top q=c^F-\gamma^* d^F$. We output $x^*$, $\gamma^*$, and the certificate $q\in\Rp^{T^F}$.

 \item  From Lemma~\ref{lem:conic-max-ratio-recurse} we obtain a failure for $\hat\delta$, in which case we return a failure.
\end{enumerate}

\paragraph{Correctness.} It is clear  that the procedure may only ever terminate with a correct output, so we only need to argue termination. Note that, for every value of $\hat\delta$, each call to \textsc{Recursive-Min-Ratio}$(K,c,d,\emptyset,\hat\delta)$ will make at most $\rk(M)\le n$ recursive calls of the form  \textsc{Recursive-Min-Ratio}$(K,c,d,F,\hat\delta)$. Furthermore, if $\hat \delta\le\delta_M$, then \textsc{Recursive-Min-Ratio}$(K,c,d,\emptyset,\hat\delta)$ will not detect any failure, and so it will terminate with one of the two desired outcomes.%

\paragraph{Running time analysis} The running time analysis is identical to the one in Section~\ref{sec:recurse-max}, so we omit it. %
\begin{remark}\label{rem:could-certify}\em
The above recursive algorithm can be modified to explicitly compute a ``certified failure'' for the current $\hat\delta$, instead of simply declaring that $\hat \delta>\delta_M$. There are three places in the above recursive procedure where a failure is detected. One is when $d^F=0$, the second when $d^F$ and $c^F$ are linearly independent, and the third when $\gamma<\gamma^*$ in outcome (ii). In the first two cases, we proceed as in the conic-validity algorithm to repeatedly apply Lemma~\ref{lem:combine-vectors} to pull-back certificates of the form $(M^F)^\top \lambda=d^F$, $\lambda\ge 0$. Since we cannot recover a certificate $M^\top \lambda=d$, $\lambda\ge 0$, at some point Lemma~\ref{lem:combine-vectors} must compute a failure. If $\gamma<\gamma^*$ in outcome (ii), then we proceed as in the above algorithm to repeatedly apply Lemma~\ref{lem:combine-vectors} to pull-back the certificate $(M^{F'})^\top \bar y= c^{F'}-\gamma^* d^{F'}$ to a certificate $M^\top y=c-\gamma^* d$, $y\ge 0$. Since such a certificate does not exist, because $\gamma^*$ is not optimal, it follows that  at some point Lemma~\ref{lem:combine-vectors} must compute a failure.
\end{remark}

\section{Computing approximate dual certificates}
\label{sec:dual certificates}
Our goal in this section is to exhibit a general technique for implementing the \textsc{Approx-Conic-Dual} oracle, which can be adapted to various cutting-planes methods in the literature. We start by defining a more general notion of approximate Farkas certificates applicable also for the non-conic setting. First, we note that a simple application of the KKT conditions  provides the following variant of Farkas' lemma. For completeness, we provide a proof in Appendix~\ref{sec:impossibility}.
\begin{restatable}{lemma}{farkassphere}\label{lem:farkas-sphere}
Given $A\in\R^{m\times n}$, $u\in \R^m$, $r>0$, the system $Ax\le u$ has no solution in $\bb{B}^n(r)$ if and
only if there exists $\lambda\in\Rp^m$ such that
\begin{equation}\label{eq:farkas-sphere}
r\|A^\top \lambda\|<-\lambda^\top u\, .
\end{equation}
Further, let $v\in\R^n$ and $\nu\in\R$. If there exists $\lambda\in\Rp^m$ such that
\begin{equation}\label{eq:optimality}
r\|A^\top \lambda+v\|+\nu\le-\lambda^\top u,
\end{equation}
then the inequality $v^\top x\ge
\nu$ is valid for  $\{x\in\R^n\st Ax\le u\}\cap \bb{B}^n(r)$. Conversely, if $v^\top x\ge
\nu$ is valid for  $\{x\in\R^n\st Ax\le u\}\cap \bb{B}^n(r)$ and $Ax\le u$, $\|x\|< r$ has a solution, then  there exists $\lambda\in\Rp^m$ satisfying \eqref{eq:optimality}.
\end{restatable}
In this paper, we will only need the ``easy'' direction of the statement, that is, the existence of $\lambda$ as in \eqref{eq:farkas-sphere} and \eqref{eq:optimality} imply infeasibility and validity. Nontheless, for completeness, we provide a proof in Appendix~\ref{sec:impossibility}. Lemma~\ref{lem:farkas-sphere} motivates the following approximate certificate concept.
\begin{definition}
Given a convex set $K\subseteq \R^n$ and $r,\varepsilon>0$,  an \emph{$\varepsilon$-approximate Farkas certificate} for $K\cap \bb{B}^n(r)$ is given by a system $Ax\le u$ of valid inequalities for $K$, $A\in\R^{m\times n}$, $u\in \R^m$, and multipliers $\lambda\in \Rpp^m$
\[
\displaystyle{\lambda^\top u+r\left\|A^\top \lambda\right\|<\varepsilon},\qquad \sum_{i=1}^m \lambda_i \|a_i\|_1\ge 1\, .\]
\end{definition}
For some intuition behind this definition, note that any point $x \in K\cap \bb{B}^n(r)$ satisfies $0\le \lambda^\top u-\lambda^\top A x\le \lambda^\top u+\|A^\top \lambda\| \|x\|\le \lambda^\top u+r\left\|A^\top \lambda\right\|<\varepsilon$, implying $\lambda^\top u-\varepsilon\le\lambda^\top A x\le \lambda^\top u$, whereas the second condition ensures that $\lambda$ is sufficiently large, implying that $K\cap \bb{B}^n(r)$ is thin in the direction orthogonal to the vector $A^\top \lambda$.

If $K\subseteq\R^n$ is a cone given by a conic separation oracle, then we can assume $u_i=0$ for all oracle inequalities $\separ_i^\top x\le 0$. In particular, in the case that $K$ is a cone, setting $r=1$, an \emph{$\varepsilon$-approximate Farkas certificate} for $K\cap \bb{B}^n(1)$ using oracle inequalities coincides with the notion of an
 \emph{$\varepsilon$-approximate conic Farkas certificate} for $K$ as required in \textsc{Approx-Conic-Dual}.

Currently, the fastest algorithm to compute an approximate Farkas certificate is based on the algorithm of Jiang et al.~\cite{Jiang2020} (JLSW algorithm). Their algorithm is based on Vaidya's algorithm~\cite{vaidya96}. We will discuss in Section~\ref{sec:Vaidya} how Vaidya's algorithm and the JLSW algorithm can be modified to provide approximate Farkas certificates. The theorem below immediately implies Theorem~\ref{thm:approx-conic-oracle}. 

\begin{theorem}\label{thm:JLSW}
Let $K$ be a  convex set given by a strong
separation oracle, $r>0$, and $\varepsilon\in (0,2 r)$.  Then, in expected $O(n\log(nr/\varepsilon))$ calls to the separation oracle, and expected $O(n^3\log(nr/\varepsilon))$ arithmetic operations, the JLSW algorithm either returns a point $x\in K$, or  an $\varepsilon$-approximate Farkas certificate for $K\cap\bb{B}^n(r)$ comprising only oracle inequalities.
\end{theorem}

\medskip

The rest of this section is dedicated to showing that $\varepsilon$-Farkas certificates can be recovered from a broad class of algorithms, including seemingly `primal-only' methods such as the ellipsoid method.
For any $R\in \Snpp$ and $p\in \R^n$, we define the ellipsoid
\[
E(R,p) \eqdef\{z\in \R^d\st \|z-p\|_R\le 1\}.
\]
Given a compact set $K\subseteq \R^n$ and a vector $v\in \R^n$, we define the {\em width of $K$ along $v$} as
\begin{equation}\label{eq:set width}\width_K(v)\eqdef \max  \{v^\top
  z\st z\in K\}- \min \{v^\top x\st z\in K\}.\end{equation}
We say that $v$ is an \emph{$\varepsilon$-thin direction for $K$} if
$\width_K(v)\le \varepsilon$.
The width of an ellipsoid can be characterized as follows.
Recall that for every $v\in \R^n$, $\min\{v^\top x\st x\in E(R,p)\} =v^\top p-\|v\|_{R^{-1}}$, achieved by $x^*=p-R^{-1}v/\|v\|_{R^{-1}}$.

\begin{lemma}\label{lemma:ellipsoid width}
Given $R\in\Snpp$, and $p\in \R^n$, let $E\coloneqq E(R,p)$. For any  $v\in\R^d$,  $\width_E(v)=2\|v\|_{R^{-1}}$. In particular, for $K=E(R,p)$, $v$ is an $\varepsilon$-thin direction if and only if  $\|v\|_{R^{-1}}\le\varepsilon/2$.
\end{lemma}

Consider the feasibility problem for a polyhedron $P\subseteq \R^n$ given by a strong separation oracle. The algorithms discussed in this section---the ellipsoid method, Vaidya's cutting plane methods~\cite{vaidya96}, as well as the geometric rescaling algorithms~\cite{DVZ,rothvoss}, proceed by maintaining a containing ellipsoid $E(R,p)\supseteq P\cap \bb{B}^n(r)$. In every iteration, they either terminate with a feasible solution, or modify the containing ellipsoid; the main progress measure is the decrease in the volume of the ellipsoid.

In particular, all these methods maintain the matrix $R$ in the form
$R=\gamma_0 I_n +\sum_{i=1}^m \gamma_i \separ_i \separ_i^\top$ for coefficients $\gamma\in\R^{m+1}$ and vectors $\separ_i$ returned by the oracle calls. The next lemma shows that if the volume of $E(R,p)$ is small, or equivalently the determinant of $R$ is large, then $P$ must be thin in one of the directions $a_i$ returned by the oracle. We include the simple proof for completeness.

\begin{lemma}[{\cite[Lemma 4.11]{DVZ}}]\label{lem:c-Q-bound}
Let $R\in\Snpp$ be defined by
\[R=\gamma_0 I_n +\sum_{i=1}^m \gamma_i \separ_i \separ_i^\top,\]
where $a_1,\ldots,a_m\in\R^n$, and $\gamma_0,\ldots,\gamma_m\geq 0$, with $\gamma_0\le 1$.
Then,
for every $i\in [m]$, $\gamma_i \|\separ_i\|^2_{R^{-1}}< 1$ holds, and $\sum_{i=1}^t \gamma_i \|\separ_i\|^2_{R^{-1}}\le n$.
Further, if $\det(R)>1$, then there exists $k\in [m]$ such that
\[\|\separ_k\|_{R^{-1}}\le \frac{\|a_k\|_2}{\sqrt{\det(R)^{1/n}-1}}\, .\]
\end{lemma}
\begin{proof}
Let $Q=R^{-1}$. The bound $\gamma_i \|\separ_i\|^2_{Q}< 1$ follows by
\[
\|\separ_i\|^2_{Q}=\separ_i^\top Q R Q \separ_i=\separ_i^\top Q\left(\gamma_0 I_n+\sum_{j=1}^m \gamma_{j}\separ_j\separ_j^\top \right)Q\separ_i> \gamma_i \|\separ_i\|^4_{Q}\, .
\]
For $\sum_{i=1}^m \gamma_i \|\separ_i\|^2_{Q}<n$, we see that
\begin{align}\label{eq:a_k bound}
\sum_{i=1}^m \gamma_i \|\separ_i\|^2_Q&=\sum_{i=1}^m\gamma_i
\left(\separ_i^\T Q \separ_i\right)
= \tr\left(Q \sum_{i=1}^m\gamma_i
   \separ_i\separ_i^\T\right)\\
&= \tr(Q (R-\gamma_0 I_n)) =\tr(I_n)-\gamma_0\tr(Q)< n\, . \nonumber
\end{align}
In the final inequality we used that $\tr(Q)>0$, since $Q$ is positive definite.

For the third claim, we see that $\tr(R)=\gamma_0 n+\sum_{i=1}^m \gamma_i\|\separ_i\|^2_2$.  Noting that $\gamma_0\le 1$,
$\sum_{i=1}^m\gamma_i\|\separ_i\|^2_2\geq \tr(R)-n\ge n(\det(R)^{1/n}-1)$, using the
well-known inequality
$\det(R)^{1/n}\le \tr(R)/n$ for positive semidefinite matrices.
Let $k=\arg\min_{i\in [m]}(\|\separ_i\|_Q/\|a_i\|_2)$. Using the bound $\sum_{i=1}^m \gamma_i \|\separ_i\|^2_{Q}<n$, we see that
\[
\frac{\|\separ_k\|^2_Q}{\|a_k\|^2_2} \le \frac{\sum_{i=1}^m\gamma_i \|a_i\|^2_Q}{\sum_{i=1}^m\gamma_i \|a_i\|^2_2}<
\frac{n}{\sum_{i=1}^m\gamma_i \|a_i\|^2_2}<
\frac{1}{\det(R)^{1/n}-1}\, .
\]
\end{proof}

Thus, we can identify a thin direction $a_k$. Our goal in this section is to provide a dual certificate of thinness of $P\cap \B^n(r)$, using the oracle inequalities $a_i^\top x\le u_i$ and the initial ball constraint $\|x\|_2\le r$. For a conic set $P$, this will imply $\varepsilon$-approximate conic Farkas certificate from the algorithms mentioned.

\medskip

Our certification scheme builds on the work of
 Burrell and Todd \cite{burrell-todd} on the ellipsoid method. The key idea is to use an alternative representation of the strictly concave quadratic form $q(x)=-(x-p)^\top R(x-p)$ corresponding to the ellipsoid $E(R,p)$.
In Section~\ref{sec:quadratic-certify}, we introduce \emph{certified concave quadratic forms}, and show how this representation can be used to construct dual certificates for valid inequalities. These ingredients are combined to derive the Farkas certificate in Section~\ref{sec:thin}.
The remaining three subsections demonstrate the use of certified concave quadratic forms for the ellipsoid method (Section~\ref{sec:ellipsoid}),   for geometric rescaling methods (Section~\ref{sec:rescale}), and for volumetric cutting plane methods (Section~\ref{sec:Vaidya}).

\subsection{Dual certificates from certified quadratic forms}\label{sec:quadratic-certify}

Consider a polyhedron $P=\{x\in\R^n\st Ax\le u\}$, where $A\in\R^{n\times m}$ and $u\in \R^m$. Let $a_i$, $i\in[m]$ be the rows of $A$. In accordance with Lemma~\ref{lem:farkas-sphere}, we will refer to $\lambda\in\Rp^m$ satisfying \eqref{eq:optimality} as a {\em dual certificate of validity} of $v^\top x\ge \nu$ for $P\cap \bb{B}^n(r)$.

\begin{definition}\label{def:q-def}
Let $P=\{x\st Ax\le u\}$ for $A\in\R^{m\times n}$, $u\in\R^m$, and $r>0$.
Let $q: \R^n\rightarrow \R$ be a concave quadratic form given as
\item  \begin{equation}\label{eq:q-def}
q(x)\coloneqq \gamma_0(r^2-\|x\|^2)+\sum_{i=1}^m \gamma_i(u_i-{\separ_i^\T
  x})({\separ_i^\T x}-\ell_i)+\da^\top x-\beta
\end{equation}
for $\gamma\in \Rp^{m+1}$, $\ell\in\R^m$, $\da\in\R^n$, $\beta\in\R$.
We say that $q$ is {\em a certified concave quadratic form} for $P\cap \bb{B}^n(r)$ if we are also given
$\ld^{(i)}\in\Rp^m$, $i\in[m]$, $\vartheta\in\Rp^m$ such that
\begin{itemize}
  \item  $r\|A^\top \ld^{(i)}+a_i\|+\ell_i\le-u^\top \ld^{(i)}$ (certifying $a_i^\top x\ge \ell_i$ is valid for $P\cap \bb{B}^n(r)$)
\item $r\|A^\top \vartheta+\da\|+\beta\le-\vartheta^\top u$ (certifying $\da^\top x\ge \beta$ is valid for $P\cap \bb{B}^n(r)$)
\end{itemize}
We will also say that the quadratic form \eqref{eq:q-def}  is {\em certified} by $\ld^{(i)}$, $i\in[m]$, and $\vartheta$.
\end{definition}

The certificates  $\ld^{(i)}$, $i\in[m]$, and $\vartheta$ guarantee that $P\cap \bb{B}^n(r)\subseteq \{x\st q(x)\ge 0\}$, because for every $x\in P$ and $i\in[m]$ we have that $\ell_i\le \separ_i^\T x\le u_i$, implying $(u_i-{\separ_i^\T
  x})({\separ_i^\T x}-\ell_i)\ge 0$. We will show in Lemma~\ref{lemma:gamma-optimality} that, if $q(x)$ is certified for $P$, then any inequality $v^\top x\ge \nu$ that is valid for $\{x\st q(x)\ge 0\}$ admits a closed-form dual certificate of its validity for $P\cap \bb{B}^n(r)$ that can be derived from the representation of $q(x)$.
As the first step, we need the following technical lemma.

\begin{lemma}
\label{lemma:gamma-infeasibility}
Let $q$ be a strictly concave quadratic form $q$ for $A,u,r,\ell,\gamma,\da,\beta$ as in  Definition~\ref{def:q-def}. In particular, assume we are also given $\vartheta\in \Rp^m$ certifying $\da^\top x\ge \beta$ for $P\cap \B^n(r)$ by $r\|A^\top \vartheta+\da\|+\beta\le-\vartheta^\top u$. Let $\centre\in\R^n$ be the maximizer of $q(x)$.  Define $\lambda\in\R^m$ by $\lambda_i=\gamma_i(\ell_i+u_i-2{\separ_i^\T \centre})$, $i\in [m]$.
Then
\[r\|A^\top (\lambda-\vartheta)\|\le q(p)+\ell^\top \lambda^+-u^\top \lambda^--u^\top \vartheta.\]
\end{lemma}
\begin{proof} Since $q$ is a strictly concave quadratic form, it achieves a unique maximum, which we denoted $p$. The maximum must satisfy  $\nabla q(\centre)=0$. We use the notation
$\bar\ell\coloneqq A\centre-\ell$ and $\bar
u\coloneqq u-A\centre$.
The following equation states that $\nabla q(\centre)=0$, and the next one  computes the value of $q(\centre)$, expressed with $\bar u_i$ and $\bar \ell_i$:
\begin{eqnarray}
\sum_{i=1}^m \gamma_i(\bar u_i-\bar\ell_i)\separ_i +\da &=& 2\gamma_0 p \label{eq:gradient=0}\\
\gamma_0(r^2-\|p\|^2)+\sum_{i=1}^m \gamma_i\bar u_i \bar \ell_i+\da^\top p-\beta   &=&q(\centre)\, . \label{eq:less than zero}
\end{eqnarray}
Note that $\lambda_i=\gamma_i(\bar u_i-\bar \ell_i)$ for all $i\in [m]$. Hence, \eqref{eq:gradient=0} can be written as $A^\top \lambda+\da=2\gamma_0p$.
Let ${I^+}\coloneqq \{i\in [m] \st \lambda_i\geq 0\}$ and ${I^-}\coloneqq [m]\sm {I^+}$.
The proof is completed by
\begin{eqnarray*}
\ell^\top \lambda^+-u^\top \lambda^- -u^\top \vartheta&=& - \sum_{i\in {I^+}} \lambda_i \bar\ell_i+\sum_{i\in {I^-}}\lambda_i\bar u_i+\sum_{i=1}^m\lambda_i{\separ_i^\T \centre}-u^\top \vartheta\\
(\mbox{by \eqref{eq:gradient=0}})&=& -\sum_{i\in {I^+}} \gamma_i (\bar u_i- \bar \ell_i) \bar \ell_i +\sum_{i\in {I^-}}\gamma_i(\bar u_i-\bar \ell_i)\bar u_i+2\gamma_0\|p\|^2-\da^\top p-u^\top \vartheta\\
      &=& -\sum_{i=1}^m \gamma_i\bar u_i \bar \ell_i +\sum_{i\in {I^+}} \gamma_i \bar \ell^2_i+ \sum_{i\in {I^-}}\gamma_i\bar u^2_i+2\gamma_0\|p\|^2-\da^\top p-u^\top \vartheta\\
(\mbox{by \eqref{eq:less than zero}}) &\ge& -q(\centre)+\gamma_0(r^2+\|p\|^2)-\beta -u^\top \vartheta\\
(\mbox{by the definition of $\vartheta$)} &\ge& -q(\centre)+2r\gamma_0 \|p\|+r\|A^\top \vartheta+\da\| \\
(\mbox{by \eqref{eq:gradient=0}, taking the norm}) &=&-q(\centre)+r\left\|A^\top  \lambda+\da\right\|+r\|A^\top \vartheta+\da\|\\
&\ge& -q(\centre)+r\left\|A^\top  (\lambda- \vartheta)\right\|\, .
\end{eqnarray*}
\end{proof}

The next lemma shows that, if $P$ is a polyhedron and $q$ is a strictly concave certified  quadratic form for $P\cap \bb{B}_n(r)$, then, for any $v\in\R^n$, we can compute in closed form a dual certificate for $P\cap \bb{B}_n(r)$ for the inequality $v^\top x\ge \nu$ where $\nu=\min\{v^\top x\st q(x)\ge 0\}$. This is a variant of \cite[Proposition 3.1 and Theorem 3.2]{burrell-todd}.
Recall that if $q$ is a strictly concave quadratic form, then there exists $R\in\Snpp$ such that  $q(x)=-(x-p)^\top R(x-p)+q(p)$, where  $p$ is the unique maximizer of $q$. In particular, if $q(p)>0$ then $\{x\in\R^n\st q(x)\ge 0\}=\sqrt{q(p)} E(R,p)$.

\begin{lemma}\label{lemma:gamma-optimality} Let $P=\{x\in\R^n\st Ax\le u\}$, where $A\in\R^{m\times n}$ and $u\in \R^m$. Let $r>0$, and let $q: \R^n\rightarrow \R$ be a strictly concave quadratic form as in~\eqref{eq:q-def},
certified for $P\cap \bb{B}^n(r)$ by $\ld^{(i)}\in\Rp^m$, $i\in[m]$ and $\vartheta\in\Rp^m$. Assume that $\max_{x\in \R^n} q(x)>0$,  and let $p\coloneqq \arg\max q(x)$, $\alpha\coloneqq \sqrt{q(p)}$. Define $R=\gamma_0 I_n+\sum_{i=1}^m \gamma_i a_i a_i^\top$.

Given $v\in\R^n$, the optimal solution and optimum value to $\min\{v^\top x\st q(x)\ge 0\}$ are
\[ x^*=p-\alpha R^{-1}v/\|v\|_{R^{-1}}\,\quad\mbox{and}\quad \nu\coloneqq v^\top p-\alpha\|v\|_{R^{-1}}\, ,\]
respectively.
 Define $\lambda, \tilde \lambda\in\R^m$ by
$$\lambda_i=\frac{\|v\|_{R^{-1}}}{2\alpha} \gamma_i(\ell_i+u_i-2{\separ_i^\T x^*})
\quad \forall i\in [m]\,,\qquad \tilde \lambda\coloneqq \sum_{i=1}^m \lambda_i^+\ld^{(i)}+\lambda^-+\frac{\|v\|_{R^{-1}}}{2\alpha}\vartheta.$$
Then, $\tilde \lambda$ is a dual certificate for ${v^\T x}\geq \nu$  for $P\cap \bb{B}^n(r)$, that is, $r\|A^\top  \tilde\lambda+v\|+\nu<-\tilde \lambda^\top u$. Furthermore, $\tilde \lambda$ can be computed in time $O(n^2 m+n^{\omega})$.
\end{lemma}
\begin{proof} By the KKT conditions, it follows easily that $x^*=\arg\min\{{v^\T x}\st x\in E(R,p)\}$ and $v^\top x^*=\nu$. Define
$\sigma\coloneqq \frac{\|v\|_{R^{-1}}}{2\alpha}$ and $\tilde\gamma_i\coloneqq \sigma\gamma_i$ for $i=0,\ldots,m$.
Consider the polyhedron $\tilde P=\{x\st Ax\le u,\, v^\top x\le \nu\}$, and the quadratic form defined by $\tilde q(x)=\sigma q(x)-v^\top x+\nu$. Observe that
$$\tilde q(x)\coloneqq \tilde \gamma_0(r^2-\|x\|^2)+\sum_{i=1}^m \tilde\gamma_i(u_i-{\separ_i^\T x})({\separ_i^\T x}-\ell_i)+(\sigma \da-v)^\top  x -(\sigma\beta-\nu)\, ,$$
hence $\tilde q$ is certified for $\tilde P\cap B^n(r)$ by $\lambda^{(i)}$, $i\in [m]$, and by $\|A^\top (\sigma\vartheta)+v+(\sigma \da-v)\|+\sigma\beta-\nu\le -\sigma\vartheta^\top u-\nu$.

Since $\tilde q(x)=\sigma(\alpha^2-(x-\centre)^\T R(x-\centre))+\nu-{v^\T x}$, we have that $\tilde q$ is strictly concave and $\nabla \tilde q(x)=-2\sigma R(x-\centre)-v$. This implies that $x^*$ is the unique maximizer of $\tilde q$ since $\nabla\tilde q(x^*)=0$. Furthermore, one can compute that  $\tilde q(x^*)=0$.
Applying Lemma~\ref{lemma:gamma-infeasibility} to $\tilde q$, and
observing that $\lambda_i=\tilde{\gamma}_i(u_i+\ell_i-2{\separ_i^\T
  x^*})$, $i\in [m]$,
we obtain
\begin{equation}\label{eq:intermediate-certificate}\ell^\top\lambda^+-u^\top \lambda^- -u^\top (\sigma\vartheta)  -\nu\geq \tilde{q}(x^*)+r\left\|A^\top (\lambda-\sigma\vartheta) -v\right\|=r\left\|A^\top (\lambda-\sigma\vartheta) -v\right\|.\end{equation}

We need to show that $r\|A^\top  \tilde\lambda+v\|+\nu\le-\tilde \lambda^\top u$. From \eqref{eq:intermediate-certificate}, we have
\begin{eqnarray*}
\nu&\le &\ell^\top\lambda^+-u^\top (\lambda^- +\sigma\vartheta) - r\left\|A^\top (\lambda-\sigma\vartheta)-v\right\|\\
\mbox{(by definition of $\mu^{(i)}$)}&\le&\sum_{i=1}^m \lambda^+_i \left(-u^\top \ld^{(i)}-r\left\|A^\top \ld^{(i)}+a_i\right\|\right)-u^\top (\lambda^- +\sigma\vartheta) - r\left\|A^\top (\lambda-\sigma\vartheta)-v\right\|\\
\mbox{(triangle inequality)}&\le&-u^\top\left(\sum_{i=1}^m \lambda^+_i  \ld^{(i)}+\lambda^-+\sigma\vartheta\right)-r\left\|\sum_{i=1}^m \lambda^+_i \left(A^\top \ld^{(i)}+a_i\right)-A^\top (\lambda-\sigma\vartheta)+v\right\|\\
&=&-u^\top\left(\sum_{i=1}^m \lambda^+_i  \ld^{(i)}+\lambda^-+\sigma\vartheta\right)-r\left\|A^\top \left(\sum_{i=1}^m \lambda^+_i \ld^{(i)}+\lambda^- + \sigma\vartheta\right)+v\right\|\\
&=& -\tilde \lambda^\top u-r\|A^\top  \tilde\lambda+v\|.
\end{eqnarray*}
To compute $\tilde \lambda$, we need to compute $R$, which can be done in time $O(n^2m)$, as well as $x^*$ and $p$. The time to compute  these two points is dominated by the computation of $R^{-1}$, which can be performed in time $O(n^\omega)$. Computing $\lambda$ and $\tilde \lambda$ requires time $O(nm)$.
\end{proof}

The above lemma will be used to compute $\varepsilon$-approximate Farkas certificates from the ellipsoid method (Section~\ref{sec:ellipsoid}), from the geometric rescaling algorithms~\cite{DVZ,rothvoss} (Section~\ref{sec:rescale}), and from volumetric cutting plane methods~\cite{Jiang2020,Lee2015,vaidya96} (Section~\ref{sec:Vaidya}).
For all three, we show how one can find an appropriate certified quadratic form for the polyhedron defined by the current set of oracle inequalities. For the ellipsoid method and the geometric rescaling algorithms, such quadratic form will need to be maintained explicitly at every iteration. For volumetric cutting plane algorithms, we will instead show that, once the algorithm has achieved the required level of accuracy, we can a-posteriori compute a suitable certified quadratic form directly from the information that is maintained by the algorithm.

\subsection{Finding approximate Farkas certificates}\label{sec:thin}

The following theorem is the main technical tool for finding an $\varepsilon$-approximate Farkas certificate. The theorem shows that, given a convex set $K$, if we have a strictly concave certified quadratic form $q(x)$ for $K\cap \bb{B}^n(r)$, and if we have a direction $v$, $\|v\|\ge 1$, such that the ellipsoid $\{x\st q(x)\ge 0\}$ has small width in the direction of $v$, then we can compute an $\varepsilon$-approximate Farkas certificate for $K\cap \bb{B}^n(r)$. All methods we consider start from some initial simple set containing $K\cap \bb{B}^n(r)$. For the ellipsoid method and the geometric rescaling algorithm, the initial relaxation is simply $\bb{B}^n(r)$, whereas for the volumetric cutting plane algorithms, the initial relaxation is $[-r,r]^n$. In particular, for the ellipsoid method and the geometric rescaling algorithms,  the certified quadratic form~\eqref{eq:q-def} has $\gamma_0>0$ (this corresponds to the initial quadratic form $r^2-\|x\|^2\ge 0$), whereas for Vaidya's algorithm we will always have $\gamma_0=0$, but some of the inequalities $a_i^\top x\le u_i$ may be the initial box-constraints $x_j\le r$ or $-x_j\le r$. Note that, in both cases, the $\varepsilon$-approximate Farkas certificate computed using the theorem below will always be purely in terms of the oracle inequalities for $K$.

\begin{theorem}\label{thm:approximate-farkas}
Let $K\subseteq \R^n$ be a convex set, $r>0$, and  $\varepsilon\in(0,4r)$. Assume we are given inequalities $Ax\le u$ valid for $K$, $A\in\R^{m\times n}$, $u\in \R^m$, and let $\bar A x\le \bar u$, $\bar A\in\R^{k\times n}$, $\bar u\in \R^k$, $k\ge m$,  be a system comprising all inequalities in $Ax\le u$ and some of the ``box constraints'' $x_j\le r$ or $-x_j\le r$, $j\in [n]$ (in particular, we assume that the first $m$ inequalities of  $\bar A x\le \bar u$ define the system $Ax\le u$, and for $i=m+1,\ldots,k$ the inequalities $a_i^\top x\le u_i$ denote box constraints). Assume we are also given $\ell\in\R^k$ and $\ld^{(i)}\in\Rp^k$, $i\in [k]$ such that $r\|\bar A^\top \lambda^{(i)}+a_i\| +\ell_i\le -\bar u^\top \ld^{(i)}$, and $\da\in\R^n$, $\beta\in\R$, $\vartheta\in\Rp^k$ such that $\|\bar A\vartheta+\da\|+\beta\le -\bar u^\top \vartheta$. Let $\gamma\in \Rp^{k+1}$, and assume that the quadratic form
\[q(x)=\gamma_0 (r^2-\|x\|^2)+\sum_{i=1}^k \gamma_i (u_i-a_i^\top x)(a_i^\top x-\ell_i)+\da^\top x-\beta.\]
is strictly concave. Let $E=\{x\st q(x)\ge 0\}$. If we are given $v\in\R^n$ such that $\|v\|\ge 1$ and $\width_E(v)\le\varepsilon/3$, then in time $O(n^2m+n^\omega)$ we can compute an $\varepsilon$-approximate Farkas certificate for $K\cap \bb{B}^n(r)$ in terms of the inequalities $Ax\le u$, that is, $\lambda\in\Rp^m$ such that $\lambda^\top u+r\left\|A^\top \lambda\right\|<\varepsilon$, $\sum_{i=1}^m \lambda_i \|a_i\|\ge 1$.
\end{theorem}

\begin{proof}
Recall that $E$ is an ellipsoid centered at $p\coloneqq \arg\max_{x\in\R^n} q(x)$. Since $\width_E(v)\le\varepsilon/3$, it follows that $E\subseteq \{x\in\R^n\st -\varepsilon/6\le v^\top x-v^\top p\le \varepsilon/6\}$.

It will be convenient to assume $\|a_i\|_1=1$ for $i\in [m]$, which is without loss of generality. By Lemma~\ref{lemma:gamma-optimality}, in time $O(n^2m+n^\omega)$ we can compute dual certificates for $\{x\st \bar Ax\le \bar u\}\cap \bb{B}^n(r)$ of both inequalities $-v^\top x\ge -v^\top p-\varepsilon/6$ and $v^\top x\ge v^\top p-\varepsilon/6$. To distinguish the roles of the constraints in $Ax\le u$ from the box constraints, we express these certificates by two vectors $(\lambda',\mu'), (\lambda'',\mu'')\in\Rp^m\times \R^n$, where $(\mu')^+, (\mu'')^+$ define the multipliers for the inequalities $x_j\le r$, and $(\mu')^-, (\mu'')^-$ define the multipliers for the inequalities $-x_j\le r$ (where the multiplier will have 0 value if the corresponding box constraint is not part of the system $\bar A x\le \bar u$). Hence $(\lambda',\mu')$ and $(\lambda'',\mu'')$ satisfy
\[r\|A^\top \lambda'+\mu'-v\|-v^\top p-\varepsilon/6\le -u^\top \lambda'-r\|\mu'\|_1,\quad r\|A^\top \lambda''+\mu''+v\|+v^\top p-\varepsilon/6\le -u^\top \lambda''-r\|\mu''\|_1.\]
Note that, since $\|\mu'\|\le \|\mu'\|_1$ and $\|\mu''\|\le \|\mu''\|_1$, the triangle inequality implies
\begin{equation}\label{eq:two-side-cert}r\|A^\top \lambda'-v\|-v^\top p-\varepsilon/6\le -u^\top \lambda',\quad r\|A^\top \lambda''+v\|+v^\top p-\varepsilon/6\le -u^\top \lambda''.\end{equation}

In what follows, we show that $\lambda=(\lambda'+\lambda'')/\|\lambda'+\lambda''\|_1$ is an $\varepsilon$-approximate Farkas certificate. By definition, $\|\lambda\|_1=1$, hence $\sum_{i=1}^m \lambda_i \|a_i\|_1\ge 1$ by our assumption that $\|a_i\|_1=1$.

Adding up the two inequalities in \eqref{eq:two-side-cert} we obtain
\begin{equation}\label{eq:combined-cert}
\frac{\varepsilon}{3}\ge u^\top (\lambda'+\lambda'')+ r\|A^\top \lambda'-v\|+r\|A^\top \lambda''+v\|
\end{equation}

From the triangle inequality, we have
\[
\alpha\coloneqq \|A^\top \lambda'-v\|+\|A^\top \lambda''+v\|-\|A^\top (\lambda'+\lambda'')\|\ge 0\, .
\]

If $\alpha>\frac{\varepsilon}{3r}$, then the two equations above give $u^\top (\lambda'+\lambda'')+ r\|A^\top (\lambda'+\lambda'')\|<0$, proving $u^\top \lambda+ \|A^\top \lambda\|<0$, and we are done.

Assume therefore that $\alpha\le\frac{\varepsilon}{3r}$. From \eqref{eq:combined-cert}, we have 
\[
\frac{\varepsilon}{3\|\lambda'+\lambda''\|_1}\ge u^\top \lambda+ r\|A^\top \lambda\|\,,
\]
hence it suffices to show that $\|\lambda'+\lambda''\|_1=\|\lambda'\|_1+\|\lambda''\|_1\ge 1/3$.

From the triangle inequality, using that $\|a_i\|\le \|a_i\|_1=1$ for all $i\in[m]$, we obtain
\[
1\le \|v\|\le \left\|A^\top \lambda'\right\|+ \left\|A^\top \lambda'- v\right\|\le \sum_{i=1}^m \lambda'_i\|a_i\| + \left\|A^\top \lambda'- v\right\| \le \|\lambda'\|_1 + \left\|A^\top \lambda'- v\right\|\, ,
\]
and a similar inequality holds for $\lambda''$.
Adding up the two bounds and using the definition of $\alpha$, we get
\[
\alpha+\left\|A^\top(\lambda'+\lambda'') \right\|=\left\|A^\top \lambda'- v\right\|+\left\|A^\top \lambda''- v\right\|\ge 2-(\|\lambda'\|_1+\|\lambda''\|_1)\, .
\]
Using the assumption $\alpha\le \varepsilon/(3r)$ and the upper bounds $\|a_i\|\le \|a_i\|_1= 1$,
\[
\frac{\varepsilon}{3r}+\|\lambda'\|_1+\|\lambda''\|_1\ge \alpha+\sum_{i=1}^m (\lambda'_i+\lambda''_i)\|a_i\|\ge 2- (\|\lambda'\|_1+\|\lambda''\|_1).
\]
Since $\varepsilon\le 4r$, the above implies $\|\lambda'\|_1+\|\lambda''\|_1\ge \frac{1}{3}$ as required.
\end{proof}

\subsection{Approximate Farkas certificates from the ellipsoid method}\label{sec:ellipsoid}
In this section we describe an algorithm, which we will call the {\em Certified Ellipsoid Method} and prove the following.

\begin{theorem}\label{thm:certified-ellipsoid}
Let $K$ be a  convex set given by a strong
separation oracle, $r>0$, and $\varepsilon\in (0,2 r)$. Then, the Certified Ellipsoid Method runs in oracle-polynomial time and, by making $O(n^2 \log(nr/\varepsilon))$ calls to the strong separation oracle, either returns a point $x\in K$, or  an $\varepsilon$-approximate Farkas certificate for $K\cap\bb{B}^n(r)$ comprising only oracle inequalities.
\end{theorem}

As noted at the beginning of Section~\ref{sec:dual certificates}, this yields a subroutine \textsc{Approx-Conic-Dual} with
$\mathcal{T}(n,\varepsilon)=O(n^2\log(1/\varepsilon))$ oracle calls.

\begin{remark}\label{rmk:r=1} In the remainder of the section, we will assume $r=1$. Indeed, if we define $K'=K/r$, and $\varepsilon'=\varepsilon/r$, observe that an inequality $a^\top x\le \beta$ is valid for $K$ if and only if $ra^\top x\le \beta$ is valid for $K'$. Hence, given a system $rAx\le u$ of $m$ valid inequalities for $K'$, an $\varepsilon'$-approximate Farkas certificate for $K'\cap \bb{B}^n(1)$ is of the form $\lambda'\in\Rp^m$ such that $\|rA^\top \lambda'\|+u^\top \lambda'\le \varepsilon'$, $\sum_{i=1}^m \lambda'_i r\|a_i\|\ge 1$, hence $\lambda=r\lambda'$ is an $\varepsilon$-approximate Farkas certificate for $K\cap \bb{B}^n(r)$.
\end{remark}

We give a self-contained exposition of the ellipsoid method and the certification procedure that may provide new insights into the classical algorithm. Section~\ref{sec:basic-ellipsoid} describes the Basic Ellipsoid Method in a slightly stronger form. Section~\ref{sec:quadratic} introduces the Certified Ellipsoid Method, where we modify the original framework to show how the containing ellipsoid can always be maintained in terms of certified quadratic forms, which will imply Theorem~\ref{thm:certified-ellipsoid}.

\subsubsection{The Basic Ellipsoid Method}\label{sec:basic-ellipsoid}

\begin{theorem}\label{thm:basic-ellipsoid}
Let $K$ be a convex set given by a strong
separation oracle, and $\varepsilon>0$. Then, there exists an oracle-polynomial
time
algorithm that, by making $O(n^2 \log(n/\varepsilon))$ calls to the strong separation oracle, either returns a point $x\in K$, or  returns a
vector $\separ\in\R^n$ with $\|\separ\|_2\in [1,2]$ such that $\width_{K\cap\bb{B}^n(1)}(\separ)\le \varepsilon$. Moreover, the vector
$\separ$ will be one of the vectors returned by the separation oracle during the algorithm.
\end{theorem}

The statement differs from the usual form in the case when no
feasible solution is found. For this case, the usual
outcome as in \cite[Theorem 3.2.1]{gls} is a small volume
ellipsoid containing $K$. Given such an ellipsoid $E$, one can show
that $E$ (and thus
$K$) is thin in the direction in one of the principal axes
of $E$ (see e.g. the proof of \cite[Lemma 6.4.2]{gls}). Instead, using Lemma~\ref{lem:c-Q-bound}, we observe
that a small volume ellipsoid is thin in one of the directions
returned by the oracle.

\medskip

At the $t$-th iteration, the algorithm maintains positive definite matrices
$R_t,Q_t\in \Snpp$ such that $Q_t=R_t^{-1}$, and a vector $p_t\in
\R^n$. We let $E_t=E(R_t,p_t)$. Throughout,  we maintain
\begin{equation}\label{eq:p-contain}
K\cap\bb{B}^n(1)\subseteq E_t\, .
\end{equation}
We initialize $E_0=\bb{B}^n(1)=E(I_n,0)$.
At iteration $t=1,2,\ldots$, the strong separation oracle is called to check if $p_{t-1}\in
K$. If the answer is yes, the algorithm terminates. Otherwise, the
oracle returns a direction $\separ_t$ such that $\separ_t^\top p_{t-1}>\separ_t^\top z$ for any $z\in K$.
In this case,
the matrices are updated as follows:
\begin{equation}\label{eq:e-update-1}
  R_{t}=\frac{n^2-1}{n^2}\cdot  R_{t-1}+\frac{2n+2}{n^2}\cdot
  \frac{\separ_t \separ_t^\T}{\|\separ_t\|_{Q_{t-1}}^2}\, ,\quad Q_{t}=R_{t}^{-1}\, ,\quad p_{t}=p_{t-1}-
  \frac{Q_t \separ_t}{(n+1)\|\separ_t\|_{Q_t}}\, .\end{equation}
We note that computing $Q_{t}$ does not require a matrix inversion,
but can be computed from $Q_{t-1}$ via a simple rank-1 update using the Sherman–Morrison formula.

The following lemma is the key in showing the progress in the
volumetric potential. In Section~\ref{sec:quadratic}, we present a new proof
using quadratic forms.
\begin{lemma}\label{lem:ellipsoid-progress}
If $K\cap \bb{B}^n(1)\subseteq E_{t-1}$, then $K\cap\bb{B}^n(1)\subseteq
E_{t}$ also holds, and $\det(R_{t})>\ee^{1/n} \det(R_{t-1})$.
\end{lemma}
The next lemma is immediate from the construction sequence.

\begin{lemma}\label{lemma:R-form}
At  iteration $t\ge 1$ of the Basic Ellipsoid Method, there exist coefficients
$0<\gamma_0^{(t)}<1$,
and $\gamma^{(t)}_i>0 $ for $i\in [t]$ such that
\begin{equation}\label{eq:R-form}
R_t=\gamma_0^{(t)} I_n +\sum_{i=1}^t \gamma^{(t)}_i \separ_i \separ_i^\top.
\end{equation}
\end{lemma}

The next lemma completes the proof of Theorem~\ref{thm:basic-ellipsoid}.
\begin{lemma}\label{lem:basic-thin}
If the Basic Ellipsoid Method does not terminate in $T=O(n^2\log
({n}/\varepsilon))$ iterations, then $\width_{K\cap\bb{B}^n(1)}(\separ_t)\le \varepsilon$ for  some $t\in [T]$.
\end{lemma}

\begin{proof}
Initially, $\det(R_0)=1$. Using Lemma~\ref{lem:ellipsoid-progress}, we
see that after $t=O( n^2\log(n/\varepsilon))$ iterations,
$\det(R_t)^{1/{n}}> 1/\varepsilon^2 +1$. The claim follows using
Lemmas~\ref{lemma:ellipsoid width} and \ref{lem:c-Q-bound}.
\end{proof}

\subsubsection{The certified ellipsoid method}\label{sec:quadratic}

Let us denote by $q_t(x)$ the strictly concave quadratic form defined as
\begin{equation}\label{eq:quadratic-form}
q_t(x)\coloneqq 1-(x-p_t)^\T R_t (x-p_t)\, ,
\end{equation}
so that $E_t=\{x\in\R^n\st q_t(x)\geq 0\}$.
Note that initially $q_0(x)=1-\|x\|^2$.
The Certified Ellipsoid Method will maintain $q_t$ as a certified quadratic form as in Definition~\ref{def:q-def}. Theorem~\ref{thm:certified-ellipsoid} follows by combining Theorems~\ref{thm:approximate-farkas} and~\ref{thm:basic-ellipsoid}.

Assume that, at the current iteration $t$, the separation
oracle returned $\separ_t$ such that $\separ_t^\top p_{t-1}> \separ_t^\top
z$ for all $z\in K$. If we let
\[u_t\coloneqq \separ_t^\top p_{t-1}\,\]
then the inequality $a_t^\top x\le u_t$ is the inequality returned by the oracle, and it is valid for $K$. Further, if we let
\[\ell_t\coloneqq \separ_t^\top p_{t-1}-\|\separ_t\|_{Q_{t-1}}\,,\]
then the fact that $K\cap\bb{B}^n(1)\subseteq E_{t-1}$ ensures that $\separ_t^\top x\ge \ell_t$ is also valid for $K\cap\bb{B}^n(1)$.

The following lemma shows that the ellipsoid update \eqref{eq:e-update-1} corresponds to the following update of the quadratic form, for some values $\alpha,\beta>0$.
\begin{equation}\label{eq:q-t}
\begin{aligned}
q_{t}(x)&=\alpha q_{t-1}(x)+\frac{\beta}{\|\separ_t\|^2_{Q_{t-1}}} (u_t-\separ_t^\top x) (\separ_t^\top x-\ell_t)\, ,
\end{aligned}
\end{equation}

\begin{lemma}
\label{lemma:quadratic form}
For some $\alpha,\beta>0$, consider the expression $q_{t}(x)$ as in
\eqref{eq:q-t}  with $u_t=\separ_t^\top p_{t-1}$ and $\ell_t=\separ_t^\top p_{t-1}-\|\separ_t\|_{Q_{t-1}}$, and let $\centre_{t}$ be the unique maximizer of $q_{t}(x)$. There exists  $R_{t}\in\Snpp$ such that $q_{t}(x)=1-(x-\centre_{t})^\T R_{t} (x-\centre_{t})$ if and only if
\begin{equation}\label{eq:alpha and beta}
\alpha=\frac{1-2\gamma}{(1-\gamma)^2}\, ,\quad \beta=\frac{2\gamma}{(1-\gamma)^2}\, ,\quad\mbox{for some } 0<\gamma<\frac 12\, .
\end{equation}
Furthermore, for any such $\alpha$, $\beta$, and $\gamma$,
\begin{equation}\label{eq:new R ellipsoid} R_{t}=\alpha R_{t-1}+\beta \frac{\separ_t\separ_t^\T}{\|\separ_t\|_{Q_{t-1}}^2}\, ,\quad \centre_{t}=\centre_{t-1}+\gamma \frac{Q_{t-1}\separ_t}{\|\separ_t\|_{Q_{t-1}}}\, .\end{equation}
Finally, the choice of $\gamma$ that minimizes the volume of $E_{t}$ (or, equivalently, maximizes $\det(R_{t})$) is $\gamma^*\coloneqq 1/(n+1)$, for which choice $\det(R_{t})/\det(R_{t-1})\geq \ee^{1/n}$.
\end{lemma}
\begin{proof}
Denote $q_{t}\coloneqq q$, $\separ\coloneqq \separ_t/\|\separ_t\|_{Q_{t-1}}$,
$\centre\coloneqq \centre_{t-1}$, $\centre'\coloneqq \centre_{t}$, $R\coloneqq R_{t-1}$, $Q\coloneqq Q_{t-1}$, and $R'=R_t$.
If there exists
a positive definite matrix $R'$  and $p'$ such that $q(x)\equiv 1-(x-\centre')^\T
R'(x-\centre')$ then
$q(\centre')=1$. Since $\centre'$ is the unique maximizer of $q(x)$,
we have $\nabla q(\centre')=0$. Using that
\begin{equation}\label{eq:gradient ellispoids}
\nabla q(x)=-2\alpha R(x-\centre)-\beta (1+2\separ^\T (x-\centre)) \separ\, ,
\end{equation}
 from $\nabla q(\centre')=0$  we see that  $\centre'=\centre-\gamma Q\separ$, where   $\gamma\in \R$ satisfies
\begin{equation}\label{eq:gradient at new center}
2\alpha\gamma-\beta (1-2\gamma)=0.
\end{equation}
The condition $q(\centre')=1$ implies
\begin{equation}\label{eq:value at c'}
\alpha(1-\gamma^2)+\beta\gamma(1-\gamma)=1.
\end{equation}
Solving the linear system given by \eqref{eq:gradient at new center}, \eqref{eq:value at c'} for $\alpha$ and $\beta$ in terms of the parameter $\gamma$, we obtain \eqref{eq:alpha and beta}, and $\alpha,\beta>0$ if only if $0<\gamma<1/2$.

\medskip

Let us now fix a value of $0<\gamma<1/2$ and compute $\alpha$ and $\beta$ as above.
Let us set $R'=\alpha R+\beta \separ\separ^\top$. We show that
$q(x)\equiv q'(x)$ for $q(x)=q_{t}(x)$ as defined in \eqref{eq:q-t},
and for $q'(x)=1-(x-p')^\top R' (x-p')$. This is a consequence of
the following simple claim.
\begin{claim}
Two strictly concave quadratic functions are identical if and only if
the following are the same for the two functions:
{\em (i)} the quadratic terms; {\em (ii)} the unique maximizers; and
{\em (iii)} the
maximum values.
\end{claim}
The quadratic term in both $q(x)$ and $q'(x)$ is $-x^\top(\alpha
R+\beta \separ\separ^\top)x$. The maximizer of both functions is $p'$,
and the maximum value is 1 in both cases.
This completes the proof of \eqref{eq:new R ellipsoid}.

\medskip

Finally, we need to determine the choice of $\gamma$ in order to minimize the volume of $E_{t}$. This is equivalent to maximizing $\det(R')$. Note that
      \begin{equation}\label{eq:det-update}
\begin{aligned}
        \det(R')&=\alpha^{n}\det\left(R^{1/2}\left(I+\frac{\beta}{\alpha}Q^{1/2}\separ\separ^\T Q^{1/2}\right)R^{1/2}\right)=\det(R)\alpha^{n}\left(1+\frac{\beta}{\alpha}\right)\\
       &=\det(R)\alpha^{n-1}(\alpha+\beta)=\det(R)\frac{(1-2\gamma)^{n-1}}{(1-\gamma)^{2n}}\, .
\end{aligned}
      \end{equation}
This is maximized for $\gamma=1/(n+1)$, for which choice $$\frac{\det(R_{t})}{\det(R_{t-1})}=\left(\frac{n-1}{n}\right)^{n-1}\left(\frac{n+1}{n}\right)^{n+1}\geq \ee^{1/n}\, .$$
\end{proof}

Observe that the above lemma immediately implies Lemma~\ref{lem:ellipsoid-progress}, hence this provides an alternative exposition of the standard volumetric argument for the ellipsoid method.

\begin{corollary}\label{cor:ellipsoids}  At the $t$-th iteration of the Basic Ellipsoid Method, we can maintain $q_t(x)$ in the form
\begin{equation}\label{eq:q-t-def}
q_t(x)=\gamma^{(t)}_0 (1-\|x\|^2)+\sum_{i=1}^t \gamma^{(t)}_i(u_i-{\separ_i^\T
  x})({\separ_i^\T x}-\ell_i),
\end{equation}
where $\gamma^{(t)}_0=\alpha^t$, and for each $i\in [t]$,
 $\separ_i$ is the vector returned by the separation oracle at the $i$-th iteration, $u_i={\separ^\T_i\centre_{i-1}}$, $\ell_i=u_i-\|\separ_i\|_{Q_{i-1}}$, and
 $\gamma^{(t)}_i=\beta \alpha^{t-i}/\|\separ_i\|^2_{Q_i}$, where $\alpha$, $\beta$ are defined as in \eqref{eq:alpha and beta} for $\gamma=1/(n+1)$. Furthermore, for every $k\in[t]$, we can compute a certificate $\ld^{(k)}\in\Rp^t$ for the validity of $a_k^\top x\ell_k$, that is
 \begin{equation}\label{eq:lower-bound-certificate}
 \left\|\sum_{i=1}^t \ld^{(i)} a_i+a_k\right\|+\ell^{(k)}\le -\sum_{i=1}^t \lambda^{(k)}_i u_i.\end{equation}
\end{corollary}
\begin{proof}
The first statement follows by construction and by Lemma~\ref{lemma:quadratic form}. For the last statement, we observe that $\ld^{(k)}\in\Rp^t$, $k\in[t]$ can be computed throughout the execution of the ellipsoid method. Indeed, suppose that up to iteration $t-1$ we have computed $q_{t-1}$ in the form~\eqref{eq:q-t-def}, along with $\ld^{(1)},\ldots,\ld^{(t-1)}$. Since the inequality $a_t^\top x\ge \ell_t$ is valid for $E_{t-1}=\{x\st q_{t-1}(x)\ge 0\}$, it follows from Lemma~\ref{lemma:gamma-optimality} that we can compute $\ld^{(t)}$ satisfying \eqref{eq:lower-bound-certificate}.
\end{proof}

\paragraph{Maintaining a compact representation}
\label{sec:compact}
So far, we kept adding a new term $\separ_i\separ_i^\top$ to $R_{t}$ in every
iteration, and thus $R_t$ will be the weighted sum of the identity and $t$
rank-1 matrices $\separ_i\separ_i^\top$. This would lead to $O(n^2\log
(n/\varepsilon))$ terms in $R$ when running the  Basic Ellipsoid
Method to obtain a $\varepsilon$-thin direction, according to
Lemma~\ref{lem:basic-thin}.
Thus, the space complexity of the algorithm is large, albeit polynomial. Furthermore, we must maintain certificates $\ld^{(i)}$ for each inequality $a_i^\top x\ge \ell_i$, $i\in[t]$.
In what follows, we observe that one can maintain a subset $I_t$ of only $O(n^2)$ vectors $a_i$ returned from the oracle so that the quadratic form $q_t$  can be expressed only in  terms of the form $(u_i-a_i^\top x)(a_i^\top x-\ell_i)$ for $i\in I_t$,
and that one needs to only keep $O(n^3)$ furthers vectors $\separ_k$, $k\in[t]$ to certify the inequalities $a_i^\top x\ge \ell_i$ for all $i\in I_t$.

Indeed, we can write 
\[q_t(x)=(1,x^\top)
H \begin{pmatrix} 1\\ x\end{pmatrix},\, \mbox{ for } H=\left(\begin{array}{c|c}
1-p_t^\top R p_t& p_t^\top R_t\\
\hline
R_t^\top p_t & -R_t
\end{array}\right)\]
hence, from the expression~\eqref{eq:q-t-def}, the vector $(\gamma_1^{(t)},\ldots,\gamma_t^{(t)})$ is a feasible solution to the following system of  ${n+1\choose 2}$ equations in the nonnegative variables $\gamma_1,\ldots,\gamma_t$
\[\sum_{i=1}^t\gamma_i\left(\begin{array}{c|c}
-\ell_i u_i& a_i^\top (\ell_i+u_i)/2\\
\hline
a_i (\ell_i+u_i)/2 & -a_ia_i^\top
\end{array}\right)=H-\gamma^{(t)}_0 I_{n+1}.\]
If we choose $\gamma\in\Rp^t$ to be a basic solution and let $I_t\subseteq[t]$ be its support, it follows that $|I_t|\le  {n+1\choose 2}$ and \[q_t(x)=\gamma^{(t)}_0 (1-\|x\|^2)+\sum_{i\in I_t} \gamma_i(u_i-{\separ_i^\T
  x})({\separ_i^\T x}-\ell_i).\]
Furthermore, for every $k\in I_t$, given $\ld^{(k)}\in\Rp^t$ satisfying \eqref{eq:lower-bound-certificate}, we can compute a basic solution $\tilde \ld^{(k)}\in\Rp^t$ to the system $\sum_{i=1}^t \tilde\ld^{(i)} a_i=\sum_{i=1}^t \ld^{(i)} a_i$ with $\sum_{i=1}^t \tilde\ld^{(i)} u_i\le \sum_{i=1}^t \ld^{(i)} u_i$, hence obtaining a certificate for $a_k^\top x\ge \ell_k$ of support size at most $n$. If we denote by $C_k$ the support of $\tilde \ld^{(k)}$, then we only need to maintain the vectors $a_i$,  $i\in I^t\cup \bigcup_{i\in I_t} C_i$, along with the certificates $\tilde\ld^{(k)}\in\R^{I_t}_+$.

By induction, we will guarantee that we obtain $\ld^{(t)}$  as in
\eqref{eq:lower-bound-certificate} of support size $O(n^3)$. This can
be reduced to an $O(n)$ size basic solution in time $O(n^5)$. This amounts to a substantial running time overhead:
whereas an update takes $O(n^2)$ time for the Basic Ellipsoid Method,
maintaining a small $I$ amounts to $O(n^4)$ per update, and maintaining
the certificates to $O(n^5)$.\footnote{It is possible to improve the
  complexity of maintaining a small $I$ to $O(n^3)$ by only
  approximately maintaining $q(x)$.
} The increased complexity bound is
however still much lower than the algorithm for finding a dual optimal
solution  in \cite[Lemma 6.5.15]{gls} that would amount to running
the Ellipsoid Method for a second time with
$O(n^2\log (n/\varepsilon))$ variables (for an appropriate
$\varepsilon$ in the context of rational polyhedra).

\subsection{Approximate conic certificates from  a geometric rescaling method}\label{sec:rescale}

We consider a particular variant of the geometric rescaling algorithms
introduced by Dadush et al.~\cite{DVZ,DVZ-submodular} and by Hoberg and Rothvo{\ss}~\cite{rothvoss}. A similar certification should be applicable for other variants as well. The algorithm discussed here
takes as input a cone $K\subseteq \R^n$ defined by a conic separation oracle, along with $\varepsilon>0$, and either returns a point $x\in K$ or determines a ``thin direction'', that is, a vector $v\in\R^n$, $\|v\|=1$, such that $K\subseteq \{x\in\R^n\st |v^\top x|\le \varepsilon\}$. Here we briefly describe the algorithm and the analysis, to show that in the second case we can compute an $\varepsilon$-approximate conic Farkas certificate for $K$.

The algorithm will maintain, at each iteration $t$, a matrix $R_t\in\Snpp$ such that $K\cap\bb{B}^n(1)\subseteq E(R_t,0)$ and such that $\det(R_t)$ increases at each iteration. Throughout we denote $Q_t=R^{-1}_t$. Initially $R_0\coloneqq I_n$.
The algorithm sets a threshold $\thr\coloneqq 1/(11n)$.
At iteration $t+1$, the algorithm applies von Neumann's algorithm (first described by Dantzig in \cite{Dantzig-92}) to compute a set of vectors $\{m_i\st i\in J_t\}$ returned by the conic separation oracle and $\zeta^{(t)}\in\Rp^{J_t}$, $\|\zeta^{(t)}\|_1=1$, such that the vector
\[y\coloneqq \sum_{i\in J_t} \frac{ \zeta^{(t)}_i m_i}{\|m_i\|_{Q_{t}}},\]
either satisfies that $Q_{t} y\in K$ (in which case we stop), or $\|y\|_{Q_{t}}\le \thr$.  In the latter case, we update
  \begin{equation}\label{eq:rescale}
R_{t+1}\coloneqq \frac{R_t+P_t}{1+\thr},\qquad \mbox{ where } P_t\coloneqq \sum_{i\in J_t} \frac{\zeta^{(t)}_i}{\|m_i\|_{Q_{t}}^2}
  m_im_i^\T .
\end{equation}
We have $\det(R_{t+1})\ge (16/9)\det(R_{t})$ \cite[Lemma 11]{DVZ}.

Next we observe that the quadratic form $q_t(x)\coloneqq 1-x^\top R_t x$ can be maintained as a certified quadratic form for $K\cap \bb{B}^n(1)$. Inductively, assume that
\begin{equation}\label{eq:q-t-geometric}q_t(x)=\gamma_0 (1-\|x\|^2)+\sum_{i\in I_t} \gamma_i m_i^\top x (u_i-m_i^\top x)+\nu-v^\top x\, ,\end{equation}
where the oracle inequalities are $m_i^\top x\ge 0$, $i\in I_t$, and we maintain dual certificates of validity over  $K\cap \bb{B}^n(1)$ for the inequalities $m_i^\top x\le u_i$, $i\in I_t$, and for $v^\top x\le \nu$.

Note that, for all $i\in J_t$, by Cauchy-Schwartz the inequality $m_i^\top x\le \|m_i\|_{Q_t}$ is valid for $K\cap \bb{B}^n(1)$, because $\|x\|_{R_t}\le 1$  for all  $x\in K\cap \bb{B}^n(1)\subseteq E(R_t,0)$. Since $E(R_t,0)=\{x\st q_t(x)\ge 0\}$ and $q_t$ is a certified strictly concave quadratic form, by Lemma~\ref{lemma:gamma-optimality} we can compute dual certificates of validity for $m_i^\top x\le \|m_i\|_Q$ over $K\cap \bb{B}^n(1)$. Similarly, we have that $y^\top x\le \|y\|_{Q_t}\le \thr$ is valid for $K\cap \bb{B}^n(1)$, and we can compute a dual certificate of validity for it. Hence the quadratic form
\[s_t(x)=\sum_{i\in J_t} \frac{\zeta^{(t)}}{\|m_i\|^2_{Q_t}} m_i^\top x(\|m_i\|_{Q_t}-m_i^\top x) +\thr-y^\top x\, ,\]
is certified for $K\cap \bb{B}^n(1)$. It follows from the definition of $y$ and $P_t$ that $s_t(x)=\thr-x^\top P_t x$, hence
$q_{t+1}(x)=(1+\thr)^{-1}(q_t(x)+s_t(x))$. Since $q_{t}$ and $s_t$ are both certified, also $q_{t+1}$ is certified, as can be seen by observing that inequality  $(v+y)^\top x\le \nu+\thr$, corresponding to the linear-term in $q_{t+1}(x)$, can be certified as follows. If $\vartheta$ and $\vartheta'$ certify the validity of $v^\top x\le \nu$ and $y^\top x\le \thr$ for $K\cap \bb{B}^n(1)$, then $\vartheta+\vartheta'$ certifies the validity of $(v+y)^\top x\le \nu+\thr$, since $\|\sum_{i\in I_t}(\vartheta_i+\vartheta'_i)m_i +v+y\|\le \|\sum_{i\in I_t}\vartheta_i m_i +v\|+\|\sum_{i\in I_t}\vartheta'_im_i +y\|\le \nu+\thr$.

Note that, exactly as explained in Section~\ref{sec:quadratic}, we can apply Carath\'eodory to maintain the expression $\sum_{i\in I_t} \gamma_i m_i^\top x (u_i-m_i^\top x)$ in \eqref{eq:q-t-geometric} so that $|I_t|\in O(n^2)$.

\begin{theorem}
Let $K\subseteq\R^n$ be a  convex cone given by a conic
separation oracle, and $\varepsilon>0$. Then, in $O(n^3\log(n/\varepsilon))$ calls to the separation oracle, and  $O(n^5\log(n/\varepsilon))$ arithmetic operations, the geometric rescaling algorithm either returns a point $x\in K$, or  an $\varepsilon$-approximate conic Farkas certificate.
\end{theorem}
\begin{proof}
Since $\det(R_{t+1})\ge (16/9)\det(R_{t})$, after $T\in O(n\log(n/\varepsilon))$ iterations, there exists $k\in I_T$ such that $\|m_k\|_{Q_T}\le \varepsilon \|m_k\|_2$. Since the quadratic form $q_T(x)$ is certified, we can compute a certificate of validity $\lambda\in\Rp^{I_T}$ for the inequality $(m_k/\|m_k\|_2)^\top x\le \varepsilon$, that is, $\|\sum_{i\in I_T}\lambda_i m_i+ m_k/\|m_k\|_2\|\le \varepsilon$. This gives a $\varepsilon$-approximate conic Farkas certificate.
For the running time, each von Neumann call requires $\ceil{1/\thr^2}=O(n^2)$ calls to the separation oracle and $O(n^3)$ arithmetic operations \cite[Lemma 8]{DVZ}, so in particular $|J_t|\in O(n^2)$ at each iteration $t$. Hence the total number of oracle calls is $O(n^3\log(n/\varepsilon))$.  Computing the new matrix $R_t$ for $t=1,\ldots, T$ requires time $O(n^2 |J_t|)\in O(n^4)$, which gives $O(n^5\log(n/\varepsilon))$ arithmetic operations overall. \end{proof}

\subsection{Approximate Farkas certificates from volumetric cutting plane methods}\label{sec:Vaidya}

We now show how to derive Farkas certificates from volumetric cutting plane methods. This family of methods was introduced by Vaidya~\cite{vaidya96};
 similar cutting plane methods with improved arithmetic complexity were given by Lee, Sidford, and Wong~\cite{Lee2015} and by Jiang, Lee, Song, and Wong~\cite{Jiang2020}. For simplicity, we present the implementation for Vaidya's original algorithm, but the same framework is applicable for the subsequent variants as well, as we will discuss briefly at the end of this subsection. As previously noted, \cite{Lee2015} also includes an explicit statement on dual certificates.

 In contrast to the ellipsoid method, we do not maintain the certified concave quadratic form during the algorithm, but construct it at termination using Lemma~\ref{lemma:sandwich}.

For a full-dimensional polytope $P=\{x\in\R^n\st Ax\le u\}$, $A\in\R^{m\times n}$, $u\in\R^m$, the log-barrier function is defined over the interior of $P$ by
$f(x)=-\sum_{i=1}^m \log(u_i-a_i^\top x)$,
 and its Hessian can be readily computed to be
\[\nabla^2 f(x)=\sum_{i=1}^m \frac{a_i a_i^\top}{(u_i-a_i^\top x)^2}.\]
Let $F$ be the function defined over the interior of $P$ by
$F(x)=\log(\det(\nabla^2 f(x)))$.

The function $F$ is strictly convex~\cite{vaidya96}, and the unique minimizer $\tilde\omega$ of $F$ is called the {\em volumetric center} of $P$.
The gradient of $F$ is
\[\nabla F(x)=\sum_{i=1}^m \sigma_i(x)\frac{a_i}{u_i-a_i^\top x},\quad \mbox{where }\quad \sigma_i(x)=\frac{a_i^\top (\nabla^2 f(x))^{-1}a_i}{(u_i-a_i^\top x)^2}, \qquad i\in [m].\]
The $\sigma_i(x)$ values are the \emph{leverage scores} of the current system $Ax\le u$, and they can be thought of as a measure of the relative importance of the constraints in the current system. It is easy to show that $\sum_{i=1}^m \sigma_i(x)=n$ and $\sigma_i(x)\le 1$ for $i\in [m]$. Finally, let us define
\[Q(x)=\sum_{i=1}^m \sigma_i(x)\frac{a_i a_i^\top}{(u_i-a_i^\top x)^2}.\]

Since $\sigma_i(x)\le 1$, $i\in[m]$,  $\nabla^2 f(x)\succeq Q(x)$. Let $\mu(x)$ be the largest number $\tilde \mu$ such that $Q(x)\succeq\tilde \mu \nabla^2 f(x)$, and note that $\min_{i\in[m]}\sigma_i(x)\le \mu(x)<1$.

Let $K$ be a convex set defined by a strong separation oracle, let $\varepsilon>0$, and assume we want to find a point in $K$ or an $\varepsilon$-approximate Farkas certificate for $K\cap \bb{B}^n(r)$; as in  Remark~\ref{rmk:r=1}, we can assume $r=1$. Let us set the values 
\[\delta\le 10^{-4}\, ,\quad\mbox{and}\quad c\le 10^{-3}\delta\, .
\]
 At every iteration, Vaidya's algorithm maintains  a set of inequalities $Ax\le u$ defining a polytope, where $a_i^\top x\le u_i$ is either an inequality returned by the oracle, or one of the bound inequalities $x_j\le 1$ or $-x_j\le 1$. The algorithm also maintains a point $z\in\R^n$ satisfying $A z<u$ such that
\begin{equation}\label{eq:Vaidya-proximity}
F(z)- F(\tilde \omega)\le c^4 \mu(\tilde \omega).
\end{equation}
The system is initialized to $-1\le x_j\le 1$, $j\in [n]$, and $z=\tilde\omega=0$. At any given iteration, if $\sigma_i(z)<c$ for some $i\in[m]$, the inequality $a_i^\top x\le u_i$ is removed from the system. If $\min_{i\in [m]}\sigma_i(z)\ge c$, the separation oracle is queried for $z$; if $z\in K$ then the algorithm terminates, else the oracle returns a vector $a_i\in\R^n$ such that $a_i^\top x<a_i^\top z$ for all $x\in K$, and the algorithm adds the inequality $a_i^\top x\le u_i$ to the system, for some appropriately chosen value $u_i>a_i^\top z$ (note that the right-hand-side of the new inequality is chosen to maintain the current point $z$ in the interior of the current polytope). In both cases (whether an inequality has been added or removed), the algorithm performs a fixed number of Newton steps to compute a new point $z$ satisfying \eqref{eq:Vaidya-proximity} with respect to the new system. Each iteration requires  $O(n^\omega)$ arithmetic operations; the algorithms \cite{Jiang2020,Lee2015} improve on the arithmetic complexity.

Let us denote by $\rho^{(k)}$ the value of $F(\tilde \omega)$ at the $k$th iteration. Note that at the beginning $\nabla^2 f(z)=I_n$, hence $\rho^{(0)}=\log(\det(I_n))=0$. Vaidya~\cite{vaidya96} shows
\begin{equation}\label{eq:Vaidya-progress}
\rho^{(k)}\ge \frac{ck} 2.
\end{equation}

\subsubsection{Certifying volumetric cutting plane methods}
Our goal is now to show the following certified version of Vaidya's algorithm.
\begin{theorem}\label{thm:Vaidya}
Let $K$ be a  convex set given by a strong
separation oracle, $r>0$, and $\varepsilon\in (0,2 r)$. Then, in $O(n \log(nr/\varepsilon))$ calls to the strong separation oracle and $O(n^{\omega+1}\log(nr/\varepsilon))$ arithmetic operations, Vaidya's algorithm returns either a point $x\in K$ or  an $\varepsilon$-approximate Farkas certificate for $K\cap\bb{B}^n(r)$ comprising only oracle inequalities. The support of the certificate has cardinality  $O(n)$. 
\end{theorem}  

We show a general technique to extract a certified quadratic form, given a strictly feasible point. We will apply this lemma with $z$ satisfying \eqref{eq:Vaidya-proximity}, and with $\gamma_i=\sigma_i(z)$; these will satisfy the technical requirements.

\begin{lemma}
\label{lemma:sandwich}
Let $P=\{x\in\R^n\st Ax\le u\}$ be a polytope, where $A\in\R^{m\times n}$ and $u\in\R^m$. Let $z$ be a point in the interior of $P$, and $\gamma\in\Rpp^m$.
Define
\begin{equation}
\label{eq:weighted hessian}
R\coloneqq \sum_{i=1}^m \frac{\gamma_i a_ia_i^\top}{(u_i-a_i^\top z)^2},\qquad w=\sum_{i=1}^m \frac{\gamma_i a_i}{u_i-a_i^\top z}
\end{equation}

Let $\rho \coloneqq \|w\|_{R^{-1}}$ and $\varphi\coloneqq \max_{i\in[m]} \|a_i\|_{ R^{-1}}/(u_i-a_i^\top z)$. Assume that $\varphi\rho\le 1/2$. Let the vector $\ell\in\R^m$ and the quadratic form $q$  be defined by
\begin{equation}\label{eq:certified-quadratic-form}
\ell_i\coloneqq a_i^\top z-3\varphi^2 \|\gamma\|_1(u_i-a_i^\top z),\qquad q(x)=\sum_{i=1}^m \gamma_i \frac{(u_i-a_i^\top x)(a_i^\top x-\ell_i)}{(u_i-a_i^\top z)^2}.
\end{equation}

The following hold.
\begin{enumerate}[$(i)$]
\item For any $k\in [m]$, in $O(n^\omega)$ time we can compute coefficients $\ld^{(k)}\in\Rp^m$, $k\in[n]$,  such that $A^\top \ld^{(k)}=-a_k$, $u^\top\ld^{(k)}\le -\ell_k$, i.e. a dual certificate of validity of
$a_k^\top x\ge \ell_k$ for $P$.
\item $q$ is strictly concave, $\max_{x\in\R^n}q(x)\le (3\varphi \|\gamma\|_1)^2$, proving that $P\subseteq 3\varphi \|\gamma\|_1 E(R,p)$ for $p=\arg\max q(x)$.
\end{enumerate}
\begin{proof}
Define $\bar a_i =a_i/(u_i-a_i^\top z)$ for $i\in [m]$. With this notation, $R=\sum_{i=1}^m\gamma_i\bar a_i\bar a_i^\top$ and $w=\sum_{i=1}^m \gamma_i \bar a_i$.
 Observe that $P=\{x\in \R^n\st \bar a^\top _i(x-z)\le 1,\; i\in [m]\}$.
For part {\em (i)}, for any $k\in [m]$, let us define $\bar\ld^{(k)}$ by
\begin{equation}\label{eq:weighted-lower-bounds}\bar \ld^{(k)}_i \coloneqq \gamma_i\left(2\varphi^2(1-\bar a_i^\top R^{-1}w)-\bar a_i^\top R^{-1}\bar a_k\right),\qquad   i=1,\ldots,m.\end{equation}

Observe that $\bar\ld^{(k)}\ge 0$ because $|\bar a_i R^{-1}w|\le\|\bar a_i\|_{R^{-1}}\|w\|_{R^{-1}}\le\rho\varphi\le  1/2$ and $|\bar a_i R^{-1}\bar a_k|\le \|\bar a_i\|_{R^{-1}} \|\bar a_k\|_{R^{-1}}\le \varphi^2$, by definition of $\varphi$ and $\rho$. We also show that
\begin{equation}\label{eq:rho-phi}
\rho \le \varphi\|\gamma\|_1\, .
\end{equation}
This holds because 
\[
  \rho^2=\sum_{i=1}^m \gamma_i \bar a_i^\top R^{-1}w\le \sum_{i=1}^m \gamma_i \varphi \rho=\|\gamma\|_1 \varphi\rho\, ,
\]
by the definition of $w$. 
Next, observe that
\begin{equation}\label{eq:mu-a-k-1}\sum_{i=1}^m \bar\ld^{(k)}_i \bar a_i=2\varphi^2\left(\sum_{i=1}^m \gamma_i \bar a_i-\sum_{i=1}^m \gamma_i \bar a_i\bar a_i^\top R^{-1} w\right)-\sum_{i=1}^m \gamma_i \bar a_i\bar a_i^\top R^{-1}\bar a_k=-\bar a_k\end{equation}
and
\begin{equation}\label{eq:mu-a-k-2}
\begin{aligned}
\sum_{i=1}^m \bar\ld^{(k)}_i &=2\varphi^2(\|\gamma\|_1-\sum_{i=1}^m\gamma_i \bar a_i R^{-1}w)-\left(\sum_{i=1}^m\gamma_i \bar a_i\right) R^{-1}\bar a_k\\
&=2\varphi^2(\|\gamma\|_1-\rho^2)-w^\top R^{-1} \bar a_k\le 2\varphi^2\|\gamma\|_1+\varphi\rho\\
&\le 3\varphi^2\|\gamma\|_1\, ,
\end{aligned}
\end{equation}
where the last inequality follows from \eqref{eq:rho-phi}.

We define the vector $\ld^{(k)}\in\Rp^m$  by
\[\ld^{(k)}_i\coloneqq \frac{u_k-a_k^\top z}{u_i-a_i^\top z}\bar\ld^{(k)}_i,\] and show that it satisfies $A^\top \ld^{(k)}=-a_k$, $u^\top\ld^{(k)}\le -\ell_k$. Indeed
\[A^\top \mu^{(k)}=(u_k-a_k^\top z)\sum_{i=1}^m \bar\ld^{(k)}_i \bar a_i=-(u_k-a_k^\top z)\bar a_k=-a_k\, ,\]
using \eqref{eq:mu-a-k-1}, 
and
\[
\begin{aligned}u^\top\ld^{(k)}&=\sum_{i=1}^m (u_i-a_i^\top z)\ld^{(k)}_i +\sum_{i=1}^m a_i^\top z\ld^{(k)}_i=(u_k-a_k^\top z)\sum_{i=1}^m \bar\ld^{(k)}_i -a_k^\top z\\
&\le (u_k-a_k^\top z)3\varphi^2\|\gamma\|_1-a_k^\top z=-\ell_k\, ,
\end{aligned}
\]
using \eqref{eq:mu-a-k-2}.

\medskip

For part {\em (ii)}, note that $q$ is strictly concave because $P$ is a polytope, hence $\rk(A)=n$. Consider any $x\in \R^n$ and let $y=x-z$. We have
\begin{eqnarray*}
q(x)&=&q(y+z)=\sum_{i=1}^m \gamma_i (1-\bar a_i^\top y)(\bar a_i^\top y+3\varphi^2\|\gamma\|_1)\\
&=&
-y^\top \left(\sum_{i=1}^m \gamma_i \bar a_i \bar a_i^\top\right)y  +\left(1-3\varphi^2\|\gamma\|_1\right)\sum_{i=1}^m \gamma_i \bar a_i^\top y + 3\varphi^2\|\gamma\|_1\sum_{i=1}^m \gamma_i \\
&=&-y^\top R y+\left(1-3\varphi^2\|\gamma\|_1\right)w^\top y+3\varphi^2\|\gamma\|_1^2\, .
\end{eqnarray*}

Setting the gradient of the last expression to zero, i.e. $-2R y+\left(1-3\varphi^2\|\gamma\|_1\right)w=0$, we obtain that the expression is maximised by $\bar y=\frac 12(1-3\varphi^2\|\gamma\|_1)R^{-1}w$. It follows that the point $p=z+\bar y$ is the unique maximizer of $q(x)$, and
\begin{eqnarray*}
\max_{x\in\R^n}q(x)&=&  -\bar y^\top R \bar y+\left(1-3\varphi^2\|\gamma\|_1\right)w^\top \bar y+3\varphi^2\|\gamma\|_1^2\\
&=& 3\varphi^2\|\gamma\|_1^2+\frac 14 (1-3\varphi^2\|\gamma\|_1)^2\rho^2 \le (3\varphi \|\gamma\|_1)^2\, ,
\end{eqnarray*}
where the last inequality follows from  $|1-3\varphi^2\|\gamma\|_1|\rho\le \max\{\rho, 3\varphi^2\|\gamma\|_1\rho\}\le 3\varphi\|\gamma\|_1$, since $\rho\le \varphi\|\gamma\|_1$ by \eqref{eq:rho-phi} and 
$3\varphi^2\|\gamma\|_1\rho\le\frac 32 \varphi\|\gamma\|_1$ since $\varphi\rho\le 1/2$ by assumption.
\end{proof}
\end{lemma}

\begin{proof}[Proof of Theorem~\ref{thm:Vaidya}]
As previously noted, we may assume $r=1$. Observe that, since $\sum_{i=1}^m \sigma_i(z)=n$, at every iteration $m\le n/c$ holds, since otherwise $\min_{i\in[m]}\sigma_i(z)<c$, and an inequality is removed from the system. This shows that the number of inequalities in the system is $O(n)$. By \eqref{eq:Vaidya-progress}, after  $3(n/c)\log(1/c)+4n\log(20n/\varepsilon)= O(n\log(n/\varepsilon))$ iterations, we have \[\det(\nabla^2 f(z))=2^{F(z)}\ge 2^{F(\tilde{\omega})}\ge c^{-1.5 n}(20n/\varepsilon)^{2n}.\] Furthermore, continuing Vaidya's algorithm by removing constraints as long as  $\min_{i\in [m]}\sigma_i(z)<c$, which requires at most $m\le n/c$ further iterations, we can also guarantee that $\sigma_i(z)\ge c$ for all $i\in [m]$. In particular, it follows that $\mu(z)\ge c$ and  $\det (Q(z))\ge c^n \det(\nabla^2 f(z))\ge c^{-n/2}(20n/\varepsilon)^{2n}$.

Vaidya (Lemma 10 in~\cite{vaidya96}) shows that, if $F(z)-F(\tilde \omega)\le \delta\sqrt{\mu(\tilde\omega)}$, then $\nabla F(z)^\top Q(z)^{-1}\nabla F(z)\le (F(z)-F(\tilde\omega))/0.14$. Since we maintain a point $z$ satisfying~\eqref{eq:Vaidya-proximity}, and since $c^4\mu(\tilde \omega)< \delta\sqrt{\mu(\tilde\omega)}$ (since $c<\delta$ and $\mu(\tilde\omega)<1$), it follows that 
\begin{equation}\label{eq:norm-w}
\nabla F(z)^\top Q(z)^{-1}\nabla F(z)\le c^4\mu(\tilde\omega)/0.14.
\end{equation}

Vaidya (Claim 4 in~\cite{vaidya96}) also shows that \begin{equation}\label{eq:norm-a_i}
a_i^\top Q(z)^{-1}a_i/(u_i-a_i^\top z)^2\le 1/\sqrt{\mu(z)}\le c^{-1/2},\quad i\in[m].
\end{equation}

Let us now define $\gamma\in\Rpp^m$ by $\gamma_i\coloneqq \sigma_i(x)$, $i\in[m]$, $R\coloneqq Q(z)$, $w=\nabla F(z)$. Observe that $\gamma$, $z$, $R$, and $w$ satisfy \eqref{eq:weighted hessian}. By \eqref{eq:norm-w} and \eqref{eq:norm-a_i}, we have $\rho\coloneqq \|w\|_{R^{-1}}\le \sqrt{c^4\mu(\tilde\omega)/0.14}$ and $\varphi\coloneqq \max_{i\in[m]} \|a_i\|_{ R^{-1}}/(u_i-a_i^\top z)\le c^{-1/4}$. Note that $\rho\varphi\le 1/2$, hence $\gamma$ satisfies the assumptions of Lemma~\ref{lemma:sandwich}, which shows that the quadratic form $q(x)$ defined in \eqref{eq:certified-quadratic-form} can be certified for $Ax\le u$ in time  $O(n^\omega)$, and that the ellipsoid $E=\{x\in\R^n \st q(x)\ge 0\}$ satisfies $E=\alpha E(R,p)$, $p$ being the center of $E$ and $\alpha\le 3 \varphi\|\gamma\|_1\le 3n c^{-1/4}$ (because $\|\gamma\|_1=\sum_{i=1}^m \sigma_i(x)=n$). It now follows from Lemmas~\ref{lem:c-Q-bound} that there exists $k\in [m]$ such that $\|a_k\|_{R^{-1}}/\|a_k\|_2\le (\det(R)^{1/n}-1)^{-1/2}$, and from Lemma~\ref{lemma:ellipsoid width} we have that, for $v=a_k/\|a_k\|_2$,
\[\width_E(v)=2\alpha\|v\|_{R^{-1}}\le 6n c^{-1/4} (\det(R)^{1/n}-1)^{-1/2}\le  6n c^{-1/4} (c^{-1/2}(20n/\varepsilon)^{2}-1)^{-1/2}\le \varepsilon/3.\]
It follows from Theorem~\ref{thm:approximate-farkas} that we can compute an $\varepsilon$-approximate Farkas certificate for $K\cap\bb{B}^n(1)$ comprising only oracle inequalities.
\end{proof}

\paragraph{Improving the number of arithmetic operations}
The bottleneck in terms of the number of arithmetic operations in Vaidya's algorithm is the computation of the leverage scores $\sigma_i(x)$ in every step.
 Jiang et al.~\cite{Jiang2020} provide a sophisticated data structure to maintain the leverage scores, which guarantees an amortized cost of $O(n^2)$ arithmetic operations per iteration, as opposed to Vaidya's $O(n^\omega)$ operations per iteration. This implies that the $O(n^{\omega+1}\log(nr/\varepsilon))$ arithmetic operations required in the statement of Theorem~\ref{thm:Vaidya} can be reduced to $O(n^{3}\log(nr/\varepsilon))$. This justifies the running time stated in Theorem~\ref{thm:JLSW}.

\section{Rational polyhedra}\label{sec:rational}

In this section, we consider rational polyhedra in the bit complexity model as in
 \cite{gls}, and show some implications of our algorithms. Let $\size{\alpha}$ denote the binary encoding length of
a rational number $\alpha$, and $\size{D}$ denote the encoding length
of a matrix $D$, defined as follows: for an integer $n\neq 0$, $\size{n}\coloneqq \lceil \log_2 |n|+1\rceil$; for a rational number $\alpha=p/q$ given as the ratio of co-prime integers with $q>0$, $\size{\alpha}\coloneqq \size{p}+\size{q}$; and for a matrix $D\in \Q^{n\times n}$, $\size{D}$ is the sum of the encoding length of all entries.
Let us recall the definitions of facet and
vertex complexity; see \cite[Definition (6.2.2)]{gls}.

\begin{definition}
Let $P\subseteq \R^n$ be a polyhedron.
\begin{enumerate}[(i)]
\item We say that $P$ has
\emph{facet-complexity at most $\varphi$}, if $P$ can be defined by a
system of linear inequalities with rational coefficients such that each inequality has encoding
length at most $\varphi$. If $P=\R^n$, we require  $\varphi\ge n+1$. The triple $(P; n,\varphi)$ is
called a \emph{well-described polyhedron}.
\item We say that $P$  has \emph{vertex-complexity at most $\nu$}, if
there exist finite sets of vectors
  $V_P,D_P\in \Q^n$
  such that $P=\conv(V_P)+\cone(D_P)$, and all vectors in
  $V_P\cup D_P$ have encoding length at most at most $\nu$. If
  $P=\emptyset$, then we require $\nu\ge n$.
\end{enumerate}
\end{definition}
Let us also note the following bounds relating facet- and vertex-complexity.

\begin{lemma}[{\cite[(6.2.4)]{gls}}]  \label{lem:vertex-complexity}
Let $P\subseteq\R^n$ be a polyhedron.
\begin{enumerate}[(i)]
\item If $P$ has facet-complexity at most $\varphi$, then $P$ has
  vertex-complexity at most $4n^2\varphi$.
\item If $P$ has vertex-complexity at most $\nu$, then $P$ has
  facet-complexity at most $3n^2\nu$.
\end{enumerate}
\end{lemma}

\paragraph{Dual solutions in the oracle model}
The results in \cite{gls}
 appear to be the only known methods in the literature for
obtaining dual certificates of infeasibility and optimality in the oracle model.
In this context, it is important to clarify which inequalities
can be involved in the dual solution. Two different concepts are considered in  \cite{gls}:  
{\em optimal standard dual solutions} and {\em optimal dual solutions with oracle inequalities}.

The former concept means that the dual solution assigns non-zero multipliers only to
facet-defining inequalities for $P$, in some standard representation of $P$. Since there is typically no guarantee that the separation oracle returns facet-defining
inequalities, it is not conceivable to be able to derive such certificates directly from the
execution of the method, and indeed getting an optimal standard dual solution requires
repeated applications of the ellipsoid method and the use of polarity, see \cite[Theorem~6.5.14]{gls}.

The second concept, instead, requires that the nonzero entries of the dual solution correspond
to inequalities that have been output by the oracle during the application of the method.
Gr\"otschel, Lov\'asz, and Schrijver~\cite{gls} point out that this is not possible in general.
Assumption \eqref{eq:bounded-facet-assumption} on bounded encoding length of the oracle inequalities plays a key role in this context:
under this assumption, one can obtain a dual
certificate as follows (cf. Lemma 6.5.15 in \cite{gls}). We first run
the (primal) ellipsoid method, obtaining a set $\cal F$ of oracle inequalities.
Next, we ``tighten'' each inequality in $\cal F$, by another run of the
ellipsoid per inequality. Finally, we apply the ellipsoid method to
the dual LP, with the variable set corresponding to $\cal F$.
Note that  the dimension of this latter
problem will be very large: 
while it is still polynomially bounded, it depends on a higher power of $\varphi$, whereas the running time of the primal algorithm only depends linearly on $\varphi$.

It has to be noted that, while assumption \eqref{eq:bounded-facet-assumption} is natural and applies widely in combinatorial optimization, it is not without loss of generality. A notable exception to assumption \eqref{eq:bounded-facet-assumption}  is optimizing over the $\vartheta$-body \cite[Chapter 9]{gls}. 
In Section~\ref{sec:without-bounded}, we sketch how our method can be adapted to settings that do not satisfy assumption \eqref{eq:bounded-facet-assumption}.

\paragraph{Bounding $\delta$ by $\varphi$}
The next lemma shows bounds on our condition number $\delta$ in terms of the bit-complexity.
\begin{lemma}\label{lem:delta-phi}
If every row of a matrix $M\in\R^{m\times n}$ has bit-complexity at most $\varphi$, then $\delta_M\ge 1/2^{O(n^3\varphi)}$.
\end{lemma}
\begin{proof} 
Let $m_i^\top$ be the $i$-th row of $M$. Consider a set $\{m_i\st i\in J\}$ of linearly independent vectors and nonzero coefficients $\lambda\in \R^J$; let $v=\sum_{i\in J} \lambda_i m_i$.
 Let $B\in \Q^{n\times n}$ be a matrix whose first $|J|$ columns are the vectors $m_i$, $i\in J$, with arbitrary further $n-|J|$ unit vectors added such that $B$ is non-singular.  Then, for $x=B^{-1}v$, we have $x_i=\lambda_i$ for $i\in J$ and $x_j=0$ otherwise.
  We have $\size{B}=O(n\varphi)$, and consequently, $\size{B^{-1}}\le 4n^2\size{B}=O(n^3\varphi)$ (see \cite[Exercise 1.3.5(d)]{gls}). In particular, the norm of every row in $B^{-1}$ is bounded by $2^{O(n^3\varphi)}$, and thereby, $|x_i|\le 2^{O(n^3\varphi)}\|v\|$. We also have $\|m_i\|\le 2^{O(\varphi)}$; thus,
  $\max_{i\in J}|\lambda_i|\|m_i\|_2\le 2^{O(n^3\varphi)} \|v\|$, showing the bound on $\delta_M$.
\end{proof}

We can use this bound together with Theorem~\ref{thm:LP-main}, to give a running time bound for the case when a polyhedral separation oracle is available.
\begin{corollary}\label{cor:LP-bit-model}
Let  $P=\{x\in\R^n\st Ax\le b\}$, and assume we are given polyhedral separation oracles for both $P$ as well as the recession cone $\mathrm{rec}(P)=\{x\in\R^n\st Ax\le 0\}$. Further, assume for every inequality $a_i^\top x\le b_i$ in the system,
$(a_i^\top ,b_i)$ has bit-complexity at most $\varphi$. Then, there exists oracle-polynomial algorithms for linear feasibility and linear optimization, also returning dual certificates, using $O(n^5\varphi)$ oracle calls and $O(n^7\varphi)$ arithmetic operations.
\end{corollary}

\begin{remark}\label{rem:ellipsoid-rounding}\em In order to obtain a polynomial algorithm in the bit-complexity model using an implementation of \textsc{Approximate-Conic-Dual} as in Section~\ref{sec:oracles}, one needs to provide an implementation of the subroutine that maintains rational numbers of bounded encoding length. This can be done by suitably rounding the coefficients encountered by the algorithms described in Section~\ref{sec:dual certificates}, but we do not elaborate the details here as the main focus of this paper is on the real model of computation.
\end{remark}

\subsection{Strong separation oracles with bounded bit complexity}\label{sec:continued}
The algorithms in \cite{gls} can return dual certificates with oracle inequalities under assumption
\eqref{eq:bounded-facet-assumption}. Recall that this assumption requires a strong separation oracle returning vectors whose encoding size is polynomially bounded by the facet complexity $\varphi$. However, in Corollary~\ref{cor:LP-bit-model}, we use a seemingly stronger polyhedral separation oracle model, where we require the oracle to return inequalities $a_i^\top x\le b_i$ such that $(a_i,b_i)$ has encoding size polynomially bounded by $\varphi$. Furthermore, Corollary~\ref{cor:LP-bit-model} requires such separation oracle also for  the recession cone.

 In this section, we show that a strong separation oracle satisfying assumption
\eqref{eq:bounded-facet-assumption} can be  turned into  polyhedral separation oracles of bounded bit-complexity.
 Assume that for the well-described polyhedron $(P;n,\varphi)$ with vertex complexity at most $\nu$, and we have a strong separation oracle that for each $\bar x\in \Q^n$, returns a vector $a\in \Q^n$,  such that $\max\{a^\top x:\, x\in P\}<a^\top \bar x$ with $\size{a}=O(\varphi)$. This oracle returns the inequality
 $a^\top x< a^\top \bar x$, that is, the right hand side $a^\top \bar x$ depends on the point $x$ queried, and a priori there could be an infinite number of potential inequalities returned.

Regarding the recession cone, \cite[Lemma (6.4.8)]{gls} shows that the strong separation oracle of $P$ implies a strong separation oracle for $\rec(P)$, returning inequalities $a^\top x\le 0$, where $a^\top x\le a^\top \bar x$ is an inequality that can be returned by the separation oracle for $P$. Hence, we obtain a polyhedral separation oracle for $\mathrm{rec}(P)$, with the inequalities of encoding length $\varphi$.

Next, we show that a polyhedral separation oracle can be implemented for $P$ by rounding the inequalities $a^\top x\le a^\top \bar x$ to $a^\top x\le b$ such that $b$ has bit-complexity $O(\varphi+\nu)$. Thus, there exists a finite description $Ax\le b$ of $P$ such that every inequality has encoding length $O(\varphi+\nu)$. Together with the rounding procedure, we can map the strong separation oracle to a polyhedral separation oracle with respect to this system, and therefore, Corollary~\ref{cor:LP-bit-model} is applicable.

Consider now the well-described polyhedron $(P;n,\varphi)$ with vertex complexity at most $\nu$, and a valid inequality $a^\top x\le u$ such that $\size{\separ}\leq \varphi$.
Let $b$ be the largest rational number $b< u$
such that $\size{b}\leq {\varphi+\nu}$.
Such a value $b$  can be computed in polynomial time
in $\size{u}$ and $\varphi+\nu$ via continued fractions, as
in Lemma~\ref{lem:cont-frac-lower} below.
 We claim that
$\separ^\T x\leq b$ is a valid inequality for $P$.
This follows since the linear function $a^\top x$ is bounded on $P$, and thus takes its maximum value on an extreme solution that has encoding length $\le \nu$. Hence, the value $\max\{a^\top x\st x\in P\}$ has encoding length $\leq \varphi+\nu$.

\paragraph{Rounding using the continued fractions method}
The key tool to find such a $b$ is the {\em continued fractions method}, an efficient algorithm for the following existential result by
Legendre. We do not describe the method but summarize its properties in the next theorem, see  \cite[Section 6.1]{Schrijver}
\begin{theorem}\label{thm:cont-frac}
  For given real number $\alpha\ge 0$ and integer $N>0$, there exists at most one
  pair $(p,q)$ of nonnegative integers such that $\left|\alpha-\frac pq\right|\le \frac{1}{2N^2}$ and $q\le N$.
  There exists an algorithm that, in $O(\log N)$ arithmetic operations,
  either finds such a pair $(p,q)$, or concludes that no such pair exists. Further,
  if $\alpha$ is rational, then the space complexity of the algorithm is  poly$(\size{\alpha},\log N)$.
\end{theorem}

\begin{lemma}\label{lem:cont-frac-lower} For given rational number $\alpha\ge 0$ and integer $\sigma>0$, in $O(\sigma)$ iterations
the continued fraction method finds
the largest rational number $q$ such that $q\le\alpha$ and $\size{q}\leq \sigma$.
If $\alpha$ is rational, then the space complexity of the algorithm is  poly$(\size{\alpha},\sigma)$.
\end{lemma}
\begin{proof}
This follows as in the proof of \cite[Theorem (5.1.9)]{gls}. The continued fractions method constructs a sequence of iterates $\beta_k=g_k/h_k$ such that $h_k$ grows exponentially. Further, $\beta_{2k}\leq \alpha\leq \beta_{2k+1}$.
Thus, there are $O(\sigma)$ iterates of size at most $\sigma$; the largest such $\beta_{2k}$ provides the required answer.
\end{proof}

\begin{remark}\em
It is not necessary to apply the rounding procedure to every inequality returned by the oracle. We can postpone the roundings and perform them only to the inequalities that participate in the $\varepsilon$-approximate Farkas certificates.
\end{remark}

\subsection{Removing the assumption on the complexity of oracle inequalities}\label{sec:without-bounded}
Let
us now briefly sketch how one can (inefficiently) recover the
result in~\cite{gls} without the simplifying assumption~\eqref{eq:bounded-facet-assumption}
within the polyhedral separation oracle model. In this context, one assumes
that the bit-complexity of the facets of the polyhedron $P$ is at
most $\varphi$, but the complexity of the
output of the separation oracle grows polynomially as a function of the
bit-size of the query point (i.e., it is not bounded by a fixed function of
$\varphi$). To convert any such oracle to a polyhedral one with lower bounded
$\delta$, one may simply post-process each inequality outputted by the oracle
using Diophantine approximation. Using the iterated
Diophantine approximation method of Frank and Tardos~\cite{franktardos}, one can
convert any valid inequality $a^\top x \leq b$ for $P$ to a ``nearby'' valid
inequality $\tilde{a}^\top x \leq \tilde{b}$ for $P$, where
$\tilde{a},\tilde{b}$ have bit-size ${\rm poly}(n,\varphi)$. More precisely,
the closeness of $\tilde{a} x \leq \tilde{b}$ to $a x \leq b$ is formalized
by saying that for any low-complexity point $\bar{x} \in \Q^n$, i.e., with
bit-size ${\rm poly}(n,\varphi)$, satisfies ${\rm sign}(b-a^\top x) = {\rm
sign}(\tilde{b}-\tilde{a}^\top x)$. In particular, if we maintain that we
only query the oracle with inputs of bit-size at most ${\rm
poly}(n,\varphi)$, which is relatively straightforward to achieve in most
settings, the ``post-processed'' oracle behaves analogously to a polyhedral
separation oracle with lower-bounded $\delta$. In particular, $\log 1/\delta$
will be polynomial in $\varphi$ and $n$, using standard bit-complexity
arguments.

One could make the above reduction more efficient by only lazily
post-processing the ``important'' oracle inequalities, which would reduce to
number of Diophantine approximations to $O(n)$, similar to~\cite{gls}. We omit the details of such a reduction.

\section{Relations to the circuit imbalance measure}\label{sec:circuits}

\subsection{Explicit LPs and the circuit imbalance measure}
Let us now discuss the implications of our results on explicitly given LP, and compare the running time achieved by our algorithm with the currently fastest algorithms for this setting. We consider linear programs in the standard equality form
\begin{equation}\label{eq:equality-form}
\min c^\top x\, ,\quad  Ax=b\, ,\,  x\ge 0\, ,
\end{equation}
where $A\in \R^{m\times n}$, $b\in\R^m$, $c\in\R^n$.
The dual can be written as
\begin{equation}\label{eq:equality-dual}
\max b^\top y\, ,\, \quad A^\top y\le c\, .
\end{equation}
One can obtain an $O(mn)$ time polyhedral separation oracle for this problem by computing the vector $A^\top y$.

Using the JLSW algorithm \cite{Jiang2020} to implement the approximate oracle for \eqref{eq:equality-dual}, Theorem~\ref{thm:LP-main} yields a complexity bound
$O\left(nm^3\log(n/\delta_{A^\top})+m^5\log\log(n/\delta_{A^\T})\right)$ for the feasibility of \eqref{eq:equality-form} (that is, the dual of \eqref{eq:equality-dual}), and
$O\left(nm^3\log(n/\delta_{N})+m^5\log\log(n/\delta_{N})\right)$ for optimization, where 
\begin{equation}\label{eq:N-matrix}
N=\begin{pmatrix}\mathbf{0}& 1\\
-A^\top & c
\end{pmatrix}\, .
\end{equation}
We compare this with recent work on explicitly given linear programs \cite{DadushHNV20,DadushNV20} that parametrize the running time by the circuit imbalance measure.
For a linear space $W\subseteq\R^n$, the set of \emph{elementary vectors} ${\cal E}(W)$ is the set of support minimal nonzero vectors in $W$; the supports of elementary vectors are the circuits in the associated linear matroid. The \emph{circuit imbalance measure} $\kappa_W$ is 1 if $W\in \{\{\bO\},\R^n\}$, and otherwise the maximum ratio  $|g_j/g_i|$ over all   $g\in {\cal E}(W)$ and all $i,j\in \supp(g)$. For a matrix $A\in\R^{m\times n}$, we let $\kappa_A$ denote $\kappa_W$ for $W=\ker(A)$.
In particular, $\kappa_A=1$ for totally unimodular matrices.

Dadush et al.~\cite{DadushNV20}, strengthening Tardos's result \cite{Tardos86} on combinatorial linear programs, gave an algorithm with running time poly$(n,m,\log(\kappa_A+n))$ for solving linear programs of the form \eqref{eq:equality-form}. We refer the reader to the survey \cite{ENV22} on the various uses of circuit imbalance measures in linear programming.

The condition number $\kappa_A$ is within a factor $1/n$ from the Dikin--Stewart--Todd condition number $\bar\chi_A$ used in \cite{Vavasis1996}, see \cite{DadushHNV20,DadushNV20}. Hence,  $\log(\kappa_A+n)=\Theta(\log(\bar\chi_A+n))$.

The algorithm in \cite{DadushNV20} is of a black-box nature: for linear optimization, it requires $O(nm)$ calls to an approximate linear programming solver with accuracy $\varepsilon=1/(n\kappa_A)^{O(1)}$.
Combined with the solver of van den Brand \cite{vdb20}, for feasible and bounded instances we get a deterministic algorithm with running time  $O\big(m n^{\omega+1+o(1)}$ $\log(\kappa_A+n)\big)$, and combined with the solver of van den Brand et al.~\cite{vdB2020-tall-dense}, we get a randomized algorithm with running time $O\big(m^2n(n +m^2 ) \log^{O(1)}(n)\log(\kappa_A+n)\big)$.

 For the linear feasibility problem $Ax=b$, $x\ge 0$, $O(m)$ calls suffice. For this problem, we get a deterministic running time $O\big(m n^{\omega+o(1)}\log(\kappa_A+n)+(m^2n+n^{\omega+o(1)}) \log\log(\kappa_A+n)\big)$, and randomized running time $O\big(m^2(n +m^2 ) \log^{O(1)}(n)\log(\kappa_A+n)+(m^2n+n^{\omega+o(1)}) \log\log(\kappa_A+n)\big)$, respectively.

\medskip

The condition numbers $1/\delta$ and $\kappa$ are reconciled in
Section~\ref{sec:circuits-delta}.
In particular, we show that for a matrix of the form $A=(I_m|A')$, $\log(n/\delta_{A^\top})=\Theta(\log(\kappa_A+n))$. Using the conic validity algorithm in Theorem~\ref{thm:conic-main}, and the JLSW algorithm \cite{Jiang2020} to implement the approximate oracle, we obtain the following. 

\begin{theorem}\label{thm:LP-runtime} Let $A\in\R^{m\times n}$, $b\in \R^m$, and consider the linear feasibility problem $Ax=b$, $x\ge 0$; assume $\rk(A)=m$. There exists an $O(nm^3\log(n+\kappa_A)+m^5\log\log(n+\kappa_A))$ time algorithm that either finds a feasible solution, or a Farkas certificate $A^\top y\ge 0$, $b^\top y<0$.
\end{theorem}

At a high level, the feasibility algorithm in \cite{DadushNV20} and our conic validity algorithm both use approximate solutions to $Ax\approx b$, $x\ge 0$ to reduce the problem size, and project out variables with high $x_i$ values. The main difference is that \cite{DadushNV20} requires a stronger approximate oracle that enables a more efficient `pullback' of a Farkas certificate in case of infeasibility. Our algorithm has an additional term $O(m^5\log\log(n+\kappa_A))$ compared to $O((nm+n^{\omega+o(1)})\log\log(\kappa_A+n))$ in \cite{DadushNV20}. %

We note that our method  cannot reproduce the main result of \cite{DadushNV20} of a
poly$(n,m,\log(\kappa_A+n))$ algorithm for linear optimization: our running time depends on $\delta_{N}$ for $N$ as in \eqref{eq:N-matrix}, in particular, it also depends on $c$.
On the other hand, \cite{DadushNV20} heavily uses that the system is explicitly given, while our method extends to the oracle setting.
As noted in Proposition~\ref{prop:need-b}, in the oracle model there cannot exist an algorithm with running time dependence only on a condition number of the matrix.

\subsection[Comparing kappa and delta]{Comparing $\kappa$ and $\delta$}\label{sec:circuits-delta}
We now derive bounds between the condition numbers $\kappa_A$ and $\delta_{A^\top}$, and prove Theorem~\ref{thm:LP-runtime}.
We use the following result on self-duality of $\kappa_W$.
\begin{lemma}[{\cite{DadushHNV20}}]\label{lem:kappa-dual}
For every linear subspace $W\subseteq \R^n$ and the orthogonal complement $W^\top$, $\kappa_W=\kappa_{W^\top}$. 
\end{lemma}
We start with the following general bound. Recall that the minimum singular value of $A^\top$ for a full row rank matrix $A$ equals the minimum nonzero singular value of $A$.
\begin{lemma}\label{lem:delta-kappa}
Let $A\in\R^{m\times n}$ be a matrix with full row rank and $m<n$ with $\|a_i\|=1$ for all columns $i\in[n]$.
Let $\sigma_{\min}(A^\top)$ be the minimum singular value of $A^\top$. Then,
\[
 \frac{\sigma_{\min}(A^\top)}{\sqrt{n}\delta_{A^\top}}\le \kappa_A\le \frac{1}{\delta_{A^\top}} \, .
\]
\end{lemma}
\begin{proof}
Let us start with the upper bound on $\kappa_A$. Let  $g\in{\cal E}(\ker(A))$ be an elementary vector.
Select an arbitrary $i\in \supp(g)$, and let $J=\supp(g)\setminus \{i\}$. Then, the columns $\{g_j\st j\in J\}$ are linearly independent, and $-g_i a_i=\sum_{j\in J} g_j a_j$. Thus,
\[
|g_i|\cdot\|a_i\|=\left\|\sum_{j\in J} g_j a_j\right\|\ge \delta_{A^\top} \max_{j\in J} |g_j|\cdot\|a_j\|\, ,
\]
and using that all columns have unit norm, we get $|g_j/g_i|\le 1/\delta_{A^\top}$ for all $j\in J$. This shows that $\kappa_{A}\le 1/\delta_{A^\top}$.

We now turn to the lower bound on $\kappa_A$.
According to Lemma~\ref{lem:inverse-norm}, there exists an index set $B\subseteq [n]$, $|B|=m$, and an index $k\in B$, such that  $A_B$ is non-singular, and
$1/\delta_{A^\top}$ equals the norm of the $k$-th row of $A_B^{-1}$. %
This  row of $A_B^{-1}$ is the unique solution $z\in\R^m$ to the system $(A_B)^\top z={e}^m_k$, where ${e}^m_k$ is the $k$-th unit vector in $\R^m$.

Let us now consider the vector $y=A^\top z\in \R^n$; this is a vector in $\im(A^\top)$, and we claim that it is an elementary vector. Indeed, suppose a nonzero vector $y'\in \im(A^\top)$ exists with strictly smaller support. We have $y_i=0$ for all $i\in B\setminus \{k\}$. If $y'_k=0$, then the nonsingularity of $A_B$ implies that $y'=0$. If $y'_k\neq 0$, then we can normalize to $y'_k=1$; but now $y'_B=y_B$, and again by the nonsingularity of $A_B$, we must have $y'=y$.

Let us use duality of $\kappa_W$ (Lemma~\ref{lem:kappa-dual}) for $W=\ker(A)$ and $W^\perp=\im(A^\top)$. Since $y_k=1$, we obtain $\|y\|_\infty\le \kappa_{W^\perp}=\kappa_W$. We thus get
\[
\frac{\sigma_{\min}(A^\top)}{\delta_{A^\top}}=\sigma_{\min}(A^\top)\|z\|\le \|A^\top z\|=\|y\|\le \sqrt{n}\|y\|_\infty\le \sqrt{n}\kappa_{W}\, .
\]
\end{proof}

\begin{corollary}\label{cor:kappa-rescale}
Let $A\in\R^{m\times n}$ be of the form $A=(I_m|A')$. Then,
\[
\frac{1}{\delta_{A^\top}}\le \sqrt{nm}\kappa_A^3\, .
\]
\end{corollary}
\begin{proof}
For any column $a_i$ of $A$, $i\in [m+1,n]$, we have a circuit $g\in\ker(A)$ defined as $g_j=-a_{ij}$ for $j\in [m]$, $g_i=1$, and $g_j=0$ otherwise. Hence, all nonzero entries of $A$ are between $1/\kappa_A$ and $\kappa_A$, and therefore all column norms are between $1/\kappa_A$ and $\sqrt{m}\kappa_A$.

Let $\hat A$ be the matrix arising by normalizing all columns of $A$ so that all columns have norm 1, and let us apply Lemma~\ref{lem:delta-kappa} to $A$. Note that, since $A$ contains an identity matrix, $\sigma_{\min}(\hat A^\top)\ge 1$. Hence,
${1}/{\delta_{\hat A^\top}}\le \sqrt{n}\kappa_{\hat A}$.
Note that $\delta_{\hat A^\top}=\delta_{A^\top}$ because the definition of $\delta_M$ is invariant under row rescalings. Further, $\kappa_{\hat A}\le \sqrt{m}\kappa_A^3$ because
Renormalization may increase the ratio between two entries of an elementary vector by at most $\sqrt{m}\kappa_A^2$. Thus, the claim follows.
\end{proof}

We are ready to prove Theorem~\ref{thm:LP-runtime}
\begin{proof}[Proof of Theorem~\ref{thm:LP-runtime}]
First, we can use Gaussian elimination to  bring $A$ to the basic form $A_B^{-1}A$ in $O(nm^2)$ time. To simplify the notation, we henceforth replace $A$ by $A_B^{-1}A$ and $b$ by $A_B^{-1}b$.
Consider the conic validity problem for the vector $b$ and the cone $K=\{y\in\R^m:\, A^\top y\ge 0\}$. The output is either a primal feasible solution or a Farkas certificate to our problem.

We obtain the claimed running time from  Theorem~\ref{thm:conic-main} with the oracle as in Theorem~\ref{thm:approx-conic-oracle} implementing the JLSW algorithm \cite{Jiang2020}, and noting that every oracle call takes $nm$ time by checking each inequality $a_i^\top y$. Finally, note that by Corollary~\ref{cor:kappa-rescale}, $\log(n/\delta_{A})=O(\log(n+\kappa_A))$.
\end{proof}

\begin{remark}\label{rmk:remark-inapproximable}\em Lemma~\ref{lem:delta-kappa} also shows that even for an explicitly given matrix $A$, computing or even approximating $\delta_A$ is hard. This is because it is NP-hard to approximate $\kappa_A$ within a  $2^{\mathrm{poly}(m)}$-factor. See \cite{DadushHNV20}; this follows by the bounds between $\bar\chi_A$ and $\kappa_A$, and Tun{\c{c}}el's inapproximability result on  $\bar\chi_A$ \cite{Tuncel1999}.
\end{remark}

\subsection[Optimizing the measure delta]{Optimizing the measure $\delta$}\label{sec:circuit-rescale}
Consider a full-row rank matrix $A\in\R^{m\times n}$. As we have already observed, the measures $\kappa_A$ and $\delta_{A^\top}$ are invariant under different rescalings: as noted above, for any positive diagonal matrix $D\in\R^{n\times n}$, $\delta_{(AD)^\top}=\delta_{A^\top}$, but $\kappa_{AD}$ might be different from $\kappa_A$. On the other hand, for any nonsingular $T\in \R^{m\times m}$, the opposite holds: $\kappa_{TA}=\kappa_A$, but $\delta_{(TA)^\top}$ and $\delta_{A^\top}$ may be different.
One can naturally ask for the following quantities.
\begin{eqnarray*}
\kappa^*_A&\coloneqq &\inf\left\{\kappa_{AD}\st D\in\R^{n\times n} \mbox{ positive diagonal}\right\}\\
\delta^*_{A^\top}&\coloneqq &\sup\left\{\delta_{(TA)^\top}\st T\in\R^{m\times m} \mbox{ nonsingular}\right\}\\
\end{eqnarray*}
The quantity $\kappa^*_A$ was studied in \cite{DadushHNV20}; they gave a min-max characterization, as well as the following algorithmic result:
\begin{theorem}[{\cite{DadushHNV20}}]\label{thm:kappa-star}
Given a matrix $A\in\R^{m\times n}$, in $O(n^2m^2+n^3)$ time one can find a positive diagonal matrix $D$ such that $\kappa_{AD}\le (\kappa^*_A)^3$.
\end{theorem}
As a consequence, all results with running time dependence on $\log(\kappa_A+n)$ can be strengthened to dependence on $\log(\kappa^*_A+n)$.
Note that the approximability result is in surprising contrast with  Tun{\c{c}}el's above mentioned result~\cite{Tuncel1999}. 

We now present analogous bounds on $\delta^*_{A^\top}$.
\begin{lemma}\label{lemma:delta-kappa}
Given a matrix $A\in\R^{m\times n}$, $\delta^*_{A^\top}\le 1/\kappa^*_A$.
\end{lemma}
\begin{proof}
Let $D$ be the diagonal matrix with $i$th diagonal entry equal to $1/\|a_i\|$, so that all columns of $AD$ have unit norm. For any nonsingular $T\in\R^{m\times m}$, Lemma~\ref{lem:delta-kappa} implies
\[
\delta_{{(TA)}^\top} =\delta_{{(TAD)}^\top} \le \frac{1}{\kappa_{TAD}}=
\frac{1}{\kappa_{AD}}\le  \frac{1}{\kappa^*_A}\, .
\]
\end{proof}

Combined with Corollary~\ref{cor:kappa-rescale}, we obtain the following result on $\delta^*_{A^\top}$.
\begin{corollary}\label{cor:delta-opt}
For a matrix $A\in \R^{m\times n}$, in $O(n^2 m^2+m^3)$ time we can find a nonsingular $T\in\R^{m\times m}$ such that
\[
\frac{(\delta^*_{A^\top})^9}{\sqrt{nm}}\le \delta_{{(TA)}^\top}\, .
\]
\end{corollary}
\begin{proof}
Let us compute the near optimal rescaling $D$ such that $\kappa_{AD}\le (\kappa^*_A)^3$ as in Theorem~\ref{thm:kappa-star}. Compute any non-singular $m\times m$ matrix $B$ of $AD$, and return $T=B^{-1}$. We show that $T$ satisfies the statement. Indeed,  $TAD$ contains an $m\times m$ identity matrix, therefore, by Corollary~\ref{cor:kappa-rescale}, we get
\[
\delta_{{(TA)}^\top}=\delta_{{(TAD)}^\top}\ge \frac{1}{\sqrt{nm}(\kappa_{TAD})^3}= \frac{1}{\sqrt{nm}(\kappa_{AD})^3}\ge \frac{1}{\sqrt{nm}(\kappa^*_{A})^9}\ge \frac{(\delta^*_{A^\top})^9}{\sqrt{nm}}\, ,
\]
where the last inequality follows from Lemma~\ref{lemma:delta-kappa}.
\end{proof}

Note that we are only able to use the above renormalization for an explicitly given matrix $A$, but not in the oracle model.

\subsection*{Acknowledgement}
The authors would like to thank Bento Natura for discussions on minimum ratio problems and circuit imbalances.

\appendix

\section{Missing proofs}\label{sec:impossibility}

\begin{proposition}\label{prop:A-M}
For $A\in\R^{m\times n}$, and for the matrix $M$ defined in \eqref{eq:M-matrix}, it holds that $ \delta_A\ge \delta_{M}$.
\end{proposition}
\begin{proof}
Indeed, given $I\subseteq [m]$ such that  $\{a_i\st i\in I\}$ are linearly independent, it follows that $\{(a_i\mid b_i)\st i\in I\}\cup \{(\mathbf{0}\mid 1)\}$ are linearly independent, hence, for every $\lambda\in\R^I$,
\begin{eqnarray*}
 \left\|\sum_{i\in I}\lambda_i a_i\right\|&=&\left\|\sum_{i\in I}\lambda_i (a_i\mid b_i)-\left(\sum_{i\in I}\lambda_i b_i\right)(\mathbf{0}\mid 1)\right\|\\
 &\ge & \delta_{M} \max\left\{\max_{i\in I}|\lambda_i| \left\|(a_i\mid b_i)\right\|,\left|\sum_{i\in I}\lambda_i b_i\right|\right\}\ge \delta_{M} \max_{i\in I}|\lambda_i| \|a_i\|.
\end{eqnarray*}
\end{proof}

\nostrongly*
\begin{proof}
For $m\in\N$, let $\mathcal{P}_m$ denote the set of polytopes $P\subseteq \R^{2}$ of the form
\begin{equation}\label{eq:P-M}
P=\left\{(x_1,x_2)\in \R^2:\, \frac{x_1}{p_i}+\frac{x_2}{q_i}\le 1, i\in [m], (x_1,x_2)\ge 0\right\}\, ,
\end{equation}
for some numbers $p_1>p_2>\ldots>p_m>0$, $0<q_1<q_2<\ldots<q_m$.
For any such choice of parameters, these define a full-dimensional polytope with $m+2$ facets. Let $\mathcal{P}_m$ denote the class of all polytopes of this form.

Consider the problem $\max~x_1$, $x\in P$ for such a polytope $P\in \mathcal{P}_m$. Clearly, the optimal solution is $(p_m,0)$.
We claim that an adversary can answer oracle queries such that, for any prescribed $m^*\in \mathbb{N}$, any algorithm will require at least $m^*$ queries to compute an optimal solution. This proves the claim, since $m^*$ can be chosen independently from $n=2$.

The adversary strategy is as follows. At each oracle query  $t=1,2,\ldots$, the adversary maintains a parameter $m\le t$, and  two polytopes
$P_m\in \mathcal{P}_m$ and $Q_m\subseteq P_m$. These are initialized as $m=0$, $P_0=\Rp^2$, $Q_0=\{0\}$. For $m\ge 1$, the polyhedron $P_m$ is given in the  form \eqref{eq:P-M} with $p_1>p_2>\ldots>p_m>0$ and $0<q_1<q_2<\ldots<q_m$, and $Q_m=P_m\cap \{(x_1,x_2)\in \R^2:\, x_1\le z_1\}$, where $z=(z_1,z_2)\in \R^2$ is defined as follows.
For $m=1$, we let $z=(0,q_1)$, and for $m>1$, we let
$z=(z_1,z_2)$ be the point corresponding to the intersection of the last two separators $x_1/p_{m-1}+x_2/q_{m-1}=1$ and $x_1/p_{m}+x_2/q_{m}=1$. 

Let $x\in\R^2$ be $t$-th oracle query of the algorithm. If $x\in Q_m$, the oracle responds that the point is feasible. If $x\notin P_m$, then the oracle returns a facet defining inequalities of $P_m$ that separates $x$. In both these cases, the value of $m$ and the polyhedra $P_m$ and $Q_m$ remain the same.
 Finally, if $x\in P_m\setminus Q_m$, the oracle returns a new separating inequality $x_1/p_{m+1}+x_2/q_{m+1}\le 1$ with $0<p_{m+1}<p_{m}$ and $q_{m+1}>q_{m}$; it is easy to see that such an inequality can always be added. 
We obtain  $P_{m+1}$ from $P_{m}$ by adding this new inequality,  update $Q_{m+1}$, and increase $m$ by one. Once we reach $m=m^*$, we set $Q_m=P_m$.

Note that this oracle strategy maintains $P_{m+1}\subset P_m$, $Q_{m+1}\supset Q_m$, and all facets of $P_m$ are also facets of $P_{m+1}$. Thus, all previous separators returned by the oracle are valid to $P_{m+1}$. Moreover, $m^*=m\le t$, where $t$ is the number of oracle queries made.
\end{proof}

\needb*
\begin{proof}
The proof is similar to the previous one, but even simpler, in 1-dimension.
Consider $b\in \R^m$ with values $b_1>b_2>\ldots>b_m\in \R$. Let $A\in \R^{m\times 1}$ be a matrix with all entries 1. Let $P=\{x_1\in \R\st Ax_1\le b\}$; thus, $P=(-\infty,b_m]$.
Consider the problem of $\max~x_1$ for $x_1\in P$; the optimal solution is $x_1^*=b_m$. However, $m$ oracle queries are necessary: Similarly to the previous proof, the adversary answering the oracle queries can always maintain $P$ in this form such that the first $k$ oracle queries are  $x_1\le b_i$, $i\in [k]$.

Note that all rows of the matrix $A$ are identical, hence, any condition number $\theta_A$ unchanged by duplicating copies must be the same as for the $1\times 1$ matrix $A'=(1)$. Therefore, $f(n,\theta_A)$ is the same for all instances in this form, showing that no algorithm may terminate in $f(n,\theta_A)$ queries.
\end{proof}

\farkassphere*

\begin{proof}%
For the ``if'' direction in the first part of the statement, assume that $\lambda\ge 0$ satisfying \eqref{eq:farkas-sphere} exists and suppose by contradiction that there exists  $x\in\bb{B}^n(r)$ satisfying $Ax\le u$. Then $-\lambda^\top u> r\|A^\top \lambda\| \ge \|x\|\cdot \|A^\top \lambda\| \ge -\lambda^\top A x\ge  -\lambda^\top u $, a contradiction. For the ``only if'' part, suppose that $Ax\le b$ has no solution in $\bb{B}^n(r)$. If $Ax\le u$ is infeasible, then by Farkas lemma there exists $\lambda\ge 0$ such that $A^\top \lambda=0$ and $u^\top \lambda<0$, hence $r\|A^\top \lambda\|<-u^\top \lambda$ as required. If $Ax\le u$ is feasible, then  $r^2/2<\min\{\|x\|^2/2\st Ax\le u\}$. The latter convex quadratic problem has a primal optimum (because, for a given feasible solution $\bar x$, we can limit ourselves to the compact set $\{x\st Ax\le u,\,\|x\|\le \|\bar x\|\}$). Furthermore, since all constraints are linear, by Slater's theorem strong duality holds and a dual optimum exists. Consider optimal primal and dual solutions $x $ and $\lambda $. They must satisfy the KKT conditions, hence $x +A^\top \lambda =0$ (gradient of the Lagrangian equal to zero) and $\lambda^\top A x =\lambda^\top u$ (complementary slackness). We have $r\|A^\top \lambda\|<\|x \|\cdot \|A^\top \lambda\|=\|A^\top \lambda\|^2=-\lambda^\top A x =-\lambda^\top u$.

Consider now $v\in\R^n$ and $\nu\in\R$. If there exists $\lambda\in\Rp^n$ satisfying \eqref{eq:optimality}, then for all $x$ such that $Ax\leq u$, $\|x\|\le r$ we have 
\[
\begin{aligned}
-v^\top x&\le -v^\top x-\lambda^\top Ax+ \lambda^\top u\\
&\le \|x\|\cdot \|A^\top \lambda+v\|+\lambda^\top u\\
&\le r\|A^\top \lambda+v\|+\lambda^\top u\le-\nu\, ,
\end{aligned}
\]
 hence $v^\top x\ge \nu$ is satisfied. Conversely, assume that  $v^\top x\ge \nu$  is valid and $Ax\le u$, $\|x\|<r$ has a solution. It follows that $\nu\le\min\{v^\top x\st Ax\le u,\, \|x\|^2/2\le r^2/2\}$. The latter convex problem has an optimum because the feasible region is compact. Furthermore, strong duality holds and a dual optimum exists, since Slater's conditions are satisfied. Let $x$, $(\lambda|\lambda_0)\ge 0$ be primal and dual optimal solutions. They must satisfy the KKT conditions, hence $v +A^\top \lambda+\lambda_0 x =0$ (gradient of the Lagrangian equal to zero) and $\lambda^\top A x =\lambda^\top u$, $\lambda_0(\|x\|-r)=0$ (complementary slackness). If $\lambda_0=0$, then $A^\top \lambda+v=0$, hence $\|A^\top \lambda+v\|+\nu=\nu\le v^\top x=-\lambda^\top Ax=-\lambda^\top u$.  If $\lambda_0>0$, by complementary slackness $\|x\|=r$. Furthermore, $x=-(A^\top \lambda+v)/\lambda_0$. Hence $r\|A^\top \lambda+v\|=\lambda_0 \|x\|^2=-\lambda^\top Ax-v^\top x\le -\lambda^\top u-\nu$. 
\end{proof}

\bibliographystyle{abbrv}
\bibliography{rescaled}

    \end{document}